\documentclass[10pt,a4paper]{article}
\usepackage{amsmath}
\usepackage{times}
\usepackage{amsfonts}
\usepackage{amssymb}
\usepackage{latexsym}

\usepackage[dvips]{epsfig}
\graphicspath{{./Fig_eps/}{./Eps/}}
\DeclareGraphicsExtensions{.ps,.eps}
\usepackage{color}
\usepackage{fullpage}

\newcommand\bemph[1]{{\bf\em #1}}
 
\newcommand{\remove}[1]{×}

\numberwithin{equation}{section}

\newcommand{\Birk}{\mathcal D} 
\newcommand{\Zkc}{Z_{k-\mathrm{col}}}
\newcommand{\Zkb}{Z_{k,\mathrm{bal}}}

\newcommand{\Zrb}{Z_{\rho,\mathrm{bal}}}
\newcommand{\Zrg}{Z_{\rho,\mathrm{good}}}
\newcommand{\Zkg}{Z_{k,\mathrm{good}}}

\newcommand{\Dg}{\mathcal D_{\mathrm{good}}}
\newcommand{\Dgs}{\mathcal D_{s,\mathrm{good}}}
\newcommand{\Rg}{\mathcal R_{\mathrm{good}}}
\newcommand{\Rgs}{\mathcal R_{s,\mathrm{good}}}

\newcommand\gnp{G(n,p)}
\newcommand\gnm{G_{\mathrm{ER}}(n,m)}
\newcommand\gnd{G(n,d)}
\newcommand\Gnd{\mathcal G(n,d)}

\newcommand{\dk}{d_{k-\mathrm{col}}}

\newcommand\MPCPS{Mathematical Proceedings of the Cambridge Philosophical Society}

\newcommand\COMB{Combinatorica}

\def\vec#1{\mathchoice{\mbox{\boldmath$\displaystyle#1$}}
{\mbox{\boldmath$\textstyle#1$}}
{\mbox{\boldmath$\scriptstyle#1$}}
{\mbox{\boldmath$\scriptscriptstyle#1$}}}

\DeclareMathOperator{\pr}{P}

\newcommand{\qed}{\hfill$\Box$\smallskip}
\newenvironment{proof}{\emph{Proof.}}{}

\newtheorem{definition}{Definition}[section]

\newtheorem{remark}[definition]{Remark}
\newtheorem{theorem}[definition]{Theorem}
\newtheorem{lemma}[definition]{Lemma}
\newtheorem{proposition}[definition]{Proposition}
\newtheorem{corollary}[definition]{Corollary}

\newtheorem{fact}[definition]{Fact}

\newcommand\id{\mathrm{id}}

\newcommand\cA{\mathcal{A}}
\newcommand\cB{\mathcal{B}}
\newcommand\cC{\mathcal{C}}

\newcommand\cF{\mathcal{F}}
\newcommand\cG{\mathcal{G}}
\newcommand\cE{\mathcal{E}}

\newcommand\cN{\mathcal{N}}

\newcommand\cH{\mathcal{H}}
\newcommand\cS{\mathcal{S}}

\newcommand\cI{\mathcal{I}}

\newcommand\cL{\mathcal{L}}
\newcommand\cM{\mathcal{M}}

\newcommand\cX{\mathcal{X}}

\newcommand\cV{\mathcal{V}}

\newcommand\cZ{\mathcal{Z}}

\def\cR{{\mathcal R}}
\def\cC{{\mathcal C}}
\def\cE{{\cal E}}

\newcommand\eul{\mathrm{e}}
\newcommand\eps{\varepsilon}
\newcommand\ZZ{\mathbf{Z}}
\newcommand\ZZpos{\ZZ_{\geq0}} 

\newcommand\Erw{\mathrm{E}}

\newcommand{\vecone}{\vec{1}}

\newcommand{\Po}{{\rm Po}}
\newcommand{\Bin}{{\rm Bin}}

\newcommand{\bink}[2] {{{#1}\choose {#2}}}

\newcommand\ra{\rightarrow}

\newcommand\KL[2]{D_{\mathrm{KL}}\bc{{{#1},{#2}}}}

\newcommand\bc[1]{\left({#1}\right)}
\newcommand\cbc[1]{\left\{{#1}\right\}}
\newcommand\bcfr[2]{\bc{\frac{#1}{#2}}}

\newcommand\brk[1]{\left\lbrack{#1}\right\rbrack}

\newcommand\norm[1]{\left\|{#1}\right\|}
\newcommand\abs[1]{\left|{#1}\right|}

\newcommand\RR{\mathbf{R}}

\newcommand{\whp}{w.h.p.}

\newcommand{\stacksign}[2]{{\stackrel{\mbox{\scriptsize #1}}{#2}}}
\newcommand{\tensor}{\otimes}

\newcommand{\Erdos}{Erd\H{o}s}
\newcommand{\Renyi}{R\'enyi}

\newcommand{\Bollobas}{Bollob\'as}

\newcommand{\Luczak}{\L uczak}

\newcommand\Lem{Lemma}
\newcommand\Def{Definition}
\newcommand\Prop{Proposition}
\newcommand\Thm{Theorem}
\newcommand\Cor{Corollary}
\newcommand\Sec{Section}
\newcommand\Chap{Chapter}
\newcommand\Rem{Remark}

\begin{document}

\title{\bf On the chromatic number of random regular graphs\thanks{The research leading to these results has received funding from the European Research Council under the European Union's Seventh Framework
			Programme (FP/2007-2013) / ERC Grant Agreement n.\ 278857--PTCC.}}

\author{
Amin Coja-Oghlan
$\qquad$
Charilaos Efthymiou
$\qquad$
Samuel Hetterich\\[2mm]
\normalsize Goethe University, Mathematics Institute,\\\normalsize10 Robert Mayer St, Frankfurt 60325, Germany\\
\tt\normalsize$\{$acoghlan,efthymiou,hetterich$\}$@math.uni-frankfurt.de
}
\date{\today}

\maketitle

\begin{abstract}
\noindent
Let $\gnd$ be the random $d$-regular graph on $n$ vertices.
For any integer $k$ exceeding a certain constant $k_0$ we identify a number $\dk$ such that $\gnd$
is $k$-colorable \whp\ if $d<\dk$ and non-$k$-colorable \whp\ if $d>\dk$.\\

\noindent
\emph{Key words:}	random graphs, graph coloring, phase transitions.
\end{abstract}

\section{Introduction}\label{Sec_intro}

Let $\gnd$ be the 
random $d$-regular graph on the vertex set $V=\cbc{1,\ldots,n}$. 
Unless specified otherwise, we let $d$ and $k\geq3$ be $n$-independent integers.
In addition, we let $\gnm$ denote the uniformly random graph on $V$ with precisely $m$ edges
	(the ``\Erdos-\Renyi\ model'').
We say that a property $\cE$ holds {\bf\em with high probability} (`\whp') if $\lim_{n\ra\infty}\pr\brk{\cE}= 1$.

\subsection{Results}

Determining the chromatic number of random graphs is one of the longest-standing challenges in probabilistic combinatorics.
For the \Erdos-\Renyi\ model, the single most intensely studied model in the random graphs literature,
the question dates back to the seminal 1960 that started the theory of random graphs~\cite{ER}.%
	\footnote{The chromatic number problems on $\gnm$ and on the binomial random graph
		(where each pair of vertices is connected with probability $p=m/\bink n2$ independently)
		turn out to be equivalent~\cite[\Chap~1]{JLR}.}
Apart from $\gnm$, the model that has received the most attention certainly is the random regular graph $\gnd$.
In the present paper, we provide an almost complete solution to the chromatic number problem on $\gnd$,
at least in the case that $d$ remains fixed as $n\ra\infty$ (which we regard as the most interesting regime).

The strongest previous result on the chromatic number of $\gnd$ is due to 
Kemkes, P\'erez-Gim\'enez and Wormald~\cite{KPGW}.
They proved that \whp\ for $k\geq3$
	\begin{eqnarray}\label{eqKPGW1}
	\chi(\gnd)=k&\mbox{ if }&
	d\in((2k-3)\ln(k-1),(2k-2)\ln(k-1)),\mbox{ and }\\
	\chi(\gnd)\in\cbc{k,k+1}&\mbox{ if }&d\in[(2k-2)\ln(k-1),(2k-1)\ln k].\label{eqKPGW2}
	\end{eqnarray}
This result yields the chromatic number precisely for the about $\ln k$ integers $d$
in the interval specified in~(\ref{eqKPGW1}),
and up to $\pm1$ for the about $\ln k$ integers in the subsequent interval~(\ref{eqKPGW2}).
Thus, (\ref{eqKPGW1})--(\ref{eqKPGW2}) determine $\chi(\gnd)$ exactly for ``about half'' of all degrees $d$.
The main result of the present paper is

\begin{theorem}\label{Thm_main}
There is a sequence $(\eps_k)_{k\geq 3}$ with $\lim_{k\ra\infty}\eps_k=0$ such that the following is true.
\begin{enumerate}
\item If $d\leq (2k-1)\ln k-2\ln 2-\eps_k$, then $\gnd$ is $k$-colorable \whp
\item If $d\geq (2k-1)\ln k-1+\eps_k$, then $\gnd$ fails to be $k$-colorable \whp
\end{enumerate}
\end{theorem}
(We have not attempted to explicitly extract or even optimize the explicit error term $\eps_k$.)

\Thm~\ref{Thm_main} implies the following ``threshold result''.

\begin{corollary}\label{Cor_main}
There is a constant $k_0>0$ such that for any integer $k\geq k_0$ there exists a number $\dk$ with the following two properties.
\begin{itemize}
\item If $d<\dk$, then $\gnd$ is $k$-colorable \whp
\item If $d>\dk$, then $\gnd$ fails to be $k$-colorable \whp
\end{itemize}
\end{corollary}
To obtain \Cor~\ref{Cor_main}, 
let $\eps_k$ as in \Thm~\ref{Thm_main} and consider the interval
	$$I_k=((2k-1)\ln k-2\ln 2-\eps_k,(2k-1)\ln k-1+\eps_k).$$
Then $I_k$ has length $2\ln 2-1+2\eps_k\approx 0.386+2\eps_k$.
Since $\eps_k\ra0$, for sufficiently large $k$ the interval $I_k$ contains at most one integer.
If it does, let $\dk$ be equal to this integer.
Otherwise, pick $\dk$ to be any number in $I_k$.


Phrased differently, \Thm~\ref{Thm_main} allows us to pin down the chromatic number $\chi(\gnd)$ exactly
for ``almost all'' $d$.

\begin{corollary}\label{Cor_chi}
There exist numbers $k_0,d_0>0$, a sequence $(\dk)_{k\geq k_0}$ 
and a function $\cF:\ZZpos\ra\ZZpos$ with the following properties.
\begin{enumerate}
\item[i.] We have $d_{(k+1)-\mathrm{col}}-\dk> 2\ln k$ for all $k\geq k_0$.
\item[ii.] For all $d\geq d_0$ such that $d\not\in\cbc{\dk:k\geq k_0}$ we have
		$\chi(\gnd)=\cF(d)$ \whp
\end{enumerate}
\end{corollary}

To obtain \Cor~\ref{Cor_chi}, let $(\dk)_{k\geq k_0}$ be the sequence from \Cor~\ref{Cor_main} and define $\cF(d)$ to be the largest integer $k$ such that $d<\dk$.
Then \Cor~\ref{Cor_main} directly implies ii., and i.\ follows from elementary calculations.


%
%

\subsection{Coloring random graphs: techniques and outline}

The best current results on coloring $\gnm$ as well as the best prior result
on $\chi(\gnd)$ are obtained via the {\em second moment method}~\cite{AchNaor,ACOVilenchik,KPGW}.
So are the present results.
Generally, suppose that $Z\geq0$ is a random variable such that $Z(G)>0$ only if $G$ is $k$-colorable.
If there is a number $C=C(k,d)>0$ such that
	\begin{equation}\label{eqSMM1}
	0<\Erw\brk{Z^2}\leq C\cdot\Erw\brk{Z}^2,
	\end{equation}
then the {\em Paley-Zygmund inequality}
	\begin{equation}\label{eqSMM2}
	\pr\brk{Z>0}\geq\frac{\Erw\brk{Z}^2}{\Erw\brk{Z^2}}
	\end{equation}
implies that 
there exists a $k$-coloring  with probability at least $1/C>0$.

What random variable $Z$ might be suitable?
The obvious choice seems to be the total number $\Zkc$ of $k$-colorings.
However, the calculations simplify substantially by working with the number $\Zkb$
of \bemph{balanced} $k$-colorings, in which all of the $k$ color classes are the same size (let us assume for now that $k$ divides $n$).
Indeed, the core of the paper by Achlioptas and Naor~\cite{AchNaor} is to establish the second moment bound~(\ref{eqSMM1})
for $\Zkb(\gnm)$ under the assumption that $d=2m/n\leq (2k-2)\ln k-2+o_k(1)$, with $o_k(1)$ a term that tends to $0$ as $k$ gets large.
Achlioptas and Naor rephrase this problem as a non-convex optimization problem over the \bemph{Birkhoff polytope}, i.e.,
the set of doubly-stochastic $k\times k$ matrices, and establish~(\ref{eqSMM1}) by solving a relaxation of this problem.
Thus, (\ref{eqSMM2}) implies that $\gnm$ is $k$-colorable with a non-vanishing probability if $d\leq (2k-2)\ln k-2+o_k(1)$.
This probability can be boosted to $1-o(1)$  by means of the sharp threshold result of Achlioptas and Friedgut~\cite{AchFried}.
In addition, a simple first moment argument shows that $\gnm$ is non-$k$-colorable \whp\ if $d>2k\ln k-1$.

Achlioptas and Moore~\cite{AMoColor} suggested to use the same random variable $\Zkb$ on $\gnd$.
They realised that the solution to the (relaxed) optimization problem over the Birkhoff polytope from~\cite{AchNaor} can be used as a ``black box''
to show that $\Zkb(\gnd)$ satisfies~(\ref{eqSMM1}) for {\em some} constant $C>0$.
Hence, (\ref{eqSMM2}) implies that $\gnd$ is $k$-colorable with a {\em non-vanishing} probability if $d\leq(2k-2)\ln k-2+o_k(1)$.
But unfortunately, in the case of random regular graphs there is no sharp threshold result to boost this probability to $1-o(1)$.
To get around this issue, Achlioptas and Moore instead adapt concentration arguments from~\cite{Luczak,ShamirSpencer} to the random regular graph $\gnd$.
However, these arguments inevitably one extra ``joker'' color.
Hence, Achlioptas and Moore obtain that $\chi(\gnd)\leq k+1$ \whp\ for $d\leq(2k-2)\ln k-2+o_k(1)$.

The contribution of Kemkes, P\'erez-Gim\'enez and Wormald~\cite{KPGW} is to remove the need for this additional color.
This enables them to establish~(\ref{eqKPGW1})--(\ref{eqKPGW2}), thus matching the result established in~\cite{AchNaor} for the \Erdos-\Renyi\ model $\gnm$.
Instead of employing ``abstract'' concentration arguments, Kemkes, P\'erez-Gim\'enez and Wormald use
the {\em small subgraph conditioning} technique from Robinson and Wormald~\cite{RobinsonWormald}.
Roughly speaking, they observe that the constant $C$ that creeps into the second moment bound~(\ref{eqSMM1})
results from the presence of {\em short cycles} in the random regular graph.
More precisely, in $\gnd$ any bounded-depth neighborhood of a {\em fixed} vertex $v$ is just a $d$-regular tree \whp\
However, in the {\em entire} graph $\gnd$ there will likely be a few cycles of bounded length.
In fact, it is well-known that for any length $j$ the number of short cycles is asymptotically a Poisson variable with mean $(d-1)^j/(2j)$.
As shown in~\cite{KPGW}, accounting carefully for the impact of short cycles allows to 
boost the probability of $k$-colorability to $1-o(1)$ without spending an extra color.

Recently, Coja-Oghlan and Vilenchik~\cite{ACOVilenchik} improved the result from~\cite{AchNaor} on the chromatic number of $\gnm$.
More precisely, they proved that $\gnm$ is $k$-colorable \whp\ if 
	\begin{equation}\label{eqACOVilenchik}
	d=2m/n\leq(2k-1)\ln k-2\ln 2-o_k(1),
	\end{equation}
gaining about an additive $\ln k$.
This improvement is obtained by considering a different random variable, namely the number $\Zkg$ of ``good'' $k$-colorings.
The definition of this random variable draws on intuition from non-rigorous statistical mechanics work on random graph coloring~\cite{pnas,LenkaFlorent}.
Crucially, the concept of good colorings facilitates the computation of the second moment.
The result is that the bound~(\ref{eqSMM1}) holds for $\Zkg(\gnm)$ for $d$ as in~(\ref{eqACOVilenchik}).
Hence, (\ref{eqSMM2}) shows that $\gnm$ is $k$-colorable with a non-vanishing probability for such $d$, and
the sharp threshold result~\cite{AchFried} boosts this probability to $1-o(1)$.

\Thm~\ref{Thm_main} provides a result matching~\cite{ACOVilenchik} for $\gnd$.
Following~\cite{KPGW}, we combine the second moment bound from~\cite{ACOVilenchik} (which we can use largely as a ``black box'')
with small subgraph conditioning.
Indeed, for the small subgraph conditioning argument we can use some of the computations performed in~\cite{KPGW} directly.
In the course of this, we observe a fairly simple, abstract link between partitioning problems on
$\gnd$ and on $\gnm$ that seems to have gone unnoticed in previous work (see \Sec~\ref{Sec_partitions}).
Due to this observation, relatively little new work is required to put the second moment argument together.
In effect, the main work in establishing the first part of \Thm~\ref{Thm_main} consists in computing the {\em first} moment of the number of good $k$-colorings in $\gnd$,
a task that turns out to be technically quite non-trivial.

The previous {\em lower} bounds on the chromatic number of $\gnd$ were based on simple first moment arguments over the number of $k$-colorings.
The bound that can be obtained in this way, attributed to Molloy and Reed~\cite{MolloyReed},
is that $\gnd$ is non-$k$-colorable \whp\ if $d>(2k-1)\ln k$.
By contrast, the second assertion in \Thm~\ref{Thm_main} marks a strict improvement over this naive first moment bound.
The proof is via an adaptation of techniques developed in~\cite{KostaNAE} for the random $k$-NAESAT problem.
Extending this argument to the chromatic number problem on $\gnd$ this requires substantial of technical work.
A matching improved lower bound on the chromatic number of $\gnm$ was recently obtained via a different argument~\cite{ACOcovers}.

After a discussion of further related work and some background and preliminaries in \Sec~\ref{Sec_pre}, 
we adapt the concept of good $k$-colorings from~\cite{ACOVilenchik} to $\gnd$ in \Sec~\ref{Sec_good}.
In \Sec~\ref{Sec_Prop_goodFirstMoment} we compute the first moment of the number of good colorings, thus accomplishing the main technical task
in proving the first part of \Thm~\ref{Thm_main}.
Then, in \Sec~\ref{Sec_second} we compute the second moment.
Finally, in \Sec s~\ref{Sec_lower} and~\ref{sec:prop:NoOfCompl+ItsDistr} we prove the second part of \Thm~\ref{Thm_main}, i.e., the lower bound on $\chi(\gnd)$.

%
%

\subsection{Further related work}

The chromatic number problem on $\gnm$ has attracted a big deal of attention.
A straight first moment argument yields a lower bound on $\chi(\gnm)$ that is within a factor two of
the number of colors that a simple greedy coloring algorithm needs~\cite{AchMolloy,GMcD}.
Closing this gap was a long-standing challenge until
\Bollobas~\cite{BBColor} managed to determine the asymptotic value of the chromatic number in the ``dense'' case $d=2m/n\gg n^{2/3}$.
His work improved Matula's result~\cite{Matula} published only shortly before.
Subsequently, \Luczak~\cite{LuczakColor} built upon Matula's  argument~\cite{Matula}
to determine $\chi(\gnm$) within a factor of $1+o(1)$ in the entire regime $d\gg1$.

In the case that $d$ remains bounded as $n\ra\infty$, \Luczak's result~\cite{LuczakColor} only yields
 $\chi(\gnm)$ up to a multiplicative $1\pm\eps_d$, where $\eps_d\ra0$ slowly in the limit of large $d$.
The aforementioned result of Achlioptas and Naor~\cite{AchNaor} marked a significant improvement
by computing $\chi(\gnm)$ for $d$ fixed as $n\ra\infty$ up to an {\em additive} error of $1$ for all $d$, and precisely for ``about half'' of all $d$.
Coja-Oghlan, Panagiotou and Steger~\cite{Angelika} combined  the techniques from~\cite{AchNaor} with concentration arguments from
Alon and Krivelevich~\cite{AlonKriv}
to obtain improved bounds on $\chi(\gnm)$ in the case $d\ll n^{1/4}$.

With respect to random regular graphs $\gnd$,
Frieze and \Luczak~\cite{FriezeLuczak} proved a result akin to \Luczak's \cite{LuczakColor} for $d\ll n^{1/3}$.
In fact, Cooper, Frieze, Reed and Riordan~\cite{CFRR} extended this result to the regime $d\leq n^{1-\eps}$ for any fixed $\eps>0$,
and Krivelevich, Sudakov, Vu and Wormald~\cite{KSVW} further still to $d\leq0.9n$.
For $d$ fixed as $n\ra\infty$, the bounds from~\cite{FriezeLuczak} were improved by the aforementioned contributions~\cite{AMoColor,KPGW}.

In addition, several papers deal with the $k$-colorability of random regular graphs for $k=3,4$.
This problem is not solved completely by~\cite{KPGW} (nor by the present work).
Achlioptas and Moore~\cite{AMo} and Shi and Wormald~\cite{Shi1} proved that $\chi(G(n,4))=3$ \whp, while
 Shi and Wormald~\cite{Shi2} showed that $\chi(G(n,6))=4$ \whp\
Moreover,  Diaz, Kaporis, Kemkes, Kirousis, P\'erez and Wormald~\cite{DKKKPW} proved that {\em if} a certain four-dimensional
optimization problem (which mirrors a second moment calculation) attains its maximum at a particular point, then $\chi(G(n,5))=3$ \whp\
Thus, determining $\chi(G(n,5))$ remains an open problem.


Precise conjectures as to the chromatic number of both $\gnm$ and $\gnd$ have been put forward on the basis
of sophisticated but non-rigorous physics considerations~\cite{BMPWZ,KPW,MPWZ,vanMourik,LenkaFlorent}.
These conjectures result from the application of generic (non-rigorous) methods, namely
	the {\em replica method} and the {\em cavity method}~\cite{MM}.
\Thm~\ref{Thm_main} largely confirms the physics conjecture on $\chi(\gnd)$  in the case of sufficiently large $d$.
More precisely, in physics terms the upper bound on $\chi(\gnd)$ provided by the first part of \Thm~\ref{Thm_main} corresponds to the ``replica symmetric ansatz'',
while the upper bound (asymptotically) matches the prediction of the ``1-step replica symmetry breaking ansatz''.
Indeed, the concept of ``good'' colorings, which is the basis of~\cite{ACOVilenchik} as well as the current work,
is directly inspired by physics ideas.


\section{Preliminaries}\label{Sec_pre}

In this section we collect a few elementary definitions and facts that will be refered to repeatedly throughout the paper.

\subsection{Basics}


Since \Thm~\ref{Thm_main} is a ``with high probability'' statement, we are generally going to assume that the number $n$ of vertices is sufficiently large.
Furthermore, \Thm~\ref{Thm_main} is an asymptotic statement in terms of $k$ due to the presence of the $\eps_k$ ``error term''.
Therefore, we are going to assume implicitly throughout that $k\geq k_0$ for a sufficiently large constant $k_0>0$.

We are going to use asymptotic notation with respect to both $n$ and $k$.
More precisely, we use $O(\cdot),\Omega(\cdot)$, etc.\ to denote asymptotics with respect to $n$.
For instance, $f(n)=O(g(n))$ means that there exists a number $C>0$ such that for $n>C$ we have $|f(n)|\leq C|g(n)|$.
This number $C$ may or may not depend on  $k$, the number of colors.
By contrast, we denote asymptotics with respect to $k$ by the symbols $O_k(\cdot),\Omega_k(\cdot)$, etc.;
	these asymptotics are understood to hold uniformly in $n$.
Thus, $f(k)=O_k(g(k))$ means that there is a number $C>0$ that is independent of both $n$ and $k$ such that
for $k>C$ we have $|f(k)|\leq C|g(k)|$.
Furthermore, we use the notation $f(k)=\tilde O_k(g(k))$ to indicate that for some $C>0$ independent of $n$ and $k$ and for $k>C$ we have
	$$|f(k)|\leq |g(k)|\cdot\ln^Ck.$$

If $\xi=(\xi_1,\ldots,\xi_l)$ is a vector and $1\leq p\leq\infty$, then $\norm\xi_p$ denotes the $p$-norm of $\xi$.
For a matrix $A=(a_{ij})_{i\in\brk M,j\in\brk N}$ we let $\norm A_p$ signify the
$p$-norm of $A$ viewed as the $N\cdot M$-dimensional vector $(a_{11},\ldots,a_{MN})$.

We also need some basic facts from the theory of large deviations.
Let $\cX$ be a finite set and let $\mu,\nu:\cX\ra\brk{0,1}$ be two maps such that
	$\sum_{x\in\cX}\mu(x),\sum_{x\in\cX}\nu(x)\leq1$ and such that $\mu(x)=0$ if $\nu(x)=0$ for all $x\in\cX$.
Let
	$$H(\mu)=-\sum_{x\in\cX}\mu(x)\ln\mu(x)$$
denote the \bemph{entropy} of $\mu$.
In addition, we denote the \bemph{Kullback-Leibler divergence} of $\mu,\nu$ by
	$$\KL{\mu}{\nu}=\sum_{x\in\cX}\mu(x)\ln\frac{\mu(x)}{\nu(x)}.$$
Throughout the paper, we use the convention that $0\ln 0=0$, $0\ln(0/0)=0$. 
It is easy to compute the first two differentials of the function $\mu\mapsto\KL{\mu}{\nu}$:
	\begin{eqnarray}\label{eqKLdiff1}
	\frac{\partial\KL\mu\nu}{\partial\mu(x)}&=&1+\ln\frac{\mu(x)}{\nu(x)},\\ 
	\frac{\partial^2\KL\mu\nu}{\partial\mu(x)^2}&=&1/\mu(x),
		\quad\frac{\partial^2\KL\mu\nu}{\partial\mu(x)\partial\mu(x')}=0. 
			\label{eqKLdiff2}
	\end{eqnarray}
Furthermore, we need the following well-known 

\begin{fact}\label{Fact_KL}
Assume that $\mu,\nu$ are probability distributions on $\cX$ such that that $\mu(x)=0$ if $\nu(x)=0$.
\begin{enumerate}
\item  We always have $\KL\mu\nu\geq0\mbox{ while }\KL\mu\nu=0\mbox{ iff $\mu=\nu$.}$
\item The function $\mu\mapsto\KL{\mu}{\nu}$ is convex.
\item There is a number $\xi=\xi(\nu)>0$ such that for any $\mu$ we have
		$\KL{\mu}{\nu}\geq\xi\sum_{x\in\cX}(\mu(x)-\nu(x))^2.$
\end{enumerate}
\end{fact}

In the case that $\cX=\cbc{0,1}$ has only two elements,
a probability distribution $\mu$ on $\cX$ can be encoded by a single number, say, $\mu(1)$.
It is well known that with this convention, we have the following large deviations principle for the
binomial distribution:
for any $p,q\in(0,1)$,
\begin{eqnarray}
	\frac1n\ln\pr\brk{\Bin(n,q)=pn}&=& -\KL pq+\Theta\left(\frac{\ln n}{n}\right)\nonumber\\
		&=&p\ln\frac qp+(1-p)\ln\frac{1-q}{1-p}+\Theta\left(\frac{\ln n}{n}\right).\label{eq:BinomailVsKLDiv}\label{eqBinLDP}
\end{eqnarray}
Additionally, we have the following \emph{Chernoff bound}~\cite[p.~21]{JLR}.

\begin{lemma}\label{Lemma_Chernoff}
Let $\varphi(x)=(1+x)\ln(1+x)-x$.
Let $X$ be a binomial random variable with mean $\mu>0$.
Then for any $t>0$ we have
	\begin{eqnarray*}
	\pr\brk{X>\mu+t}&\leq&\exp(-\mu\cdot\varphi(t/\mu)),\quad
	\pr\brk{X<\mu-t}\leq\exp(-\mu\cdot\varphi(-t/\mu)).
	\end{eqnarray*}
In particular, for any $t>1$ we have
	$\pr\brk{X>t\mu}\leq\exp\brk{-t\mu\ln(t/\eul)}.$
\end{lemma}


For a real $a$ and an integer $j\geq0$ let us denote by
	$$(a)_j=\prod_{i=1}^j(a-i+1)$$
the \bemph{$j$th falling factorial of $a$}.
We need the following well-known result on convergence to the Poisson distribution
	(e.g., \cite[p.~26]{BB}).

\begin{theorem}\label{Thm_Poisson}
Let $\lambda_1,\ldots,\lambda_l>0$.
Suppose that $X_1(n),\ldots,X_l(n)\geq0$ are sequences of integer-valued random variables such that for any family
$q_1,\ldots,q_l$ of non-negative integers it is true that
	$$\Erw\brk{\prod_{j=1}^l(X_j(n))_{q_j}}\sim \prod_{j=1}^l\lambda_j^{q_j}\qquad\mbox{as }n\ra\infty.$$
Then for any $q_1,\ldots,q_l$ we have
	\begin{equation}\label{eqPoisson}
	\pr\brk{X_1(n)=q_1,\ldots,X_l(n)=q_l}\sim\prod_{j=1}^l\pr\brk{\Po(\lambda_j)=q_j}.
	\end{equation}
\end{theorem}
If~(\ref{eqPoisson}) holds for any $q_1,\ldots,q_l$, 
then $X_1(n),\ldots,X_l(n)$ are \bemph{asymptotically independent $\Po(\lambda_j)$ variables}.

In many places throughout the paper we are going to encounter the hypergeometric distribution.
The following well-known relationship between the hypergeometric distribution and the binomial distribution will
simplify many estimates.

\begin{lemma}\label{lem_balls_bins_cap}
For any integer $d>1$ there exists a number $C=C(d)>0$ such that the following is true.
Let $U$ be a set of size $u>1$.
Choose a set $S\subset U\times\brk d$ size $\abs S=s\geq1$ uniformly at random and let $e_v=S\cap(\cbc v\times\brk d)$.
Furthermore, let $(b_v)_{v\in U}$ be a family of independent $\Bin(d,\frac{s}{d u})$ variables.
Then for any sequence $(t_v)_{v\in U}$ of non-negative integers such that $\sum_{v \in U} t_v = \mu$ we have 
	$$\pr\brk{\forall v \in S: e_{v}=t_{v}}=\pr\brk{\forall v \in S: b_v=t_{v}\bigg|\sum_{v\in U}b_v=s}\leq C\sqrt u\cdot\pr\brk{\forall v \in M: b_v=t_{v}}.$$
\end{lemma}

%

Finally, the following version of the chain rule will come in handy.

\begin{lemma}\label{Lemma_chainrule}
Suppose that $g:\RR^a\ra\RR^b$ and $f:\RR^b\ra\RR$ are functions with two continuous second derivatives.
Then for any $x_0\in\RR^a$ and with $y_0=g(x_0)$ we have for any $i,j\in\brk a$
\begin{eqnarray*}
\frac{\partial^2f\circ g}{\partial x_i\partial x_j}\bigg|_{x_0}&=&
	\sum_{k=1}^b\frac{\partial f}{\partial y_k}\bigg|_{y_0}\frac{\partial^2 g_k}{\partial x_i\partial x_j}\bigg|_{x_0}+
		\sum_{k,l=1}^b\frac{\partial^2f}{\partial y_k\partial y_l}\bigg|_{y_0}\frac{\partial g_k}{\partial x_i}\bigg|_{x_0}\frac{\partial g_l}{\partial x_j}\bigg|_{x_0}.
\end{eqnarray*}
\end{lemma}

\subsection{The configuration model}

As our goal is to study random $d$-regular graphs on $n$ vertices, we will always assume that $dn$ is even.
To get a handle on the random regular graph $\gnd$, we work with the {\em configuration model}~\cite{BBConf}.
More precisely, an \bemph{$(n,d)$-configuration} is a map $\Gamma:V\times\brk d\ra V\times\brk d$ such
that $\Gamma\circ\Gamma=\id$. 
In other words, an $(n,d)$-configuration is a perfect matching of the complete graph on $V\times\brk d$.
Thus, the total number of $(n,d)$-configurations is equal to
	\begin{equation}\label{eqMatchings}
	(dn-1)!!=\frac{(dn)!}{2^{dn/2}(dn/2)!}=\Theta(\sqrt{(dn)!}/(dn)^\frac14).
	\end{equation}
We call the pairs $(v,j)$, $j\in\brk d$ the {\bf\em clones} of $v$.

Any $(n,d)$-configuration $\Gamma$ induces a multi-graph 
 with vertex set $V$ by contracting the $d$ clones of each $v\in V$ into a single vertex.
Throughout, we are going to denote a uniformly random $(n,d)$-configuration by $\vec\Gamma$.
Furthermore, $\Gnd$ denotes the multi-graph obtained
from $\vec\Gamma$.
The relationship between $\Gnd$ and the simple random $d$-regular graph $\gnd$ is as follows.

\begin{lemma}[\cite{BBConf}]\label{Lemma_configurationModel}
Let $\cS(n,d)$ denote the event that $\Gnd$ is a simple graph.
Then for any event event $\cB$ we have
	$\pr\brk{\gnd\in\cB}=\pr\brk{\Gnd\in\cB|\cS(n,d)}.$
Furthermore, there is an $n$-independent number $\eps_d>0$ such that
	$\pr\brk{\cS(n,d)}\geq\eps_d$.
\end{lemma}
Thus, if we want to show that some ``bad'' event $\cB$ does not occur in $\gnd$ \whp, then 
it suffices to prove that this event does not occur in the random multi-graph $\Gnd$ \whp

For two sets $A, B\subset V$ of vertices we let
	$$e_{\Gnd}(A,B) =\abs{\cbc{(v,i)\in A\times\brk d:\vec\Gamma(v,i)\in B\times\brk d}}=
			\abs{\cbc{(w,j)\in B\times\brk d:\vec\Gamma(w,j)\in A\times\brk d}}$$
denote the number of $A$-$B$-edges in $\Gnd$.
If $A=\cbc{v}$, we use the shorthand $e_{\Gnd}(v,B)$, which is nothing but the number $v$-$B$ edges.
(Of course, as $\Gnd$ is a multi-graph, this is not necessarily the same as the number of neighbors of $v$ in $B$.)
If $A=B$, we let
	$$e_{\Gnd}(A) =e_{\Gnd}(A,A).$$

\subsection{Partitions of random regular graphs}\label{Sec_partitions}

The graph coloring problem is basically just a particular kind of graph partitioning problem.
Therefore, the following (as we believe, elegant) estimate of the probability that the random
regular graph admits a particular partition will be quite useful;
	it seems to have gone unnoticed so far.

Let $K\geq2$ be an integer and let $\rho=(\rho_i)_{i\in\brk K}$ be a probability distribution on $\brk K$.
Moreover, let $\mu=(\mu_{ij})_{i,j\in\brk K}$ be a probability distribution on $\brk K\times\brk K$ 
such that $\mu_{ij}=\mu_{ji}$ for all $i,j\in\brk K$.
We say that $(\rho,\mu)$ is \bemph{$(d,n)$-admissible} if $\rho_in$, $\mu_{ij}dn$ are integers for all 
$i,j\in\brk K$ and if 	$$\sum_{j\in\brk K}\mu_{ij}=\sum_{j\in\brk K}\mu_{ji}=\rho_i\quad\mbox{ for all $i\in\brk K$.}$$
In other words, $\rho$ is the marginal distribution of $\mu$ (in both dimensions).
Let $\rho\tensor\rho$ denote the product distribution $(\rho_i\rho_j)_{i,j\in\brk K}$ on $\brk K\times\brk K$.


\begin{lemma}\label{Lemma_partition}
Let $(\rho,\mu)$ be $(d,n)$-admissible.
Moreover, let $V_1,\ldots,V_K$ be a partition of the vertex set $V$ such that $|V_i|=\rho_in$ for all $i\in\brk K$.
Then 
	\begin{equation}\label{eqLemma_partition}
	\frac1n\ln\pr\brk{\forall i,j\in\brk K:e_{\Gnd}(V_i,V_j)=\mu_{ij}dn}=-\frac d2\KL{\mu}{\rho\tensor\rho}+O(\ln n/n).
	\end{equation}
\end{lemma}

Before we prove \Lem~\ref{Lemma_partition}, let us try to elucidate the statement a little.
If we  fix the partition $V_1,\ldots,V_K$ and generate a random multi-graph $\Gnd$, then the expected number of edges between any two classes is just
	$$\Erw\brk{e_{\Gnd}(V_i,V_j)}=\rho_i\rho_j dn.$$
Thus, the ``expected edge density'' of the partition $V_1,\ldots,V_K$ is given by the product distribution $\rho\tensor\rho$.
The point of \Lem~\ref{Lemma_partition} is that it
 provides an estimate of the probability that
the fraction of edges that run between any two partition classes $V_i,V_j$ (or within one class if $i=j)$ follows some other the distribution~$\mu$.
Unless $\mu$ is very close to $\rho\tensor\rho$, the probability of this event is exponentially small, and
\Lem~\ref{Lemma_partition} yields an accurate estimate in terms of the Kullback-Leibler divergence of $\mu$ and the ``expected'' distribution $\rho\tensor\rho$.

Interestingly, a simple calculation shows that~(\ref{eqLemma_partition}) holds true if we replace $\Gnd$ by the \Erdos-\Renyi\ random graph $\gnm$ 
(with $m=dn/2$).
In other words, on a logarithmic scale the probability of observing a particular edge distribution $\mu$ is the same in both models.
This observation will be crucial for us to extend the second moment calculation that was performed in~\cite{ACOVilenchik} for $\gnm$ to the random regular graph $\gnd$.


\medskip\noindent{\em Proof of \Lem~\ref{Lemma_partition}.}
Let $\cE$ be the event that $e_{\Gnd}(V_i,V_j)=\mu_{ij}dn$ for all $i,j\in\brk K$.
Let us call a map $\sigma:V\times\brk d\ra\brk K$ a {\em $\mu$-shading} if for all $i,j\in\brk K$ we have
	$$\abs{\cbc{(v,l)\in V_i\times\brk d:\sigma(v,l)=j}}=\mu_{ij}dn.$$
Clearly, the total number of $\mu$-shadings is just
	$$\cN_{\mu}=\prod_{i=1}^K\bink{\rho_i dn}{\mu_{i1}dn,\ldots,\mu_{iK}dn}.$$

Any configuration $\Gamma$ 
that induces a multi-graph $\cG$ such that $e_{\Gnd}(V_i,V_j)=\mu_{ij}dn$ for all $i,j\in\brk K$
induces a $\mu$-shading $\sigma_\Gamma$.
Indeed, the shade of a clone $(v,l)$ is just the index $j\in\brk K$ such that $\Gamma(v,l)\in V_j\times\brk d$.

Conversely, for a given $\mu$-shading $\sigma$, how many configurations $\Gamma$ are 
there such that $\sigma=\sigma_\Gamma$?
To obtain such a configuration, we need to match the clones $(v,l)\in V_i\times\brk d$ with $\sigma(v,l)=j$
to the clones $(v',l')\in V_j\times\brk d$ such that $\sigma(v',l')=i$ for all $1\leq i\leq j\leq K$.
Clearly, the total number of such matchings is
	$$
	\cM_{\mu}=\prod_{1\leq i<j\leq K}(\mu_{ij}dn)!\cdot \prod_{i=1}^K(\mu_{ii}dn-1)!!.
	$$
Hence, 
	\begin{eqnarray}\label{eqprE}
	\pr\brk{\cE}&=&\frac{\cN_{\mu}\cM_{\mu}}{(dn-1)!!}
	\end{eqnarray}

Using Stirling's formula and~(\ref{eqMatchings}), we find that
	\begin{eqnarray*}
	\ln\cN_{\mu}&=&dn\sum_{i,j=1}^K\mu_{ij}\ln(\rho_i/\mu_{ij})+O(\ln n),\\
	\ln\frac{\cM_{\mu}}{(dn-1)!!}&=&\frac12\ln\frac{\prod_{i,j=1}^K(\mu_{ij}dn)!}{(dn)!}+O(\ln n)=-\frac12\ln\bink{dn}{(\mu_{ij}dn)_{i,j\in\brk K}}+O(\ln n)\\
		&=&\frac{dn}2\sum_{i,j=1}^K\mu_{ij}\ln\mu_{ij}+O(\ln n).
	\end{eqnarray*}
Plugging these estimates into~(\ref{eqprE}), we obtain
	\begin{eqnarray*}
	\ln\pr\brk\cE&=&\frac{dn}2\sum_{i,j=1}^K\mu_{ij}\bc{2\ln\frac{\rho_i}{\mu_{ij}}+\ln\mu_{ij}}+O(\ln n)
		=\frac{dn}2\sum_{i,j=1}^K\mu_{ij}\ln\frac{\rho_i^2}{\mu_{ij}}+O(\ln n)\\
		&=&\frac{dn}2\sum_{i,j=1}^K\mu_{ij}\ln\frac{\rho_i\rho_j}{\mu_{ij}}+O(\ln n)\qquad\mbox{[as $\mu_{ij}=\mu_{ji}$ for all $i,j\in\brk K$]}\\
		&=&-\frac{dn}2\KL{\mu}{\rho\tensor\rho}+O(\ln n),
	\end{eqnarray*}
as claimed.
%
\qed

\begin{corollary}\label{Lemma_skewed}
Let $(\rho,\mu)$ be $(d,n)$-admissible and let $Z_{\mu}$ denote the number
of partitions $V_1,\ldots,V_K$ of $V$ such that
	\begin{eqnarray}\label{eqLemma_skewedA}
	|V_i|&=&\rho_in\qquad\mbox{for all $i\in\brk K$, and }\\
	e_{\Gnd}(V_i,V_j)&=&\mu_{ij}dn\quad\mbox{for all }i,j\in\brk K.\label{eqLemma_skewedB}
	\end{eqnarray}
%
%
Then 
	\begin{equation}\label{eqLemma_skewedKL}
	\frac1n\ln\Erw\brk{Z_{\mu}}=H(\rho)-\frac{d}2\KL{\mu}{\rho\tensor\rho}+O(\ln n/n).
	\end{equation}
\end{corollary}
\begin{proof}
\Lem~\ref{Lemma_partition} provides the probability that
for any {\em fixed} partition 
	$V_1,\ldots,V_K$ we have $e_{\Gnd}(V_i,V_j)=\mu_{ij}dn$ for all $i,j\in\brk K$.
Furthermore, by Stirling's formula the total number of partitions $V_1,\ldots,V_K$ with $|V_i|=\rho_in$ for all $i\in\brk K$ is
	\begin{equation}\label{eqLemma_partition2}
	\bink{n}{\rho_1n,\ldots,\rho_kn}=\exp\brk{H(\rho)n+O(\ln n)}.
	\end{equation}
Thus, the assertion follows from~(\ref{eqLemma_partition}), (\ref{eqLemma_partition2}) and the linearity of expectation.
\qed\end{proof}

Finally, the expression~(\ref{eqLemma_skewedKL}) can be restated in a slightly more handy form if we assume that $\mu_{ii}=0$ for all $i\in\brk K$.
More precisely, we have

\begin{corollary}\label{Cor:Skewed}
Let $(\rho,\mu)$ be $(d,n)$-admissible such that $\mu_{ii}=0$ for all $i\in\brk K$.
Let $Z_{\mu}$ denote the number of partitions $V_1,\ldots,V_K$ that
satisfy~(\ref{eqLemma_skewedA}) and~(\ref{eqLemma_skewedB}).
Moreover, let $\hat\rho=(\rho_{ij})_{i,j\in\brk K}$ be the probability distribution  defined by
	$$\hat\rho_{ij}=\frac{\vecone_{i\neq j}\cdot\rho_i\rho_j}{1-\norm\rho_2^2}.$$
Then
	\begin{eqnarray}\label{eqCor:Skewed}
	\frac1n\ln\Erw\brk{Z_\mu}=H(\rho)+\frac{d}{2}\ln (1-\norm\rho_2^2)-\frac{d}2\KL{\mu}{\hat\rho}+O(\ln n/n).
	\end{eqnarray}
\end{corollary}
\begin{proof}
\Cor~\ref{Lemma_skewed} yields
\begin{eqnarray*}
	\frac1n\ln\Erw[Z_{\mu}]&=&H(\rho)-\frac d2\sum_{i,j=1}^K\mu_{ij}\ln\frac{ {\mu}_{ij}}{ {\rho}_i {\rho}_j}+O\left(\frac{\log n}{n}\right).
\end{eqnarray*}
Setting $y=\norm\rho_2^2=\sum_{i=1}^k {\rho}_i^2$, we get
\begin{eqnarray*}
	\frac1n\ln\Erw[Z]&=&H(\rho)+\frac d2\ln(1-y)-\frac d2\sum_{i,j=1}^K\mu_{ij}\ln\frac{(1-y)\mu_{ij}}{\rho_i\rho_j}+O\left(\frac{\log n}{n}\right)
		\quad\mbox{[as $\sum_{i,j=1}^K\mu_{ij}=1$]}\\
		&=&\cH+\frac d2\ln(1-y)-
			\frac d2 D_{KL}(\mu,\hat\rho)+O\left(\frac{\log n}{n}\right),\qquad\qquad\mbox{[as $\mu_{ii}=0$ for all $i\in\brk K$]}
\end{eqnarray*}
as claimed.
\qed\end{proof}

For a given collection $\rho$ of class sizes, \Cor~\ref{Cor:Skewed} identifies the edge distribution $\mu$ for which $\Erw\brk{Z_\mu}$ is maximized
subject to the condition that $\mu_{ii}=0$ for all $i$.
Indeed, the maximizer is just $\mu=\hat\rho$.
This is because $\KL{\mu}{\hat\rho}\geq0$ for all $\mu$, and $\KL{\mu}{\hat\rho}=0$ iff $\mu=\hat\rho$ (by Fact~\ref{Fact_KL}).
Furthermore, the term $\KL{\mu}{\hat\rho}$ captures precisely just how ``unlikely'' it is to see some other
edge distribution $\mu\neq\hat\rho$.

\subsection{Small subgraph conditioning}

To show that $\Gnd$ is $k$-colorable \whp\ we are going to use the second moment method.
This is facilitated by the following statement, which is
an immediate consequence of~\cite[\Thm~1]{Janson} (which, in turn, generalizes~\cite{RobinsonWormald}).

\begin{theorem}\label{Thm_smallSubgraphConditioning}
Let $d,k\geq3$ and assume that $k$ divides $n$ and that $dn$ is even.
Let
	\begin{equation}\label{eqsmallSubgraphConditioning}
	\lambda_j=\frac{(d-1)^j}{2j}\quad\mbox{ and }\quad\delta_j=-(1-k)^{1-j} 
	\end{equation}
and let
 $\Xi_l$ be the number of cycles of length $l$ in $\Gnd$ for $l\geq1$ (with 1-cycles being self-loops and 2-cycles being multiple edges).
Suppose that $Y=Y(\Gnd)\geq0$ is a random variable with the following properties.
\begin{enumerate}
\item[i.]  $\Erw\brk Y=\exp(\Omega(n))$.
\item[ii.]  For any sequence $q_1,\ldots,q_l$ of non-negative integers (that remains fixed as $n\ra\infty$) we have
			$$\Erw\brk{Y\cdot\prod_{j=1}^l(\Xi_j)_{q_j}}\sim\Erw\brk Y\cdot\prod_{j=1}^l(\lambda_j(1+\delta_j))^{q_j}.$$
\item[iii.] 	$\Erw\brk{Y^2}\leq(1+o(1))\Erw\brk Y^2\cdot\exp\brk{\sum_{j=1}^\infty\lambda_j\delta_j^2}.$ 
\end{enumerate}
Then $\pr\brk{Y>0|\Xi_1=0}=1-o(1)$.
\end{theorem}
The very same statement is also the basis of the second moment argument in~\cite{KPGW}.
The use of \Thm~\ref{Thm_smallSubgraphConditioning} is referred to as {\em small subgraph conditioning} because
verifying the assumptions of the theorem amounts to studying the random variable $Y$ {\em given} the number of short
cycles in $\Gnd$.

\section{Upper-bounding the chromatic number: outline}\label{Sec_good}
{\em Throughout this section, we assume that $k$ divides $n$ and that 
		\begin{equation}\label{eqSec_Prop_goodFirstMoment}
		(2k-2)k\ln(k-1)\leq d\leq (2k-1)\ln k-2\ln 2-\eps_k		\end{equation}
		 for a sequence $\eps_k$ 	that tends to $0$ sufficiently slowly in the limit of large $k$.}


\bigskip
\noindent
In this section we introduce the random variable upon which the proof of the first part of \Thm~\ref{Thm_main} is based.
The first random variable that springs to mind certainly is the total number $\Zkc$ of $k$-colorings.
However, the corresponding formulas for the first and the second moment turn out to be somewhat unwieldy.
Therefore, following~\cite{AchNaor,KPGW}, we confine ourselves to colorings that have the following property.

\begin{definition}
A map $\sigma:V\ra\brk k$ is {\bf\em balanced} if $|\sigma^{-1}(i)|=n/k$ for all $i\in\brk k$.
\end{definition}

The number $\Zkb=\Zkb(\Gnd)$ of balanced $k$-colorings is the random variable used in~\cite{KPGW}.
Unfortunately, it is not possible to base the proof of \Thm~\ref{Thm_main} on $\Zkb$.
Indeed, there exist infinitely many $k$ such that
for $d=\lfloor(2k-1)\ln k-2\ln 2\rfloor$ we have
	$$\Erw\brk{\Zkb^2}\geq\exp(\Omega(n))\Erw\brk{\Zkb}^2.$$
Thus, $\Zkb$ does {\em not} satisfy the second moment condition~(\ref{eqSMM1}).

To cope with this issue, we use a different random variable from~\cite{ACOVilenchik}.
Its definition is inspired by statistical mechanics predictions on the geometry of the set of $k$-colorings of the random graph.
According to these, 
for $d>(1+o_k(1))k\ln k$ the set of $k$-colorings, viewed as a subset of $\brk k^V$, decomposes into an exponential number
of well-separated `clusters'.

To formalize this notion, let $\sigma,\tau:V\ra\brk k$ be two balanced maps.
Their {\bf\em overlap matrix} is the $k\times k$ matrix $\rho(\sigma,\tau)$ with entries
	\begin{equation}\label{eqOverlapMatrix}
	\rho_{ij}(\sigma,\tau)=\frac kn\cdot |\sigma^{-1}(i)\cap\tau^{-1}(j)|\qquad\mbox{(cf.~\cite{AchNaor})}.
	\end{equation}
This matrix $\rho(\sigma,\tau)$ is doubly-stochastic.
Following~\cite{ACOVilenchik}, we define the \bemph{cluster} of a $k$-coloring $\sigma$ of a graph $G$ to be the set
	\begin{equation}\label{XeqCluster}
	\cC(\sigma)=\cC_G(\sigma)=\cbc{\tau\in\brk k^n:
		\tau\mbox{ is a balanced $k$-coloring of $G$ and $\rho_{ii}(\sigma,\tau)>0.51$ for all $i\in\brk k$}}.
	\end{equation}
Thus, $\cC(\sigma)$ consists of all balanced $k$-colorings $\tau$ that leave the color of at least
$51\%$ of the vertices in each color class of $\sigma$ unchanged.
In addition, also following~\cite{ACOVilenchik}, we have

\begin{definition}\label{Def_separable}
A balanced $k$-coloring $\sigma$ is \bemph{separable} in $G$
if for any other balanced $k$-coloring $\tau$ of $G$
and any $i,j\in\brk k$ such that $\rho_{ij}(\sigma,\tau)>0.51$ we indeed have $\rho_{ij}(\sigma,\tau)\geq1-\kappa$,
where $\kappa=\ln^{500}k/k=o_k(1)$.
\end{definition}

These definitions ensure that the clusters of two separable $k$-colorings $\sigma,\tau$ are either disjoint or identical.
In addition, we would like to formalize the notion that there are many disjoint clusters.
To this end, we simply put an explicit upper bound on the size of each cluster;
	this is going to entail that many clusters are necessary to exhaust the entire set of $k$-colorings.
We thus arrive at

\begin{definition}[\cite{ACOVilenchik}]\label{XDef_good}
A balanced $k$-coloring $\sigma$ of $\Gnd$ is {\bf\em good} if it is separable and 
	$$\abs{\cC(\sigma)}\leq\frac1n\Erw\brk{\Zkb}.$$
\end{definition}

Let $\Zkg=\Zkg(\Gnd)$ be the number of good $k$-colorings.
We need to estimate $\Erw\brk{\Zkg}$.
The first step is to compute the expected number of balanced $k$-colorings.
Fortunately, we do not need to perform this computation from scratch
since it has already been performed in~\cite{KPGW}.

\begin{proposition}[\cite{KPGW}]\label{Prop_KPGWfirstMoment}
We have
	$$\Erw\brk{\Zkb}=\Theta(n^{-(k-1)/2})\cdot k^n(1-1/k)^{dn/2}.$$
Moreover, $\Zkb$ satisfies condition ii.\ in \Thm~\ref{Thm_smallSubgraphConditioning}.
\end{proposition}

In addition to the size of the color classes, we also need to control the edge densities between them.
%
Let us call a balanced $k$-coloring $\sigma$ of $\Gnd$ \bemph{skewed} if
	$$
		\max_{1\leq i<j\leq k}\abs{e_{\Gnd}(\sigma^{-1}(i),\sigma^{-1}(j))-\frac{dn}{k(k-1)}}>\sqrt n\ln n.$$

\begin{corollary}\label{Cor_skewed}
Let $\Zkb'$ be the number of skewed balanced $k$-colorings of $\Gnd$.
Then
	$$\Erw\brk{\Zkb'}\leq\exp(-\Omega(\ln^2n))\cdot\Erw\brk{\Zkb}.$$
\end{corollary}
\begin{proof}
The proof is based on \Cor~\ref{Cor:Skewed}.
Let $\rho=k^{-1}\vecone$ be the uniform distribution on $\brk k$.
Moreover, let $\mu=(\mu_{ij})_{i,j\in\brk k}$ be a probability distribution
such that $(\rho,\mu)$ is an admissible pair, and such that $\mu_{ii}=0$ for all $i\in\brk k$.
As in \Cor~\ref{Cor:Skewed}, let $Z_\mu$ be the number of balanced $k$-colorings $\sigma$ such that
the edge densities between the color classes are given by $\mu$, i.e.,
	$$e_{\Gnd}(\sigma^{-1}(i),\sigma^{-1}(j))=\mu_{ij}dn\qquad\mbox{for all }i,j\in\brk k.$$
Furthermore, let
	$\hat\rho=(\rho_{ij})_{i,j\in\brk k}$ be the probability distribution on $\brk k\times\brk k$ defined by $\rho_{ij}=\frac{\vecone_{i\neq j}}{k(k-1)}$.
Then \Cor~\ref{Cor:Skewed} and \Prop~\ref{Prop_KPGWfirstMoment} yield
	\begin{eqnarray}\nonumber
	\frac1n\ln\Erw\brk{Z_\mu}&=&\ln k+\frac d2\ln(1-1/k)-\frac d2\KL{\mu}{\hat\rho}+O(\ln n/n)\\
		&=&\frac1n\ln\Erw\brk{\Zkb}-\frac d2\KL{\mu}{\hat\rho}+O(\ln n/n).\label{eqCor_skewedA}
	\end{eqnarray}
Furthermore, by Fact~\ref{Fact_KL} there is an $n$-independent number $\xi=\xi(k)>0$ such that
	$$\KL{\mu}{\hat\rho}\geq \xi\sum_{i,j=1}^k(\mu_{ij}-\hat\rho_{ij})^2.$$
Hence, if $\mu$ is such that $|dn\mu_{ij}-dn\rho_{ij}|>\sqrt n\ln n$ for some pair $(i,j)\in\brk k\times\brk k$,
then $\KL{\mu}{\hat\rho}=\Omega(\ln^2n/n)$.
Therefore, (\ref{eqCor_skewedA}) implies that
	\begin{eqnarray}\label{eqCor_skewedB}
	\Erw\brk{Z_\mu}&\leq&\exp(-\Omega(\ln^2n))\cdot\Erw\brk{\Zkb}.
	\end{eqnarray}
To complete the proof, let $\cM$ be the set of all $\mu$ such that $(\rho,\mu)$ is an admissible pair and such that
	$|dn\mu_{ij}-dn\rho_{ij}|>\sqrt n\ln n$ for some $(i,j)\in\brk k\times\brk k$.
Because $dn\mu_{ij}$ has to be an integer for all $i,j\in\brk k$,
we can estimate $\abs{\cM}\leq (dn)^{k^2}$ (with room to spare), i.e., $\abs\cM$ is bounded by a polynomial in $n$.
Hence, (\ref{eqCor_skewedB}) yields
	$$\Erw\brk{\Zkb'}\leq\sum_{\mu\in\cM}\Erw\brk{Z_\mu}\leq \abs{\cM}\exp(-\Omega(\ln^2n))\cdot\Erw\brk{\Zkb}\leq \exp(-\Omega(\ln^2n))\cdot\Erw\brk{\Zkb},$$
as desired.
\qed\end{proof}

In \Sec~\ref{Sec_Prop_goodFirstMoment} we use \Cor~\ref{Cor_skewed} to compare $\Zkg$ and $\Zkb$;
the result is

\begin{proposition}\label{Prop_goodFirstMoment}
We have $\Erw\brk{\Zkg}\sim\Erw\brk{\Zkb}$.
\end{proposition}
Combining \Prop~\ref{Prop_KPGWfirstMoment} and~\ref{Prop_goodFirstMoment}, we obtain the following.

\begin{corollary}\label{Cor_goodFirstMoment}
The random variable $\Zkg$ satisfies conditions i.\ and ii.\ in \Thm~\ref{Thm_smallSubgraphConditioning}.
\end{corollary}
\begin{proof}
Condition i.\ follows directly from \Prop s~\ref{Prop_KPGWfirstMoment} and~\ref{Prop_goodFirstMoment}.
Indeed, using the expansion $\ln(1-x)=-x-x^2/2+O(x^3)$, we find that
	\begin{eqnarray*}
	\frac1n\ln\Erw[\Zkg]&\sim&\frac1n\ln\Erw[\Zkb]\qquad\mbox{[by \Prop~\ref{Prop_goodFirstMoment}]}\\
		&\sim&\ln k+\frac d2\ln(1-1/k)\qquad\quad\mbox{[by \Prop~\ref{Prop_KPGWfirstMoment}]}\\
		&=&\ln k-\frac d{2k}-\frac d{4k^2}+O(d/k^2).
	\end{eqnarray*}
It is easily verified that the last expression is strictly positive if
 $d\leq(2k-1)\ln k-2\ln 2$. 

To establish condition ii., fix a sequence $q_1,\ldots,q_l$ of non-negative integers.
Recall from \Thm~\ref{Thm_smallSubgraphConditioning} that $\Xi_j$ denotes the number of cycles of length $j$ in
$\Gnd$, with $1$-cycles being self-loops and $2$-cycles being multiple edges.
With $\delta_j,\lambda_j$ as in~(\ref{eqsmallSubgraphConditioning}),
we aim to show that
	\begin{equation}\label{eqCor_goodFirstMoment}
	\Erw\brk{\Zkg\cdot\prod_{j=1}^l(\Xi_j)_{q_j}}\sim\Erw\brk\Zkg\cdot\prod_{j=1}^l(\lambda_j(1+\delta_j))^{q_j},
	\end{equation}
There are two cases to consider.
\begin{description}
\item[Case 1: $q_1>0$.]
	If $\Xi_1=q_1>0$, then $\Zkg=0$ with certainty (because a self-loop is a monochromatic edge under any coloring).
	Moreover, as $\delta_1=-1$ we also have $\prod_{j=1}^l(\lambda_j(1+\delta_j))^{q_j}=0$.	
	Thus, (\ref{eqCor_goodFirstMoment}) is trivially satisfied in this case.
\item[Case 2: $q_1=0$.]
	By \Prop~\ref{Prop_KPGWfirstMoment}, 
		for any non-negative integers $p_2,\ldots,p_l$ we have
		\begin{equation}\label{eqCor_goodFirstMoment3}
		\Erw\brk{\Zkb\cdot\prod_{j=2}^l(\Xi_j)_{p_j}}\sim\Erw\brk\Zkb\cdot\prod_{j=2}^l(\lambda_j(1+\delta_j))^{p_j}.
		\end{equation}
	For a balanced map $\sigma:V\ra\brk k$ and let $\cE_\sigma$ be the event that $\sigma$ is a $k$-coloring of $\Gnd$.
	Summing over all balanced $\sigma$ and using the linearity of expectation, we obtain
		\begin{eqnarray}\label{eqWhatIsSymmetry}
		\Erw\brk{\Zkb\prod_{j=2}^l(\Xi_j)_{p_j}}=\sum_{\sigma}\Erw\brk{\prod_{j=2}^l(\Xi_j)_{p_j}\bigg|\cE_\sigma}\cdot\pr\brk{\cE_\sigma}.
		\end{eqnarray}
	Pick and fix one balanced map $\sigma_0:V\ra\brk k$ and let $\cE=\cE_{\sigma_0}$ for the sake of brevity.
	For symmetry reasons (namely, because $\prod_j(\Xi_j)_{p_j}$ is invariant under permutations of the vertices),
	we have
		$$\Erw\brk{\prod_{j=2}^l(\Xi_j)_{p_j}\bigg|\cE_\sigma}=\Erw\brk{\prod_{j=2}^l(\Xi_j)_{p_j}\bigg|\cE}.$$
	Thus, (\ref{eqWhatIsSymmetry}) gives
		$$\Erw\brk{\Zkb\prod_{j=2}^l(\Xi_j)_{p_j}}=\Erw\brk{\prod_{j=2}^l(\Xi_j)_{p_j}\bigg|\cE}\cdot\Erw[\Zkb].$$
	Hence, (\ref{eqCor_goodFirstMoment3}) yields
		\begin{eqnarray*}
		\Erw\brk{\prod_{j=2}^l(\Xi_j)_{p_j}\bigg|\cE}\sim\prod_{j=2}^l(\lambda_j(1+\delta_j))^{p_j}.
		\end{eqnarray*}
	Therefore, \Thm~\ref{Thm_Poisson} implies that given $\cE$,
		$(\Xi_2,\ldots,\Xi_l)$ are asymptotically independent $\Po(\lambda_j(1+\delta_j))$ variables.
	Consequently,  because we keep $q_2,\ldots,q_l$ fixed as $n\ra\infty$, we get
		$$\Erw\brk{\prod_{j=2}^l\Xi_j^{2q_j}\bigg|\cE}
			\sim
			\prod_{j=2}^l\Erw\brk{\Po(\lambda_j(1+\delta_j))^{2q_j}}=
			O(1).$$
	Thus, again by symmetry and the linearity of expectation,
		\begin{eqnarray}\label{eqCor_goodFirstMoment1}
		\Erw\brk{\Zkb\prod_{j=2}^l\Xi_j^{2q_j}}&=&\Erw\brk{\Zkb}\cdot\Erw\brk{\prod_{j=2}^l\Xi_j^{2q_j}\bigg|\cE}=O(\Erw\brk{\Zkb}).
		\end{eqnarray}
	Now, by  Cauchy-Schwarz
		\begin{eqnarray}
		\Erw\brk{(\Zkb-\Zkg)\prod_{j=2}^l(\Xi_j)_{q_j}}&\leq&
			\Erw\brk{\Zkb-\Zkg}^{\frac12}\cdot\Erw\brk{\Zkb\prod_{j=2}^l(\Xi_j)_{q_j}^2}^{\frac12}\nonumber\\
			&\leq&
			\Erw\brk{\Zkb-\Zkg}^{\frac12}\cdot\Erw\brk{\Zkb\prod_{j=2}^l\Xi_j^{2q_j}}^{\frac12}\nonumber\\
			&\stacksign{(\ref{eqCor_goodFirstMoment1})}{\leq}&\Erw\brk{\Zkb-\Zkg}^{\frac12}\cdot O(\Erw\brk{\Zkb})^{\frac12}\nonumber\\
			&=&o(\Erw\brk{\Zkb})\qquad\qquad\mbox{[by \Prop~\ref{Prop_goodFirstMoment}]}.
				\label{eqCor_goodFirstMoment2}
		\end{eqnarray}
	Finally, combining~(\ref{eqCor_goodFirstMoment3}) and~(\ref{eqCor_goodFirstMoment2}), we find
		\begin{eqnarray*}
		\Erw\brk{\Zkg\prod_{j=2}^l(\Xi_j)_{q_j}}&=&\Erw\brk{\Zkb\prod_{j=2}^l(\Xi_j)_{q_j}}+o(\Erw\brk{\Zkb})
			\sim\Erw\brk\Zkb\cdot\prod_{j=2}^l(\lambda_j(1+\delta_j))^{q_j}\\
			&\sim&\Erw\brk\Zkg\cdot\prod_{j=2}^l(\lambda_j(1+\delta_j))^{q_j}
					\qquad\mbox{[by \Prop~\ref{Prop_goodFirstMoment}].}
		\end{eqnarray*}
\end{description}
Thus, (\ref{eqCor_goodFirstMoment}) holds in either case.
\qed\end{proof}

After proving \Prop~\ref{Prop_goodFirstMoment} in \Sec~\ref{Sec_Prop_goodFirstMoment}, we are going to carry out the second moment argument in \Sec~\ref{Sec_second}.
This implies that the random variable $\Zkg$ also satisfies condition iii.\ in \Thm~\ref{Thm_smallSubgraphConditioning}.
Finally, in \Sec~\ref{Sec_Thm_main}, we are going to
apply \Thm~\ref{Thm_smallSubgraphConditioning} to
 complete the proof of the upper bound on $\chi(\gnd)$ claimed in \Thm~\ref{Thm_main}.

\section{The expected number of good colorings
	}\label{Sec_Prop_goodFirstMoment}

{\em
Throughout this section we assume that $k\geq k_0$ and $n\geq n_0$ are sufficiently big. We also continue to assume that 
$d$ satisfies~(\ref{eqSec_Prop_goodFirstMoment}).}


\subsection{Outline}

The aim in this section is to prove \Prop~\ref{Prop_goodFirstMoment}.
The proof is guided by the corresponding analysis for the $\gnm$ model performed in~\cite{ACOVilenchik}.
Indeed, several of the formulas that we arrive at ultimately are quite similar to the ones in~\cite{ACOVilenchik}.
However, arguing that these ideas/formulas carry over to the random regular graph turns out to be
a technically rather non-trivial task.

The proof is by way of a $d$-regular version of the ``planted coloring'' model.
To define this model, fix a balanced map $\sigma:V\ra\brk k$ and let $V_i=\sigma^{-1}(i)$.
Moreover, let $\mu=(\mu_{ij})_{i,j=1,\ldots,k}$ probability distribution on $\brk k\times\brk k$
such that $\mu_{ij}dn$ is integral for all $i,j$ satisfying
	\begin{equation}\label{eqplantedmu}
	\mu_{ii}=0\mbox{ for all }i\in\brk k
		\mbox{ and }
	\mu_{ij} =\mu_{ji}= \frac{1}{k(k-1)}+O(n^{-1/3})\quad\mbox{for all }1\leq i<j\leq k.
	\end{equation}
We let $\vec\Gamma_{\sigma,\mu}$ denote a configuration chosen uniformly at random subject to the condition that
	\begin{equation}\label{eqMuCondition}
	\abs{\cbc{(v,l)\in V_i\times\brk d:\Gamma(v,l)\in V_j\times\brk d}}=dn\mu_{ij}\qquad\mbox{for all }i,j\in\brk k.
	\end{equation}
In addition, we denote by $\cG(\sigma,\mu)$ the multi-graph obtained from $\vec\Gamma_{\sigma,\mu}$ by contracting the clones.
Then by construction, $\sigma$ is a ``planted'' $k$-coloring of $\cG(\sigma,\mu)$, and $e_{\cG(\sigma,\mu)}(V_i,V_j)=\mu_{ij}dn$ for all $1\leq i<j\leq k$.

We prove Proposition \ref{Prop_goodFirstMoment} in two steps: the first step consist in establishing

\begin{proposition}\label{pro_p_sep}
Let $\sigma:V\ra\brk k$ be balanced and
assume that $\mu$ satisfies~(\ref{eqplantedmu}).
Then $$\pr[\sigma \textrm{ is separable in } \cG(\sigma,\mu)] \geq 1-O(1/n).$$
\end{proposition}
We defer the proof of \Prop~\ref{pro_p_sep} to \Sec~\ref{Sec_pro_p_sep}.
Furthermore, in \Sec~\ref{Sec_clusterSize} we are going to prove

\begin{proposition}\label{cor_cluster_size}
Let $\sigma:V\ra\brk k$ be balanced and
assume that $\mu$ satisfies~(\ref{eqplantedmu}).
With probability $1-O(1/n)$ the random multi-graph $\cG(\sigma,\mu)$ is such that 
\begin{equation*}
\frac1n\ln |\cC(\sigma)| \leq \bc{\frac1k+\tilde{O}_k(k^{-2})}\ln 2.
\end{equation*}
\end{proposition}

\medskip\noindent{\em Proof of \Prop~\ref{Prop_goodFirstMoment} (assuming \Prop s~\ref{pro_p_sep} and~\ref{cor_cluster_size}).}
Let $\sigma:V\ra\brk k$ be balanced and let $M_\sigma$ be the set of all probability distributions $\mu$ that satisfy~(\ref{eqplantedmu}) such that $dn\mu_{ij}$ is integral for all $i,j$.
For any balanced $\sigma$ and for any $\mu$ we let $\Lambda_{\sigma,\mu}$ be the set of all 
$(n,d)$-configurations $\Gamma$ that satisfy~(\ref{eqMuCondition}).
In addition, let $\Lambda_{g,\sigma,\mu}$ be the set of all $(n,d)$-configurations $\Gamma\in\Lambda_{\sigma,\mu}$ such that $\sigma$ is a {\em good} $k$-coloring of
the multi-graph induced by $\Gamma$.
By Propositions~\ref{pro_p_sep} and~\ref{cor_cluster_size}, for any balanced $\sigma$ and for any
$\mu\in M_\sigma$ we have
\begin{equation}\label{equ_prob_good}
 \pr \left[\sigma \mbox{ is separable in } \cG(\sigma,\mu) \mbox{ and } \frac1n\ln|\cC(\sigma)| \leq (1/k+\tilde{O}_k(k^{-2}))\ln 2 \right] \sim 1. 
\end{equation}
Because
the ``planted'' configuration $\vec\Gamma_{\sigma,\mu}$ is nothing but a uniformly random element of $\Lambda_{\sigma,\mu}$, (\ref{equ_prob_good}) implies that
	\begin{equation}\label{equ_prob_good2}
	\abs{\Lambda_{g,\sigma,\mu}}\sim\abs{\Lambda_{\sigma,\mu}}
	\end{equation}
for any balanced $\sigma$ and any $\mu\in M_\sigma$.
Now, let
	$$\Lambda_\sigma=\bigcup_{\mu\in M_\sigma}\Lambda_{\sigma,\mu},\quad \Lambda_{g,\sigma}=\bigcup_{\mu\in M_\sigma}\Lambda_{g,\sigma,\mu}.$$
Then~(\ref{equ_prob_good2}) yields
	\begin{equation}\label{equ_prob_good3}
	\abs{\Lambda_{g,\sigma}}\sim\abs{\Lambda_{\sigma}}.
	\end{equation}

Summing over all balanced $\sigma$, we obtain from~(\ref{equ_prob_good3}) and the linearity of expectation
	\begin{eqnarray}\label{equ_prob_good4}
	 \Erw[\Zkg] &\geq&\sum_{\sigma}\frac{|\Lambda_{g,\sigma}|}{(dn-1)!!}\sim\sum_{\sigma}\frac{|\Lambda_{\sigma}|}{(dn-1)!!}.
	\end{eqnarray}
To relate~(\ref{equ_prob_good4}) to $\Erw[\Zkb]$, let 
 $\Lambda_{\sigma}'$ be the set of all configurations $\Gamma$ such that
$\sigma$ is a skewed $k$-coloring of the multi-graph induced by $\Gamma$.
Then
	\begin{eqnarray}\label{equ_prob_good5}
	 \Erw[\Zkb] &=&\sum_{\sigma}\frac{|\Lambda_{\sigma}\cup\Lambda_{\sigma}'|}{(dn-1)!!}
	 	\leq\sum_{\sigma}\frac{|\Lambda_{\sigma}|}{(dn-1)!!}+
			\sum_{\sigma}\frac{|\Lambda_{\sigma}'|}{(dn-1)!!}.
	\end{eqnarray}
Letting $\Zkb'$ denote the number of skewed balanced $k$-colorings of $\Gnd$, we obtain from \Cor~\ref{Cor_skewed}
	\begin{eqnarray}\label{equ_prob_good6}
	 \Erw[\Zkb'] &=&\sum_{\sigma}\frac{|\Lambda_{\sigma}'|}{(dn-1)!!}=o(\Erw[\Zkb]).
	\end{eqnarray}
Finally, combining~(\ref{equ_prob_good4})--(\ref{equ_prob_good6}), we see that $\Erw[\Zkg]\sim\Erw[\Zkb]$, as desired.
%
\qed

\subsection{Separability: proof of \Prop~\ref{pro_p_sep}}\label{Sec_pro_p_sep}

{\em Throughout this section, we let $\sigma:V\ra\brk k$ denote a balanced map. We let $V_i=\sigma^{-1}(i)$.
Moreover, $\mu$ denotes a probability distribution that satisfies~(\ref{eqplantedmu}) such that $dn\mu_{ij}$ is integral for all $i,j$.}

\bigskip\noindent
The proof of \Prop~\ref{pro_p_sep} proceeds in several steps, all of which depend on the
binomial approximation to the hypergeometric distribution from \Lem~\ref{lem_balls_bins_cap}.
We start by proving that \whp\ in the multi-graph $\cG(\sigma,\mu)$ with planted coloring $\sigma$
there is no other coloring $\tau$ such that the overlap matrix has an entry $\rho_{ii}(\sigma,\tau)\in(0.509,1 - k^{-0.499})$ \whp\

\begin{lemma}\label{lem_s_less}
Let $i \in [k]$ and let $0.509 \leq \alpha \leq 1 - k^{-0.499}$.
Then in $\cG(\sigma,\mu)$ with probability $1-\exp(-\Omega(n))$ the following is true.
\begin{equation}\label{prop_s_less}
\parbox{12cm}{
For any set $S \subset V_i$ of size $|S| = \alpha n/k$
the number of vertices $v \in V \setminus V_i$ that do not have a neighbor in $S$ is less than $(1-\alpha)n/k-n^{2/3}$.}
\end{equation}
\end{lemma}
\begin{proof}
Without loss of generality we may assume $i=1$.
Thus, let $S\subset V_1$ be a set of size $|S|=\alpha n/k$.
Let
	$$e_{j,S}=\abs{\cbc{(v,l)\in S\times\brk d:\vec\Gamma_{\sigma,\mu}(v,l)\in V_j\times\brk d}}$$
be the number of edges from $S$ to $V_j$ in $\cG(\sigma,\mu)$ for $j=2,\ldots,k$.
Since we are fixing the numbers $(\mu_{1j}dn)_{j=2,\ldots,k}$ of edges between $V_1$ and the other color classes,
we can think of $e_{j,S}$ as follows:
choose a  subset of $V_1\times\brk d$ of size $dn\mu_{1j}$ uniformly at random; then $e_{j,S}$ is the number of chosen elements that belong to $S\times\brk d$.
Thus, we are in the situation of \Lem~\ref{lem_balls_bins_cap}, which we are going to use to estimate $e_{j,S}$.
Hence, let $p_j=k\mu_{1j}$; then $p_j\sim(k-1)^{-1}$ by our assumption~(\ref{eqplantedmu}) on $\mu$.
Further, let $\hat e_{j,S}$ be a random variable with distribution $\Bin(|S| d,p_j)$.
Let $\delta=\ln^{-1/3}k$.
Then \Lem~\ref{lem_balls_bins_cap} yields
	\begin{eqnarray}\label{eqEstimateOutgoingEdges1}
	\pr\brk{e_{j,S}<\frac{(1-\delta)d|S|}{k-1}}&\leq&O(\sqrt n)\cdot\pr\brk{\hat e_{j,S}<\frac{(1-\delta)d|S|}{k-1}}.
	\end{eqnarray}
Further, by Lemma~\ref{Lemma_Chernoff} (to which we are going to refer as ``the Chernoff bound'' from now on),
	\begin{eqnarray}\label{eqEstimateOutgoingEdges2}
	\pr\brk{\hat e_{j,S}<\frac{(1-\delta)d|S|}{k-1}}&\leq&\exp\brk{-\frac{\delta^2 d|S|}{2(k-1)}}\leq\exp(-n/k).
	\end{eqnarray}
Since the total number of possible sets $S$ is bounded by $2^{n/k}$, (\ref{eqEstimateOutgoingEdges1}) and~(\ref{eqEstimateOutgoingEdges2}) yield
	\begin{eqnarray}\label{eqEstimateOutgoingEdges3}
	\pr\brk{\exists S,j:e_{j,S}<\frac{(1-\delta)d|S|}{k-1}}\leq (k-1)2^{n/k}\exp(-n/k)=\exp(-\Omega(n)).
	\end{eqnarray}
Thus, let $\cE_S$ be the event that $e_{j,S}\geq\frac{(1-\delta)d|S|}{k-1}$ for all $j=2,\ldots,k$.
Due to~(\ref{eqEstimateOutgoingEdges3}), we may condition on the event $\cE_S$ from now on.

Given the numbers $e_{j,S}$, the actual clones in $V_j\times\brk d$ that $\vec\Gamma_{\sigma,\mu}$ joins to $S\times\brk d$ are uniformly distributed.
Thus,  we can use \Lem~\ref{lem_balls_bins_cap} to estimate the number $X_{j,S}$ of vertices in $v\in V_j$ that $\vec\Gamma$ fails to join to $S$.
To this end, let $(b_v)_{v\in V_j}$ be a family of independent $\Bin(d,\frac{e_{j,S}}{d n/k})$ random variables. 
Let $$q_j=\pr\brk{b_v=0}\quad\mbox{for any $v\in V_j$, and }\quad\hat X_{j,S}=\Bin(n/k,q_j).$$
Then \Lem~\ref{lem_balls_bins_cap} yields
	\begin{eqnarray}\label{eqEstimateOutgoingEdges4}
	\pr\brk{X_{j,S}\geq t|\cE_S}\leq O(\sqrt n)\pr\brk{\hat X_{j,S}\geq t}\qquad\mbox{for any $t>0$}.
	\end{eqnarray}
Furthermore, since we are assuming that $e_{j,S}\geq(1-\delta)d|S|/(k-1)$, we find
	\begin{eqnarray}\label{eqEstimateOutgoingEdges5}
	q_j&=&\bc{1-\frac{e_{j,S}}{d n/k}}^d\leq\exp\brk{-\frac{e_{j,S}}{n/k}}\leq\exp\brk{-\frac{(1-\delta)\alpha d}{k-1}}\leq k^{-2\alpha(1-2\delta)}.
	\end{eqnarray}
Set $q=k^{-2\alpha(1-2\delta)}$, let $\hat X_S=\Bin((1-1/k)n,q)$, and let $X_S=\sum_{j=2}^k X_{j,S}$.
Then (\ref{eqEstimateOutgoingEdges4}) and~(\ref{eqEstimateOutgoingEdges5}) imply
	\begin{eqnarray}\label{eqEstimateOutgoingEdges6}
	\pr\brk{X_{S}\geq t|\cE_S}\leq O(\sqrt n)\pr\brk{\hat X_{S}\geq t}\qquad\mbox{for any $t>0$}.
	\end{eqnarray}

Thus, we are left to estimate the binomial random variable $\hat X_S$
%
%
%
%
with mean
	$\Erw[{\hat X_S}]=|V\setminus V_1|q
			\leq qn$.
By the Chernoff bound,
 \begin{eqnarray}\nonumber
 \pr\left[\hat X_S \geq (1-\alpha)n/k - n^{2/3}\right] &\leq& \exp\left[-(1-\alpha+o(1))\frac nk \cdot \ln\left(\frac{(1-\alpha)n/k}{\eul qn}\right)\right] \\
& \leq& \exp\left[-(1-\alpha+o(1))\frac nk \cdot \ln\left(\frac{1-\alpha}{\eul kq}\right)\right]. 
\label{equ_chern_bound_1a}
 \end{eqnarray}
 Combining~(\ref{eqEstimateOutgoingEdges6}) and~(\ref{equ_chern_bound_1a}), we see that
 \begin{eqnarray}
 \pr\brk{X_S \geq (1-\alpha)n/k - n^{2/3}\bigg|\cE_S} &\leq&\exp\brk{-(1-\alpha+o(1))\frac nk \cdot \ln\left(\frac{1-\alpha}{\eul kq}\right)}.
\label{equ_chern_bound_1}
 \end{eqnarray}
Furthermore, the number of ways to choose a $S\subset V_1$ of size $\alpha n/k$ is 
\begin{align}\label{equ_numb_s}
 \binom{n/k}{(1-\alpha)n/k} \leq \left(\frac{\eul}{1-\alpha}\right)^{(1-\alpha)\frac{n}{k}} = \exp\left[\frac{n}{k}(1-\alpha)(1-\ln(1-\alpha))\right].
\end{align}
Using (\ref{equ_chern_bound_1}), (\ref{equ_numb_s}) and the union bound, we obtain
\begin{eqnarray}\nonumber
 \pr\left[\exists S:X_S\geq (1-\alpha)n/k - n^{2/3},\cE_S\right]\\
 &		\hspace{-8cm}\leq&	\hspace{-4cm}
 \exp\left[\frac{(1-\alpha)n}{k} \cdot \left(1-\ln(1-\alpha)- \ln\left(\frac{1-\alpha}{2ekq}\right)\right)+o(n)\right].
 	\label{equ_chern_bound_1b}
\end{eqnarray}
We need to verify that the last term is $\exp(-\Omega(n))$. Thus, we need to estimate 
\begin{align} 	\label{equ_chern_bound_1c}
 1-\ln(1-\alpha)- \ln\left(\frac{1-\alpha}{2\eul kq}\right) = \ln\left(\frac{2e^2}{(1-\alpha)^2}k^{1-2\alpha+4\alpha\delta}\right).
\end{align}
This is negative iff
\begin{align} \label{equ_exp_bound}
 \exp\left[\left(\frac{1}{2}-\alpha+2\alpha\delta\right)\ln k\right] < \frac{1-\alpha}{\sqrt{2}e}.
\end{align}
By the convexity the exponential function, the l.h.s.\ and the linear function on the r.h.s.\ intersect at most twice.
Between these intersections the linear function is greater.
Moreover, it is easily verified that the 
r.h.s.\ of (\ref{equ_exp_bound}) is larger than the l.h.s.\ at both $\alpha = 0.509$ and $\alpha=1- k^{-0.499}$. Thus, (\ref{equ_exp_bound}) is true in the entire range
$0.509 < \alpha < 1- k^{-0.499}$.
Consequently, for such $\alpha$ the term (\ref{equ_chern_bound_1c}) is strictly negative, whence the r.h.s.\
of~(\ref{equ_chern_bound_1b}) is $\exp(-\Omega(n))$.
Thus, the assertion follows from~(\ref{eqEstimateOutgoingEdges3}).
\qed\end{proof}

To complete the proof of \Prop~\ref{pro_p_sep}, we also need to rule out the possibility that
$\cG(\sigma,\mu)$ has a coloring
$\tau$ such that $\rho_{ii}(\sigma,\tau)\in(1 - k^{-0.499},1-\kappa)$.
To this end, we are going to use an expansion argument.
This argument is based on establishing that in $\cG(\sigma,\mu)$ ``most'' vertices outside color class $V_i$
have a good number of neighbors in $V_i$ \whp\
More precisely, we have

\begin{lemma}\label{lem_neigbours_less_15}
 With probability $1-\exp(-\Omega(n))$ the random graph $\cG(\sigma,\mu)$ has the following property.
\begin{align}\label{prop_neigbours_less_15}
 \textrm{Let } i \in [k]. \textrm{ No more than } nk^{-2}\ln^{17}k \textrm{ vertices } v \notin V_i \textrm{ have less than } 15 \text{ neighbors in } V_i.
\end{align}
\end{lemma}
\begin{proof}
Assume without loss of generality that $i=1$.
We are going to use Lemma \ref{lem_balls_bins_cap} once more.
Our assumption~(\ref{eqplantedmu}) ensures that for each
$j\in\cbc{2,\ldots,k}$ the number of $V_1$-$V_j$ edges in $\cG(\sigma,\mu)$ is $\mu_{1j}\sim k^{-1}(k-1)^{-1}$.
Thus, let $(b_v)_{v\in V_j}$ be a family of independent random variables with distribution $\Bin(d,p_j)$, with $p_j=k\mu_{1j}\sim(k-1)^{-1}$.
Furthermore, let $X_j$ be the number of $v\in V_j$  with fewer than $15$ neighbors in $V_1$, and let
	$\hat X_j=\abs{\cbc{v\in V_j:b_v<15}}$.
Then by \Lem~\ref{lem_balls_bins_cap} we have	
	\begin{equation}\label{eqlem_neigbours_less_151}
	\pr\brk{X_j\geq t}\leq O(\sqrt n)\pr\brk{\hat X_j\geq t}\qquad\mbox{for any }t>0.
	\end{equation}
Furthermore,
because the random variables $b_v$, $v\in V_j$, are independent,
 $\hat X_j$ has a distribution $\Bin(n/k,q_j)$, with $q_j=\pr\brk{\Bin(d,p_j)< 15}$.

Now, let
$X=\sum_{j=2}^kX_j$ and let $\hat X$ be a random variable with distribution $\Bin((1-1/k)n,q)$, with $q=\max_{j\geq 2}q_j$.
Then~(\ref{eqlem_neigbours_less_151}) implies
	\begin{equation}\label{eqlem_neigbours_less_152}
	\pr\brk{X\geq t}\leq O(\sqrt n)\pr\brk{\hat X\geq t}\qquad\mbox{for any }t>0.
	\end{equation}
Furthermore, our assumption~(\ref{eqSec_Prop_goodFirstMoment})  on $d$ ensures that
	$$\Erw\brk{\hat X}\leq nq=n\pr\brk{\Bin\bc{d,\frac{1+o(1)}{k-1}}\leq 15}\leq O(k^{-2}\ln^{16}k) n.$$
Hence, $\pr[\hat X\geq nk^{-2}\ln^{17}k]\leq\exp(-\Omega(n))$ by the Chernoff bound.
Thus, the assertion follows from~(\ref{eqlem_neigbours_less_152}).
\qed\end{proof}

Given \Lem~\ref{lem_neigbours_less_15}, how do we argue that
\whp\ there is no $\tau$ such that $\rho_{ii}(\sigma,\tau)\in(1 - k^{-0.499},1-\kappa)$?
Such a coloring $\tau$ would have to give color $i$ to a good number of vertices from $V\setminus V_i$
with at least 15 neighbors in $V_i$ (because there is no sufficient supply of vertices that have less than $15$ neighbors in $V_i$).
However, we are going to show that assigning color $i$ to many such vertices ``displaces'' a very large number of vertices from $V_i$
due to expansion properties, and that it is therefore not possible that $\rho_{ij}(\sigma,\tau)\in(1 - k^{-0.499},1-\kappa)$ \whp\
To put this expansion argument together, we need the following upper bound on the probability that a specific set of edges occurs
in the random configuration $\vec\Gamma_{\sigma,\mu}$.

\begin{lemma}\label{lem_edges_cond}
Let $E$ be a set of edges of the complete graph on $V\times\brk d$ of size $|E|\leq\frac n{2k}$.
Let
	$$e_{ij}=\abs{\cbc{e\in E:e\cap (V_i\times d)\neq\emptyset\neq e\cap (V_j\times\brk d})}\qquad(i,j\in\brk k)$$
be the number of edges in $e\in E$ that join a $V_i$-clone with a $V_j$-clone and assume that $e_{ii}=0$ for all $i$.
%
Then
	$$\pr\brk{E\subset\vec\Gamma_{\sigma,\mu}}\leq\left(\frac{5}{dn}
			\right)^{|E|}.$$
\end{lemma}
\begin{proof}
Let $e_i=\sum_{j=1}^ke_{ij}$ and set $e=\sum_{i=1}^ke_i=2|E|$.
Let $m_{ij}=dn\mu_{ij}$ for all $i,j\in\brk k$.
We claim that
\begin{equation}\label{eqlem_edges_cond1}
 \pr[E\subset\vec\Gamma_{\sigma,\mu}]=
 		\frac{\left[\prod_{i=1}^k \binom{dn/k-e_i}{(m_{ij}-e_{ij})_{j\in\brk k}}   \right]    \left[\prod_{1\leq i<j\leq k}(m_{ij}-e_{ij})!\right]     }
				{    \left[\prod_{i=1}^k \binom{dn/k}{(m_{ij})_{j\in\brk k}}\right]    \left[\prod_{1\leq i<j\leq k}m_{ij}!\right] } .
\end{equation}
Indeed, the numerator is obtained by (fixing the edges in $E$ and) counting the number of ways to match the remaining clones, given $\mu$.
More precisely, for any fixed $i\in\brk k$ the corresponding factor in the first product counts the number of ways to choose the $m_{ij}-e_{ij}$ clones
that are going to be matched with clones from color class $j$.
Moreover, for fixed $i, j$ the corresponding factor in the second product counts the number of matchings between the clones thus designated.
The denominator simply is the number of configurations respecting $\sigma,\mu$.

Because $m_{ij}=m_{ji}$ by assumption and $e_{ij}=e_{ji}$ by definition, (\ref{eqlem_edges_cond1}) yields
\begin{eqnarray*}
\pr[E\subset\vec\Gamma_{\sigma,\mu}]
&=&\frac{\left[\prod_{i=1}^k \binom{dn/k-e_i}{(m_{ij}-e_{ij})_j}\right]  \left[\prod_{i,j=1}^k(m_{ij}-e_{ij})!\right]^{1/2} }
	{\left[\prod_{i=1}^k \binom{dn/k}{(m_{ij})_j}\right]  \left[\prod_{i,j=1}^km_{ij}!\right]^{1/2} } 
=\left[\prod_{i=1}^k \frac{1}{(dn/k)_{e_{i}}}\right]  \left[\prod_{i,j=1}^km_{ij}!\right]^{1/2}.
\end{eqnarray*}
Furthermore, because of the assumption $|E| \leq \frac{n}{2k}$
we have $$(dn/k-e_i)! \geq \bcfr{dn/k}{2}^{e_i} = \bcfr{dn}{2k}^{e_i}.$$
Finally, recalling from~(\ref{eqplantedmu}) that  $|\mu_{ij}-k^{-1}(k-1)^{-1}|\leq0.01k^{-2}$ 
for all $i,j\in\brk k$ and
using Stirlings formula, we get  
\begin{align*}
 \pr[E\subset\vec\Gamma_{\sigma,\mu}] &\leq \left[ \prod_{i=1}^k
 	 2^{e_i} \cdot \left( \frac{k}{dn}\right)^{e_i} \right] \left[ \prod_{i,j=1}^k\left(\frac{1.01dn}{k(k-1)})\right)^{e_{ij}/2} \right] \\
&= \left[  \prod_{i=1}^k \left(\frac{4k}{k-1}\right)^{e_i/2}\left(\frac{1}{dn}\right)^{e_i} \right] \left[ \prod_{i,j=1}^k\left(1.01dn\right)^{e_{ij}/2} \right] 
	\leq\left(\frac{5}{dn}\right)^{e/2},
\end{align*}
as claimed.
\qed\end{proof}

\begin{remark}\label{rem_edges_cond}
Even though in this section we are assuming that $\mu_{ij}\sim k^{-1}(k-1)^{-1}$ for all $1\leq i<j\leq k$,
proof of \Lem~\ref{lem_edges_cond} only requires that, say, $|\mu_{ij}-k^{-1}(k-1)^{-1}|\leq0.01k^{-2}$.
Moreover, the same proof also goes through if we merely assume that, say, $|\sigma^{-1}(i)-n/k|\leq0.01 n/k$ for all $i\in\brk k$
rather than that $\sigma$ is balanced.
This observation will be needed in \Sec~\ref{sec:prop:NoOfCompl+ItsDistr}.
\end{remark}

Using \Lem~\ref{lem_edges_cond}, we can now prove that \whp\ the random graph $\cG(\sigma,\mu)$ 
does not feature a ``small dense set'' of vertices (i.e., a small set of vertices that spans a much larger number of edges than expected).
This will be the key ingredient to our expansion argument.

\begin{corollary}\label{lem_edges_5s}
With probability $1- O(1/n)$ the random graph $\cG(\sigma,\mu)$ has the following property:
\begin{equation} \label{prop_edges_5s}
\parbox[c]{12cm}{For any set $S \subset V$ of size $|S| \leq k^{-4/3}n$ the number of edges spanned by $S$ in $\cG(\sigma,\mu)$ is
 	bounded by $ 5|S|.$}
\end{equation}
\end{corollary}
\begin{proof}
Let $0\leq \alpha \leq k^{-4/3}$ and fix a set $S$ of size $s=|S|$ with $1\leq s\leq k^{-4/3}n$.
Furthermore, let $Y(S)$ be the number of edges in $\vec\Gamma_{\sigma,\mu}$ that join two clones in $S\times\brk d$.

We are going to use the union bound to estimate $Y(S)$.
Let $E$ be a set of $|E|=5s$ unordered pairs of clones in $S\times\brk d$.
Let $e_{ij}=\abs{\cbc{\cbc{x,y}\in E:\sigma(x)=i,\,\sigma(y)=j}}$.
Clearly, if $e_{ii}>0$ for some $i\in\brk k$, then $E\not\subset\vec\Gamma_{\sigma,\mu}$ (because $\vec\Gamma_{\sigma,\mu}$ respects $\sigma$).
Thus, assume that $e_{ii}=0$ for all $i\in\brk k$.
Then \Lem~\ref{lem_edges_cond} implies
	\begin{eqnarray}\label{eqlem_edges_5s1}
	\pr\brk{E\subset\vec\Gamma_{\sigma,\mu}}&\leq&\left(\frac{5}{dn}\right)^{5s}. 
	\end{eqnarray}
By the union bound and~(\ref{eqlem_edges_5s1}),
	\begin{eqnarray}
	\pr\brk{Y(S)\geq5s}&\leq&\pr\brk{\exists E\mbox{ as above}:E\subset\vec\Gamma_{\sigma,\mu}}\leq\bink{\bink{ds}{2}}{5s}
		\bcfr{5}{dn}^{5s}
		\leq\bc{\eul ds/n}^{5s}.\label{eqlem_edges_5s2}
	\end{eqnarray}

Using the union bound and~(\ref{eqlem_edges_5s2}), we find
	\begin{eqnarray}\nonumber
	\pr\brk{\exists S\subset V,|S|=s:Y(S)>5s}&\leq&
		\bink{n}{s}\bc{\eul s d/n}^{5s}\leq\brk{\frac{\eul n}{s}\cdot\bc{\eul sd/n}^{5}}^{s}\\
		&\leq&\brk{\exp(6)(s/n)^4 d^5}^{s}.\label{eqlem_edges_5s3}
	\end{eqnarray}
Finally, summing~(\ref{eqlem_edges_5s3}) up, we find
	\begin{eqnarray*}
	\pr\brk{\exists S\subset V,|S|\leq k^{-4/3}n:Y(S)>5s}&\leq&\sum_{1\leq s\leq k^{-4/3}n}\brk{\exp(6)(s/n)^4 d^5}^{s}=O(1/n),
	\end{eqnarray*}
as desired.
\qed\end{proof}

\medskip\noindent{\em Proof of \Prop~\ref{pro_p_sep}.}
 We need to show that the following holds w.h.p.
\begin{quote}
Let $\tau$ be a balanced $k$-coloring of $\cG(\sigma)$ and let $i \in [k]$ be such that $ \tau(v) = i$ for at least
$0.51n/k$ vertices $v \in V_i$. Then $|\left\{v \in V_i:\tau(v) = i \right\}| \geq \frac{n}{k}(1 - \kappa)$.
\end{quote}
By Lemmas \ref{lem_s_less}, \ref{lem_neigbours_less_15} and \ref{lem_edges_5s}, we may assume
that (\ref{prop_s_less}), (\ref{prop_neigbours_less_15}) and (\ref{prop_edges_5s}) hold.
Moreover, without loss of generality we may assume that $i=1$.

Let $\tau$ be a  balanced $k$-coloring and let $S=\tau^{-1}(1) \cap V_1$.
Assume that
\begin{align}\label{equ_0.49}
 0.51n/k \leq |S| \leq (1-k^{-0.49})n/k.
\end{align}
Let $T = \tau^{-1}(1)\setminus V_1$. 
Then $S\cup T = \tau^{-1}(1)$ is an independent set.
In particular, none of the vertices in $T$ has a neighbor in $S$.
Moreover, $|T| \geq n/k - |S|$ because $\tau$ is a balanced coloring.
But then~(\ref{equ_0.49}) contradicts~(\ref{prop_s_less}).
Thus, we know that $|S|>(1-k^{-0.49})n/k$.

Let $Q$ be the set of all vertices $v \in \tau^{-1}(1)\setminus V_1$ with at least $15$ neighbors in $V_1$.
Moreover, let $R= V_1\setminus \tau^{-1}(1)$.
Because both $\sigma$ and $\tau$ are balanced, we have
	\begin{equation}\label{eqRcupQ}
	|R\cup Q|\leq 2\brk{\frac nk-|S|}\leq 2nk^{-1.49}<k^{-4/3}n.
	\end{equation}
The set $R$ contains all the neighbors that the vertices in $Q$ have in $V_1$ (because $\tau^{-1}(1)$ is an independent set).
Hence, by the definition of $Q$, the number $E$ of edges spanned by $R\cup Q$ in $\cG(\sigma,\mu)$ is at least $E\geq 15|Q|$.
Hence, (\ref{prop_edges_5s}) and~(\ref{eqRcupQ}) yield
\begin{equation}\label{equ_1}
 15|Q|\leq E\leq 5|R\cup Q|, \mbox{\ \ \ whence } |Q|\leq|R|/2.
\end{equation}
Let $W=\tau^{-1}(1)\setminus (Q\cup V_1)$ be the set of all vertices with color $1$ under $\tau$
and another color under $\sigma$ that have fewer than $15$ neighbors in $V_1$.
Once more because $\sigma$ and $\tau$ are balanced, we get
\begin{equation*}
 |S| + |R| = n/k = |S|+|Q|+|W|
\end{equation*}
Thus, (\ref{equ_1}) yields
\begin{equation*}
 |R|=|Q|+|W| \leq |R|/2+|W|.
\end{equation*}
Hence, (\ref{prop_neigbours_less_15}) implies that $|R| \leq 2|W| \leq 2nk^{-2}\ln^{17}k\leq n\kappa/k$.
Finally, because $\tau$ is balanced this entails that
	$|\tau^{-1}(1)\cap V_1|=\frac nk-|R|\geq\frac nk(1-\kappa)$, as desired.
\qed

\subsection{Upper-bounding the cluster size: proof of \Prop~\ref{cor_cluster_size}}\label{Sec_clusterSize}

The goal in this section is to establish the bound on the cluster size $\abs{\cC(\sigma)}$ in the random graph
$\cG(\sigma,\mu)$, where we continue to assume that $\sigma$ is balanced and that $\mu$ satisfies~(\ref{eqplantedmu}).
The following definition provides the key concepts.

\begin{definition}\label{Def_core}
Let $\ell>0$ be an integer.
\begin{enumerate}
 \item The \bemph{$(\sigma,\ell)$-core} of $\cG(\sigma,\mu)$ is the largest induced subgraph $(V',E')$ 
 	such that for all $v \in V'$ and all $i \neq \sigma(v)$ we have
$\abs{e_{\cG(\sigma,\mu)}(v,V'\cap\sigma^{-1}(i))}\geq\ell.$
\item Let $V'$ be the $(\sigma,\ell)$-core and let $a\geq0$ be an integer. A vertex $u \in V$ is \bemph{$a$-free} if
$$
\abs{\{i \in [k]: e_{\cG(\sigma,\mu)}(u,V'\cap \sigma^{-1}(i))=0\}} \geq a+1.
$$
\item A vertex that fails to be $1$-free is \bemph{complete}.
\end{enumerate}
\end{definition}

In words, the $(\sigma,\ell)$-core of $\cG(\sigma,\mu)$ is the largest subgraph $V'$ such that every vertex $v\in V'$
has at least $\ell$ edges into every other color class except its own.
Furthermore, a vertex $v$ is $a$-free if there are $a$ color classes in addition to its own such that $v$ fails to have a neighbor
in that color class that belongs to the $(\sigma,\ell)$-core.
Finally, a vertex is complete if in every other color class but its own it has a neighbor that belongs to the core.
For the sake of concreteness, we let $\ell=100$  in the following.

The proof strategy is as follows.
As a first step, we show that \whp\ the random multi-graph $\cG(\sigma,\mu)$ has a huge $(\sigma,\ell)$-core.
More precisely, in \Sec~\ref{Sec_Prop_100-Core} we will establish

\begin{proposition}\label{Prop_100-Core}
 With probability $1-O(1/n)$, $\cG(\sigma,\mu)$ has a $(\sigma,100)$-core containing all but $\tilde{O}_k(k^{-1})n$ vertices.
\end{proposition}

Based on this estimate, we can bound the number of $1$-free and $2$-free vertices.
Indeed, in \Sec~\ref{Sec_pro_free_1_2} we are going to prove

\begin{proposition}\label{pro_free_1_2}
 With probability $1-O(1/n)$ the random graph $\cG(\sigma,\mu)$ has the following properties.
\begin{enumerate}
 \item At most $\frac{n}{k}(1+\tilde O_k(1/k))$ vertices are $1$-free.
 \item At most $O_k(k^{-2})n$ vertices are $2$-free.
\end{enumerate}
\end{proposition}

As a next step, we observe that, due to the expansion properties of $\cG(\sigma,\mu)$, the colors of all the complete vertices are ``frozen'' in $\cC(\sigma)$.
More specifically, \whp\ there does not exist a coloring $\tau$ in the cluster $\cC(\sigma)$ that assigns a complete vertex a different color than $\sigma$ does.

\begin{lemma}\label{lem_core_sim}
 With probability $1-O(1/n)$ the random graph $\cG(\sigma,\mu)$ has the following property.
\begin{align} \label{prop_core_sim}
 \mbox{For all complete $v$ and all } \tau \in \cC(\sigma) \mbox{ we have } \sigma(v) = \tau(v).
\end{align}
\end{lemma}
\begin{proof}
By Proposition \ref{pro_p_sep} we may assume that $\sigma$ is separable in $\cG(\sigma,\mu)$ and
by Lemma \ref{lem_edges_5s} we may assume that $\cG(\sigma,\mu)$ has the property (\ref{prop_edges_5s}).
Let $V'$ be the $(\sigma,\ell)$-core.
Moreover, let $\tau\in\cC(\sigma)$ and let
\begin{align*}
 &\Delta_i^+ = \left\{v\in V': \tau(v) = i \neq \sigma(v) \right\}\\
 &\Delta_i^- = \left\{v\in V': \tau(v) \neq i = \sigma(v) \right\}
\end{align*}
In words, $\Delta_i^+$ are the vertices that take color $i$ under $\tau$ and a different color under $\sigma$, and $\Delta_i^-$ are the
vertices that receive color $i$ under $\sigma$ and a different color under $\tau$.
Clearly,
\begin{equation}\label{equ_delta_sum}
 \sum_{i=1}^k |\Delta_i^+|=|\left\{v \in V':\sigma(v)\neq \tau(v)\right\}|=\sum_{i=1}^k |\Delta_i^-|.
\end{equation}
Moreover, because $\sigma$ is separable and as both $\sigma$, $\tau$ are balanced, we have
\begin{equation}\label{equ_delta_max}
 \max_{i \in [k]} |\Delta_i^+| \leq \frac{\kappa \cdot n}{k} \ \ \mbox{\ \ \ and } \ \ \max_{i \in [k]} |\Delta_i^-| \leq \frac{\kappa \cdot n}{k}.
\end{equation}
We are going to show that
	\begin{equation}\label{eqProofComplete}
	\cbc{v \in V':\sigma(v) \neq \tau(v)} = \emptyset.
	\end{equation}
This implies that indeed $\sigma(v)=\tau(v)$ for all complete vertices $v$,
because in order to change the color of a complete vertex it is necessary to change the color of a vertex in the core $V'$ as well.

To establish~(\ref{eqProofComplete}) 
let $S_i = \Delta_i^+ \cup \Delta_i^-$ for $i \in [k]$.
Then (\ref{equ_delta_max}) implies that $|S_i| \leq k^{-3/2}n$ for all $i$.
Furthermore, (\ref{prop_edges_5s}) implies that none of the set $S_i$ spans more than $5|S_i|$ edges.
Because $\tau$ is a $k$-coloring, all the neighbors of $v\in\Delta_i^+$ in $V'$ that take color $i$ under $\sigma$ must belong to  $\Delta_i^-$.
Since any $v\in \Delta_i^+\subset V'$ has at least $100$ neighbors in $V'\cap\sigma^{-1}(i)$, we thus obtain
\begin{equation*}
 100|\Delta_i^+| \leq 
 	5|S_i| \leq 5(|\Delta_i^+|+ |\Delta_i^-|).
\end{equation*}
Consequently, $|\Delta_i^-| \geq2|\Delta_i^+|$ for all $i$.
Therefore, (\ref{equ_delta_sum}) yields $\Delta_i^+ = \Delta_i^-=0$ for all $i$, whence~(\ref{eqProofComplete}) follows.
\qed\end{proof}

%
%
%
%

\medskip\noindent\emph{Proof of \Prop~\ref{cor_cluster_size} (assuming \Prop s~\ref{Prop_100-Core} and~\ref{pro_free_1_2}).}
By Lemma \ref{lem_core_sim} we may assume that (\ref{prop_core_sim}) holds. Let $F_a$ be the set of all $a$-free vertices.
If a vertex $v$ is $1$-free but not $2$-free, then~(\ref{prop_core_sim}) implies that there is a set $C_v\subset\brk k$ of size two such that
	$$\tau(v)\in C_v\qquad\mbox{for all }\tau\in\cC(\sigma).$$
Hence,
	\begin{equation}\label{eqcor_cluster_size1}
	\abs{\cC(\sigma)}\leq 2^{|F_1\setminus F_2|}\cdot k^{|F_2|}.
	\end{equation}
Thus, the assertion follows by comparing the bounds on $|F_1|,|F_2|$ provided by \Prop~\ref{pro_free_1_2}
with the estimate of $\Erw\brk{\Zkb}$ from \Prop~\ref{Prop_KPGWfirstMoment}.
Indeed, \Prop~\ref{pro_free_1_2} and~(\ref{eqcor_cluster_size1}) imply that with probability $1-O(1/n)$ we have
	\begin{eqnarray}\label{eqcor_cluster_size2}
	\frac1n\ln\abs{\cC(\sigma)}&\leq& \frac{|F_1\setminus F_2|}n\ln 2+\frac{|F_2|}n\ln k=
		\frac{\ln 2}k+\tilde O_k(k^{-2}).
	\end{eqnarray}
By comparison, \Prop~\ref{Prop_KPGWfirstMoment} yields
	\begin{eqnarray}\nonumber
	\frac1n\ln\Erw\brk{\Zkb}&=&\ln k+\frac d2\ln(1-1/k)\\
		&=&\ln k-\frac d2\bc{\frac1k+\frac1{2k}}+\tilde O_k(k^{-2})\quad\mbox{[as $\ln(1+z)=z+z^2/2+O(z^3)$, $d\leq 2k\ln k$]}\nonumber\\
		&=&\frac{c}{2k}+\tilde O_k(k^{-2})\qquad\qquad\qquad\qquad\ \ \ \mbox{[as $d=(2k-1)\ln k-c$]}.\label{eqcor_cluster_size3}
	\end{eqnarray}
Comparing~(\ref{eqcor_cluster_size2}) and~(\ref{eqcor_cluster_size3}), we see that
indeed $\frac1n\ln\Erw\brk{\Zkb}$ is strictly greater than $\frac1n\ln\abs{\cC(\sigma)}$ if $c\geq2\ln 2-\eps_k$ with, say, $\eps_k=\Theta_k(k^{-0.9})$.
\qed

\subsection{Proof of \Prop~\ref{Prop_100-Core}}\label{Sec_Prop_100-Core}

The ``canonical'' way of constructing the core is by iteratively evicting vertices that violate the core condition from \Def~\ref{Def_core},
i.e., that have too small a number of neighbors in some color class other than their own inside the core.
In principle, this process could be studied accurately via, e.g., the differential equations method.
However, there is a technically far simpler way to obtain the estimate promised in \Prop~\ref{Prop_100-Core}.
Roughly speaking, the simpler argument is based on the observation that, due to the expansion properties of $\cG(\sigma,\mu)$, the
core ``almost'' contains the set of vertices that have at least $3\ell$ neighbors in each color class other than their own
in the {\em entire} graph $\cG(\sigma,\mu)$.
The size of this set of vertices can be estimated fairly easily.

More precisely, to estimate the size of the core we introduce a few vertex sets.
Ultimately, the idea is to define a big subset of the core whose size can be estimated (relatively) easily.
Recall that we set $\ell=100$ and let $V_i=\sigma^{-1}(i)$.
%
First, we consider the sets
	$$W_{ij} = \left\{v \in V_i:e_{\cG(\sigma,\mu)}(v,V_j) < 3\ell \mbox{ and } e_{\cG(\sigma,\mu)}(v,V_h)<2\ell\ln k \mbox{ for all } h \in [k] \right\}\quad(i,j\in\brk k,\,i\neq j).$$
In words, $W_{ij}$ contains all vertices $v$ of color $i$ that have ``only'' $3\ell$ edges towards color class $j$,
while there is no color class $h$ where $v$ has more than $2\ell\ln k$ neighbors.
This definition is motivated by the observation that, because $\sigma$ is balanced and $d=(2+o_k(1))k\ln k$, the {\em expected} number of neighbors
that a vertex $v\in V_i$ has in some other color class $V_j$ is about $2\ln k$.
Hence, we expect that for $k$ sufficiently large only very few vertices either satisfy $e_{\cG(\sigma,\mu)}(v,V_j) < 3\ell$ or
$e_{\cG(\sigma,\mu)}(v,V_h)\geq2\ell\ln k$ for $i\neq j,h$.
Thus, we expect $W_{ij}$ to be ``small''.
In addition, we let
	\begin{equation}\label{eqW}
	W_{ii}=\emptyset,W_i=\bigcup_{j=1}^k W_{ij} \mbox{ for all $i\in\brk k$, and }W=\bigcup_{i=1}^k W_i .
	\end{equation}

Furthermore, for $i,j\in\brk k$, $i\neq j$ we let 
	\begin{eqnarray*}
	U_{ij}&=&\left\{v \in V_i\setminus W: e_{\cG(\sigma,\mu)}(v,W_j) > \ell \right\}\quad\mbox{ and }U=\bigcup_{ij} U_{ij},\\
	U'_{ij}&=&\left\{v \in V_i\setminus W: e_{\cG(\sigma,\mu)}(v,V_j) > 2\ell\ln k \right\}\quad\mbox{ and }U'=\bigcup_{ij} U'_{ij}.
	\end{eqnarray*}
Thus, $U_{ij}$ contains those vertices $v\in V_i$ that have ``a lot'' of neighbors in the ``bad'' set $W_{j}$.
Because the sets $W_j$ are small, the expansion properties of $\cG(\sigma,\mu)$ will imply that
	the set $U$ is tiny.
Moreover, $U'$ consists of vertices that have much more neighbors than the expected $2\ln k$ in one of the color classes.
The set $U'$ will turn out to be tiny as well, because the numbers $e_{\cG(\sigma,\mu)}(v,V_j)$ will emerge to be
somewhat concentrated about their expectations.

Finally, we define a sequence of sets $Y^{\bc t}$, $t\geq0$.
We let $Y^{(0)} = U \cup U'$.
For $t\geq1$, we define $Y^{\bc t}$ as follows:
	\begin{quote}
	If there exists a vertex $v \in V\setminus Y^{(t)}$ that has more than $\ell$ neighbors in $Y^{(t-1)}$, 
	then let $v_t$ be the smallest such vertex and let $Y^{(t)}=Y^{(t-1)} \cup \left\{v\right\}$.
	If there is no such vertex $v$, then let $Y^{\bc{t}}=Y^{\bc{t-1}}$.
	\end{quote}
Let
	\begin{equation}\label{eqY}
	Y=\bigcup_{t\geq 0}Y^{\bc t}.
	\end{equation}
With this construction in place, we have

\begin{proposition}\label{Prop_core_constr}
The set $V\setminus(W\cup Y)$ is contained in the $\ell$-core of $\cG(\sigma,\mu)$.
\end{proposition}
\begin{proof}
Let $V''=V\setminus(W\cup Y)$.
To show that $V''$ is contained in the $\ell$-core of $\cG(\sigma,\mu)$, it suffices to verify
that every vertex $v\in V''$ has at least $\ell$ edges into $V''\cap V_j$ for any $j\neq\sigma(v)$.
Indeed, because $v\not\in W$ we know that $e_{\cG(\sigma,\mu)}(v,V_j)\geq3\ell$.
Furthermore, as $v\not\in U\subset Y$, we have $e_{\cG(\sigma,\mu)}(v,W)\leq\ell$.
Finally, the construction of $Y$ ensures that $e_{\cG(\sigma,\mu)}(v,Y)\leq\ell$.
Hence,
	$$e_{\cG(\sigma,\mu)}(v,V''\cap V_j)\geq e_{\cG(\sigma,\mu)}(v,V_j)-e_{\cG(\sigma,\mu)}(v,W)-e_{\cG(\sigma,\mu)}(v,Y)\geq\ell,$$
as desired.
\qed\end{proof}

%



Thus, to complete the proof of \Prop~\ref{Prop_100-Core}, we are left to estimate the sizes
of the sets $W$, $U$, $U'$, $Y$.
These estimates are based on the approximation to the hypergeometric distribution from \Lem~\ref{lem_balls_bins_cap}.

\begin{lemma}\label{lem_w}\label{lem_u'}
With probability $1- \exp(-\Omega(n))$,  we have 
	$$W_{ij}\leq n\tilde O_k(k^{-3})\quad\mbox{ for all $i,j\in\brk k$}.$$
Hence, $|W_i| \leq n\cdot\tilde O_k(k^{-2})$ for all $i\in\brk k$  and $|W| \leq n\cdot\tilde O_k(k^{-1})$.
Furthermore, with probability $1- \exp(-\Omega(n))$ we have $|U'|\leq k^{-100}n$.
\end{lemma}
\begin{proof}
Fix two indices $i,j\in\brk k$, $i\neq j$, and let 
	$$W_{ij}' = \left\{v \in V_i:e_{\cG(\sigma,\mu)}(v,V_j) < 3\ell\right\}.$$
Since we are fixing the number $dn\mu_{ij}$ of $V_i$-$V_j$ edges, the set of clones in $V_i\times\brk d$
that $\vec\Gamma_{\sigma,\mu}$ matches to the set $V_j\times\brk d$ is a uniformly random set of size $dn\mu_{ij}$.
Hence, \Lem~\ref{lem_balls_bins_cap} applies.
Thus, let $(b_v)_{v\in V_i}$ be a family of independent $\Bin(d,p)$ variables, with $p=k\mu_{ij}\sim(k-1)^{-1}$.
Let $\hat W_{ij}=\abs{\cbc{v\in V_i: b_v<3\ell}}$.
Then \Lem~\ref{lem_balls_bins_cap} yields
	\begin{equation}\label{eqlem_w1}
	\pr\brk{|W_{ij}'|\geq t}\leq O(\sqrt n)\cdot\pr\brk{\hat W_{ij}\geq t}\qquad\mbox{for any }t\geq 0.
	\end{equation}
Furthermore, because the random variables $b_v$ are mutually independent, $\hat W_{ij}$ has distribution $\Bin(n/k,q)$, with $q=\pr\brk{\Bin(d,p)<3\ell}$.
Since $p\sim(k-1)^{-1}$, our assumption~(\ref{eqSec_Prop_goodFirstMoment}) on $d$ implies that
	$q\leq k^{-2}\ln^{3\ell}k.$
Therefore, by the Chernoff bound
	\begin{equation}\label{eqlem_w2}
	\pr\brk{\hat W_{ij}\geq nk^{-2}\ln^{3\ell+1}k=n \tilde O(k^{-3})}\leq\exp(-\Omega(n)).
	\end{equation}

Further, let $W_{ij}''=\abs{\cbc{v \in V_i:e_{\cG(\sigma,\mu)}(v,V_j) > 2\ell\ln k}}$.
To estimate the size of this set, we consider $\tilde W_{ij}=\abs{\cbc{v\in V_i: b_v>2\ell\ln k}}$.
Applying \Lem~\ref{lem_balls_bins_cap} once more, we see that
	\begin{equation}\label{eqlem_w3}
	\pr\brk{|W_{ij}''|\geq t}\leq O(\sqrt n)\cdot\pr\brk{\tilde W_{ij}\geq t}\qquad\mbox{for any }t\geq 0.
	\end{equation}
Due to the independence of the $b_v$, $\tilde W_{ij}$ has distribution $\Bin(n_i,\tilde q)$, where $\tilde q=\pr\brk{\Bin(d,p)>2\ell\ln k}$.
Since $p\sim(k-1)^{-1}$, we have $dp\leq 3\ln k$.
Hence, by the Chernoff bound
	$$\tilde q\leq\exp(-2\ell\ln k)\leq k^{-200}.$$
Consequently, invoking the Chernoff bound once more, we find
	\begin{equation}\label{eqlem_w4}
	\pr\brk{\tilde W_{ij}\geq n k^{-200}}\leq \exp(-\Omega(n)).
	\end{equation}

Finally,
	$$W_{i}\subset \bigcup_{j=1}^k W_{ij}'\cup W_{ij}''.$$
Hence, combining~(\ref{eqlem_w1})--(\ref{eqlem_w4}), we see that with probability $1-\exp(-\Omega(n))$ we have
	$|W_{i}|\leq\tilde O_k(k^{-2})n$.
Furthermore, 
	$$U'\subset\bigcup_{i,j=1}^k W_{ij}''.$$
Hence, (\ref{eqlem_w3})--(\ref{eqlem_w4}) show that $|U'|\leq k^{-100}n$ (with room to spare) with probability $1-\exp(-\Omega(n))$.
\qed\end{proof}

\begin{lemma}\label{lem_u}
 With probability at least $1-\exp(-\Omega(n))$ we have $|U| \leq nk^{-30}$.
\end{lemma}
\begin{proof}
For $i,j\in\brk k$, $i\neq j$ let
\begin{equation}
 U_{ij}^* = \left\{v \in V_i : e_{\cG(\sigma,\mu)}(v,W_j
		) \geq \frac{\ell}{2}\right\} 
		\supset U_{ij}.
\end{equation}
We are going to bound $| U_{ij}^*|$.
By construction, for all $v\in W_{j}$ we have $e_{\cG(\sigma,\mu)}(v,V_i)\leq 2\ell\ln k$.
Moreover, by \Lem~\ref{lem_w} we may assume that $|W_{j}|=\tilde O_k(k^{-2})n$.
Hence, the number $\eta_{ji}$ of $V_i\times\brk d$-$W_{j}\times\brk d$ edges in $\vec\Gamma_{\sigma,\mu}$ satisfies
	 $\eta_{ji}=\tilde O_k(k^{-2})n$.
Given $\eta_{ji}$, the actual {\em set} of clones in $V_i\times\brk d$ that $\vec\Gamma_{\sigma,\mu}$ connected
with $W_{j}\times\brk d$ is a uniformly random set.
This is because the definition of the set $W_{j}$ is just in terms of the {\em numbers} $e(v,V_h)$ of edges from $v\in V_j$ to $V_h$ for $h\neq j$
in the contracted multi-graph $\cG(\sigma,\mu)$.

Thus, we are in the situation described in \Lem~\ref{lem_balls_bins_cap}.
Hence, consider a family $(b_v)_{v\in V_i}$ of mutually independent random variables with distribution $\Bin(d,p)$ with $p=\frac{\eta_{ji}}{d n/k}$.
Let $\hat U_{ij}$ be the number of vertices $v\in V_i$ such that $b_v\geq l/2$.
Then \Lem~\ref{lem_balls_bins_cap} yields
	\begin{equation}\label{eqUji1}
	\pr\brk{|U_{ij}|\geq t}\leq \pr\brk{|U_{ij}^*|\geq t}\leq\pr\brk{\hat U_{ij}\geq t}\qquad\mbox{for all }t\geq0.
	\end{equation}
Furthermore, $\hat U_{ij}$ has distribution $\Bin(n_i,q)$ with
	$q=\pr\brk{\Bin(d,p)\geq \ell/2}.$
Since $\eta_{ji}=\tilde O_k(k^{-2})n$, we have $p=\tilde O_k(k^{-2})$ and thus $dp=\Erw\brk{b_v}=\tilde O_k(k^{-1})$.
Consequently, the Chernoff bound yields
	$$q=\pr\brk{\Bin(d,p)\geq \ell/2}\leq \tilde O_k(k^{-\ell/2}).$$
Hence, using the Chernoff bound once more, we find that
	\begin{equation}\label{eqUji2}
	\pr\brk{\hat U_{ij}\leq \tilde O_k(k^{-{\ell/2}})n}\geq1-\exp(-\Omega(n)).
	\end{equation}
Thus, the assertion follows from~(\ref{eqUji1}), (\ref{eqUji2}) and our choice of $\ell$.
\qed\end{proof}

\begin{lemma}\label{lem_z}
 With probability at least $1-O(1/n)$ 
 the set $Y$ satisfies $|Y| \leq 4nk^{-30}$.
\end{lemma}
\begin{proof}
By Lemmas \ref{lem_u'} and \ref{lem_u} 
we may assume that $|U \cup U'|\leq 2nk^{-30}$. 
Now, let $t_0=2nk^{-30}$.
If $Y=Y^{\bc t}$ for some $t<t_0$, then clearly $|Y|=|Y^{\bc t}|\leq 2nk^{-30}$,
because only one vertex is added at a time.
Thus, we need to show that the probability that $Y\neq Y^{\bc t}$ is $O(1/n)$.

Indeed, after completing step $t_0$,  
the subgraph of $\cG(\sigma)$ induced on $Y^{(t_0)}$ spans at least $\ell\cdot t_0$ edges, while
the number of vertices is $|Y^{(t_0)}|\leq |U\cup U'|+t_0\leq 2t_0\leq 8nk^{-30}$.
Hence, $\cG(\sigma)$ violates~(\ref{prop_edges_5s}).
 \Lem~\ref{lem_edges_5s} shows that the probability of this event is $O(1/n)$.
\qed\end{proof}

\medskip\noindent
Finally, \Prop~\ref{Prop_100-Core} follows immediately from \Prop~\ref{Prop_core_constr} and \Lem s~\ref{lem_u'}--\ref{lem_z}.

\subsection{Proof of \Prop~\ref{pro_free_1_2}} \label{Sec_pro_free_1_2}

Let $V_i=\sigma^{-1}(i)$ for $i\in\brk k$.
In order to estimate the number of complete vertices, we need to get a handle on two events.
First, the event that a vertex $v\in V_i$ fails to have a neighbor in some color class $V_j$ with $j\neq i$.
Second, the event that, given that $v$ has at least one neighbor in color class $V_j$, it indeed has a neighbor inside the core.
More precisely, with $W,Y$ the sets defined in~(\ref{eqW}) and~(\ref{eqY}), it suffices to bound the probability that
all neighbors of $v$ in $V_j$ lie in $W\cup Y$.
This is because $V\setminus(W\cup Y)$ is contained in the core by \Prop~\ref{Prop_core_constr}.

Thus, $S_0$ be the set of vertices that fail to have a neighbor in at least one color class other than their own in $\cG(\sigma,\mu)$.
Moreover, let $S_1$ be the set of vertices $v\not\in S_0$ such that for some color $i\neq\sigma(v)$ all neighbors of $v$ in $V_i$ 
belong to $W_i$.

\begin{proposition}\label{Prop_P123}
If $v$ is a $1$-free vertex, then one of the following three statements is true.
\begin{description}
\item[(P1)] $v \in S_0$.
\item[(P2)] $v \in S_1$.
\item[(P3)] 
		$v$ has a neighbor in $Y$.
\end{description}
\end{proposition}
\begin{proof}
Let $v$ be a vertex that satisfies none of {\bf(P1)--(P3)}.
Let $j\in\brk k\setminus\cbc{\sigma(v)}$.
Since $v\not\in S_0$, $v$ has at least one neighbor in $V_j$.
In fact, as $v\not\in S_1$, $v$ has a neighbor $w\in V_j\setminus W$.
Furthermore, because $v$ does not have a neighbor in $Y$, we have $w\in V\setminus(W\cup Y)$.
 \Prop~\ref{Prop_core_constr} implies that $w$ belongs to the $(\sigma,\ell)$-core,
which means that $v$ is not $1$-free.
\qed\end{proof}

Thus, in order to prove \Prop~\ref{pro_free_1_2} it suffices to estimate $|S_0|$, $|S_1|$ and the number of vertices that satisfy {\bf(P3)}.
These estimates employ the binomial approximation to the hypergeometric distribution provided by \Lem~\ref{lem_balls_bins_cap}.

\begin{lemma}\label{lem_s_0}
 With probability at least $1-O(1/n)$ we have $|S_{0}|\leq \frac nk(1+\tilde O_k(1/k))$.
\end{lemma}
\begin{proof}
 Let us fix $i,j \in [k]$, $i\neq j$, and $v\in V_j$.
Let $S_{0ij}$ be the set of all $v\in V_i$ that do not have a neighbor in $V_j$ in $\cG(\sigma,\mu)$.
Given the number $dn\mu_{ij}$ of $V_i$-$V_j$-edges, the actual set of clones in $V_i\times\brk d$ that $\vec\Gamma_{\sigma,\mu}$ joins
to a clone in $V_j\times\brk d$ is uniformly distributed.
Hence, \Lem~\ref{lem_balls_bins_cap} applies:
 let $(b_v)_{v\in V_i}$ be a family of independent $\Bin(d,p_{ij})$ random variables with $p_{ij}=k\mu_{ij}\sim(k-1)^{-1}$.
Moreover, let $$q_{ij}=\pr\brk{\Bin(d,p_{ij})=0}\sim(1-1/(k-1))^d.$$
Then with $\hat S_{0ij}$ a random variable with distribution $\Bin(n/k,q_{ij})$ we have
	\begin{equation}\label{eqStochDom1}
	\pr\brk{|S_{0ij}|\geq t}\leq O(\sqrt n)\cdot\pr\brk{\hat S_{0ij}\geq t}\qquad\mbox{for all }t\geq0.
	\end{equation}
Since by our assumption~(\ref{eqSec_Prop_goodFirstMoment}) on $d$ we have
	$$q_{ij}\sim (1-1/(k-1))^d\leq\exp(-d/(k-1))\leq k^{-2}+\tilde O_k(k^{-3}),$$ 
we see that $\Erw[\hat S_{0ij}]\leq n (k^{-3}+\tilde O_k(k^{-4}))$ for all $i\neq j$.
Hence, by the Chernoff bound we have
	$$\pr\brk{\hat S_{0ij}\leq n (k^{-3}+\tilde O_k(k^{-4}))}=o(n^{-2}).$$
Summing over all $i\neq j$ and using~(\ref{eqStochDom1}), we thus obtain
	$\pr[|S_{0}|\leq n(k^{-1}+\tilde O_k(k^{-2}))]\geq1-O(1/n)$.
\qed\end{proof}

To bound the size of $S_1$, 
consider first for every vertex $v \in V_i$ and every set of colors $J \subset [k]\setminus\cbc{i}$ the event 
	$\cB_{v,J}=\{e(v,\bigcup_{j\in J}V_j)\leq 5\}.$
Let $B_{i,J}$ be the number of vertices $v\in V_i$ for which the event $\cB_{v,J}$  occurs.

\begin{lemma}\label{lem_b_j}
 For any set $J$ of size $|J|\leq 2$ we have
	\begin{equation}\label{eqBadEvent1}
	\pr\brk{B_{i,J}\leq \frac nk\cdot\tilde O_k(k^{-2|J|})}\geq1-\exp(-\Omega(n)).
	\end{equation}
\end{lemma}
\begin{proof}
 Let $j\in J$.
Given $\mu_{ij}$, the set of clones in $V_i\times\brk d$ that $\vec\Gamma_{\sigma,\mu}$ links to $V_j\times\brk d$ is uniformly distributed.
Thus, \Lem~\ref{lem_balls_bins_cap} applies:
let $(b_{v,j})_{v\in V_i}$ be a family of independent random variables with distribution $\Bin(d,p_{ij})$, where $p_{ij}=k\mu_{ij}\sim(k-1)^{-1}$.
Let $\hat B_{i,J}$ be the number of vertices $v$ such that $b_{v,j}\leq 5$ for all $j\in J$.
Moreover, let $e_{v,j}$ be the number of clones $(v,h)$, $h\in\brk d$ with $\vec\Gamma_{\sigma,\mu}(v,h)\in V_j\times\brk d$.
The events $(\cbc{e_{v,j}\leq 5})_{j\in J}$ are negatively correlated (namely, the fact that $v$ has a small number of neighbors in some
color class $j$ makes it less likely that the same occurs for another color class $j'$).
Therefore, \Lem~\ref{lem_balls_bins_cap} yields
	\begin{eqnarray}\label{eqBadEvent2}
	\pr\brk{B_{i,J}\geq t}&\leq&O(n^{|J|/2})\cdot\pr\brk{\hat B_{i,J}\geq t}\qquad\mbox{ for any }t\geq0.
	\end{eqnarray}
Furthermore, because the random variables $(b_{v,j})$ are independent and
	$\Erw\brk{b_{v,j}}=d p_{ij}\geq2\ln k$,
the Chernoff bound yields
	\begin{equation}\label{eqBadEvent3}
	\pr\brk{\hat B_{i,J}\leq \frac nk\cdot \tilde O_k(k^{-2|J|})}\geq1-\exp(-\Omega(n)).
	\end{equation}
Thus, (\ref{eqBadEvent1}) follows from~(\ref{eqBadEvent2}) and~(\ref{eqBadEvent3}).
\qed\end{proof}

\begin{corollary}\label{lem_s_1}
 With probability at least $1-o(n^{-2})$ we have $|S_1|\leq n\cdot \tilde{O}_k(k^{-2})$.
\end{corollary}
\begin{proof}
Let $i,j\in\brk k$, $i\neq j$.
By \Lem~\ref{lem_w} we may assume that $|W_j|\leq\tilde O_k(k^{-2})n$.
Hence,
	$$e_{\cG(\sigma,\mu)}(V_i,W_j)\leq \tilde O_k(k^{-2})n,$$
because $e_{\cG(\sigma,\mu)}(w,V_i)=O_k(\ln k)$ for all $w\in W_j$ by the definition of $W_j$.
By comparison, 
	$$e_{\cG(\sigma,\mu)}(V_i,V_j)=dn\mu_{ij}\sim dn/\bc{k(k-1)}.$$
Now, condition on the event that $e_{\cG(\sigma,\mu)}(V_i,W_j)=w_{ij}$ for some specific number $w_{ij}=\tilde O_k(k^{-2})n$.
In addition, let $(e_{vj})_{v\in V_i}$ be a sequence of non-negative integers such that $\sum_{v\in V_i}e_{vj}=dn\mu_{ij}$,
and condition on the event that $e_{\cG(\sigma,\mu)}(v,V_j)=e_{vj}$ for all $v\in V_i$.
Given this event $\cF=\cF(w_{ij},\cbc{e_{vj}})$, we are interested in the random variables $f_{v}=e_{\cG(\sigma,\mu)}(v,W_j)$, $v\in V_i$.
Let $(g_{v})_{v\in V_i}$ be a family of independent random variables such that $g_v$ has distribution $\Bin(e_{vj},w_{ij}/(dn\mu_{ij}))$.
Given $\cF$, the set of clones among $V_i\times\brk d$ that $\vec\Gamma_{\sigma,\mu}$ matches to $W_{j}\times\brk d$  is simply a random subset of size $w_{ij}$
of the set of clones that get matched to $V_j\times\brk d$.
Therefore, by \Lem~\ref{lem_balls_bins_cap},
for any sequence $(t_v)_{v\in V_i}$ of integers
 we have
	\begin{eqnarray}\nonumber
	\pr\brk{\forall v\in V_i:f_v=t_v|\cF}&=&\pr\brk{\forall v\in V_i:g_v=t_v\bigg|\sum_{v\in V_i}g_v=w_{ij}}\\
		&\leq& O(\sqrt n)\pr\brk{\forall v\in V_i:g_v=t_v}.\label{eqTrivialize1}
	\end{eqnarray}

Now, let $S_{1ij}'$ be the number of all vertices $v\in V_i$ such that
all neighbors of $v$ in $V_j$ belong to $W_j$ and such that $e_{\cG(\sigma,\mu)}(v,V_j)\geq5$.
Moreover, let $\hat S_{1ij}'$ be the number of $v\in V_i$ such that $g_v=e_{vj}\geq5$.
Because $w_{ij}/(dn\mu_{ij})=\tilde O_k(k^{-1})$, we find that
	$$\Erw\brk{\hat S_{1ij}'}\leq \frac nk\cdot\tilde O_k(k^{-5}).$$
Furthermore, $\hat S_{1ij}'$ is a binomial random variable.
Therefore, the Chernoff bound yields
	\begin{equation}\label{eqTrivialize2}
	\pr\brk{\hat S_{1ij}'\leq \frac nk\cdot\tilde O_k(k^{-5})}\geq1-\exp(-\Omega(n)).
	\end{equation}
Combining~(\ref{eqTrivialize1}) and~(\ref{eqTrivialize2}), we obtain
	\begin{equation}\label{eqTrivialize3}
	\pr\brk{S_{1ij}'\leq\frac nk\cdot\tilde O_k(k^{-5})|\cF}\geq1-\exp(-\Omega(n)).
	\end{equation}
Further, because~(\ref{eqTrivialize3}) holds for all $w_{ij},\cbc{e_{vj}}$, we obtain the unconditional bound
	\begin{equation}\label{eqTrivialize4}
	\pr\brk{S_{1ij}'\leq \frac nk\cdot\tilde O_k(k^{-5})}\geq1-\exp(-\Omega(n)).
	\end{equation}

In addition, let $S_{1ij}''$ be the number of vertices $v\in V_i$ such that 
all neighbors of $v$ in $V_j$ belong to $W_j$ and $1\leq e(v,V_j)<5$.
Because we are conditioning on the numbers $e_{vj}$, the event $\cF$ determines the number $B_{i,\cbc j}$
of vertices $v\in V_i$ with $e_{vj}=e(v,V_j)<5$.
Now, consider the number $\hat S_{1ij}''$ of vertices $v\in V_i$ with $1\leq e_{vj}<5$ such that $g_v=e_{vj}$.
Then $\hat S_{1ij}''$ is a binomial random variable with
	$$\Erw\brk{\hat S_{1ij}''}\leq B_{i,\cbc j}\cdot\tilde O_k(k^{-1}).$$
Hence, by the Chernoff bound
	\begin{equation}\label{eqTrivialize5}
	\pr\brk{\hat S_{1ij}''\leq B_{i,\cbc j}\cdot\tilde O_k(k^{-1})+n^{2/3}|\cF}\geq1-o(n^{-2}).
	\end{equation}
Combining (\ref{eqTrivialize1}) and~(\ref{eqTrivialize5}), we find
	$$\pr\brk{S_{1ij}''\leq B_{i,\cbc j}\cdot\tilde O_k(k^{-1})+n^{2/3}|\cF}\geq1-o(n^{-2}).$$
Thus, Lemma \ref{lem_b_j} yields the unconditional bound
	\begin{equation}\label{eqTrivialize6}
	\pr\brk{S_{1ij}''\leq n\cdot\tilde O_k(k^{-4})}\geq1-o(n^{-2}).
	\end{equation}	
Combining~(\ref{eqTrivialize4}) and~(\ref{eqTrivialize6}) and using the union bound, we obtain
	\begin{equation}\label{eqTrivialize7}
	\pr\brk{|S_1|\leq\sum_{i,j\in\brk k:i\neq j}S_{1ij}'+S_{1ij}''\leq n\cdot\tilde O_k(k^{-2})}\geq1-o(n^{-2}),
	\end{equation}	
as claimed.
\qed\end{proof}

\begin{lemma}\label{lem_y_ny}
With probability at least $1-\exp(-\Omega(n))$ there are no more than $nk^{-26}$ vertices that have
a neighbor in $Y$. 
\end{lemma}
\begin{proof}
 \Lem~\ref{lem_z} shows that with probability $1-\exp(-\Omega(n))$ we have $\abs Y\leq nk^{-29}$.
In this case, the number of neighbors of vertices in $Y$ is bounded by $d|Y|\leq n k^{-27}$, because all vertices have degree $d\leq 2k\ln k$.
Thus, 
	$
	\pr\brk{|Y\cup N(Y)|\leq nk^{-26}}\geq1-\exp(-\Omega(n)).$
\qed\end{proof}

\medskip\noindent{\em Proof of \Prop~\ref{pro_free_1_2}.}
Since \Prop~\ref{Prop_P123} shows that any $1$-free vertex satisfies one of the conditions {\bf(P1)--(P3)}, Lemmas \ref{lem_s_0}--\ref{lem_y_ny}
imply that with probability $1-O(1/n)$ the number of $1$-free vertices is bounded by $n(k^{-1}+\tilde O_k(k^{-2}))$.
This establishes the first assertion.

To bound the number of $2$-free vertices, let $i\in\brk k$, let $J\subset\brk k\setminus\cbc i$ be a set of size $|J|=2$ and let $T_{i,J}$ be the number of vertices $v\in V_i$
that fail to have a neighbor in $\bigcup_{j\in J} V_j$.
Then $|T_{i,J}|\leq B_{i,J}$.
Therefore, Lemma \ref{lem_b_j} implies that
	\begin{equation}\label{eq2free1}
	\pr\brk{|T_{i,J}|\leq \frac nk\cdot \tilde O_k(k^{-4})}\geq1-\exp(-\Omega(n)).
	\end{equation}
Furthermore, the total number $T$ of $2$-free vertices satisfies
	\begin{equation}\label{eq2free2}
	T\leq\sum_{i=1}^k\sum_{J\subset\brk k\setminus\cbc i:|J|=2}T_{i,J}.
	\end{equation}
Combining~(\ref{eq2free1}) and~(\ref{eq2free2}) and using the union bound, we thus obtain the desired bound.
\qed

\section{The second moment}\label{Sec_second}

{\em Throughout this section, we assume that $k$ divides $n$ and that $d$ satisfies~(\ref{eqSec_Prop_goodFirstMoment}).}

\subsection{Outline}

In this section we complete the proof of the first part of \Thm~\ref{Thm_main} (the upper bound on the chromatic number of $\gnd$).
The key step is to carry out a second moment argument for the number $\Zkg$ of good $k$-colorings.
Let $\cB$ be the set of all balanced maps $\sigma:V\ra\brk k$ and let
	$\cR=\cbc{\rho(\sigma,\tau):\sigma,\tau\in\cB}$
be the set of all possible overlap matrices (as defined in~(\ref{eqOverlapMatrix})).
For each $\rho\in\cR$ we consider
	\begin{eqnarray*}
	\Zrg&=&\abs{\cbc{(\sigma,\tau):\sigma,\tau\mbox{ are good $k$-colorings }\rho(\sigma,\tau)=\rho}}\quad\mbox{ and}\\
	\Zrb&=&\abs{\cbc{(\sigma,\tau):\sigma,\tau\mbox{ are balanced $k$-colorings with }\rho(\sigma,\tau)=\rho}}\geq\Zrg.
	\end{eqnarray*}
Because the second moment $\Erw[\Zkg^2]$ of the number of good $k$-colorings of $\Gnd$ is nothing but
the expected number of {\em pairs} of good $k$-colorings, 
we have the expansion
	\begin{eqnarray}\label{eqsecond1}
	\Erw\brk{\Zkg^2}&=&\sum_{\rho\in\cR}\Erw\brk{\Zrg}.
	\end{eqnarray}

The second moment argument for the number $\Zrb$ of balanced $k$-colorings
of $\Gnd$ carried out in~\cite{KPGW} does not work for the (entire) range of $d$ in \Thm~\ref{Thm_main}.
However, an important part of that argument does carry over to this entire range of $d$.
More precisely, we can salvage the following estimate of the contribution 
of $\rho$ ``close'' to the flat matrix $\bar\rho=\frac1k\vecone$ with all entries equal to $1/k$.
	
\begin{proposition}[\cite{KPGW}]\label{Prop_centre}
Let
	\begin{equation}\label{eqbarcR}
	\bar\cR=\cbc{\rho\in\cR:\norm{\rho-\bar\rho}_{\infty}\leq n^{-1/2}\ln^{2/3}n}.
	\end{equation}
Then with $\delta_j,\lambda_j$ as in~(\ref{eqsmallSubgraphConditioning}) we have
	$$\sum_{\rho\in\bar\cR}\Erw\brk{\Zrb}\leq(1+o(1))\Erw\brk{\Zkb}^2\cdot\exp\brk{\sum_{j=1}^\infty\lambda_j\delta_j^2}.$$
\end{proposition}


Of course, to estimate the right-hand side of~(\ref{eqsecond1}), we also need to estimate the contribution of overlaps $\rho\not\in\bar\cR$.
To this end, we are going to establish an explicit connection between~(\ref{eqsecond1}) and the second moment argument
for $\gnm$ performed in~\cite{ACOVilenchik}.
%
As in~\cite{AchNaor,ACOVilenchik}, we define for a doubly-stochastic $k\times k$ matrix $\rho=(\rho_{ij})_{i,j\in\brk k}$ the functions
	\begin{eqnarray*}
	f(\rho)&=&H(\rho/k)+E(\rho),\qquad\mbox{where}\\
	H(\rho/k)&=&\ln k-\sum_{i,j=1}^k\rho_{ij}\ln\rho_{ij}\quad\mbox{ is the entropy of the distribution $\rho/k=(\rho_{ij}/k)_{i,j\in\brk k}$, and }\\
	E(\rho)&=&\frac d2\ln\brk{1-\frac2k+\frac1{k^2}\sum_{i,j=1}^k\rho_{ij}^2}.
	\end{eqnarray*}
In \Sec~\ref{Sec_Lemma_f} we are going to establish the following bound.
\begin{proposition}\label{Prop_f}
For any $\rho\in\cR$ we have
	$\Erw\brk{\Zrg}\leq\Erw\brk{\Zrb} \leq n^{O(1)}\exp\brk{nf(\rho)}$.
\end{proposition}

Similar bounds as \Prop~\ref{Prop_f} were derived, somewhat implicitly, in~\cite{AMoColor,KPGW}.
We include the proof here because the present argument is substantially simpler than those in~\cite{AMoColor,KPGW} and because
we are going to need some details of the calculation later to finish the proof of \Thm~\ref{Thm_main}.

Thus, we need to bound $f(\rho)$ for $\rho\in\cR\setminus\bar\cR$.
This is precisely the task that was solved in~\cite{ACOVilenchik} and that does, indeed, form the technical core of that paper.
Hence, let us recap some of the notation from~\cite{ACOVilenchik}.
We start by observing that the definition of ``good'' entails that {\em a priori} $\Zrg=0$ for quite a few $\rho\in\cR\setminus\bar\cR$.
More precisely, call a doubly-stochastic matrix $\rho$ \bemph{separable} if for any $i,j\in\brk k$ such that $\rho_{ij}>0.51$ we have $\rho_{ij}\geq1-\kappa$
	(with $\kappa$ as in \Def~\ref{Def_separable}).

The definition of ``good $k$-coloring'' ensures that $\Zrg=0$ unless $\rho$ is separable.
Indeed, assume that there exist balanced $k$-colorings $\sigma,\tau$ such that $\rho(\sigma,\tau)$ fails to be separable.
Then there is a permutation $\pi$ of the colors $\brk k$ such that $0.51<\rho_{11}(\sigma,\tau)<1-\kappa$.
Hence, neither $\sigma$ nor $\tau$ is a separable $k$-coloring, and thus none of them is good.

%
The set of separable matrices can be split canonically into subsets determined by the number of entries that are greater than $0.51$.
Let us say that $\rho$ is \bemph{$s$-stable} if there are precisely $s$ pairs $(i,j)\in\brk k\times\brk k$ such that $\rho_{ij}\geq1-\kappa$.
Let
	\begin{eqnarray*}
	\Rgs&=&\cbc{\rho\in\cR:\rho\mbox{ is separable and $s$-stable for some $0\leq s\leq k-1$}}\qquad\mbox{and}\\
	\Rg&=&\bigcup_{s=0}^{k-1}\Rgs.
	\end{eqnarray*}

Let us turn the problem of estimating $f(\rho)$ over $\rho$ in the discrete set $\Rg$ into 
a continuous optimization problem.
As $n\ra\infty$ the set $\cR$ of overlap matrices lies dense in the set $\Birk$
of all doubly-stochastic $k\times k$ matrices, the {\em Birkhoff polytope}.
Furthermore, the sets $\Rgs$ and $\Rg$ are dense in
	\begin{eqnarray*}
	\Dgs&=&\cbc{\rho\in\Birk:\rho\mbox{ is separable and $s$-stable for some $0\leq s\leq k-1$}},\\
	\Dg&=&\bigcup_{s=0}^{k-1}\Dgs.
	\end{eqnarray*}

%

\begin{proposition}[\cite{ACOVilenchik}]\label{Prop_opt}
For any fixed $\eta>0$ there is $\delta>0$ such that
	$$\max\cbc{f(\rho):\rho\in\Dg\mbox{ such that }\norm{\rho-\bar\rho}_{\infty}\geq\eta}<f(\bar\rho).$$
\end{proposition}
Based on this estimate, we will prove the following bound in \Sec~\ref{Sec_Prop_opt_0}.

\begin{proposition}\label{Prop_opt_0}
We have
	$$\sum_{\rho\in\cR_{0,\mathrm{good}}\setminus\bar\cR}\Erw\brk{\Zrb}=o(\Erw\brk{\Zkb}^2).$$
\end{proposition}
%

\begin{corollary}\label{Cor_opt_0}
The random variable $\Zkg$ has the properties i.--iii.\ in \Thm~\ref{Thm_smallSubgraphConditioning}.
Furthermore, we have
	\begin{equation}\label{eqCor_opt_X}
	\sum_{\rho\in\cR\setminus\bar\cR}\Erw\brk{\Zrg}=o(\Erw\brk{\Zkg}^2).
	\end{equation}
\end{corollary}
\begin{proof}
\Cor~\ref{Cor_goodFirstMoment} already 
establishes conditions i.--ii.
Recall that condition iii.\ reads
	\begin{equation}\label{Cor_opt_01}
	\Erw\brk{\Zkg^2}\leq(1+o(1))\Erw\brk{\Zkg}^2\cdot\exp\brk{\sum_{j=1}^\infty\lambda_j\delta_j^2}.
	\end{equation}
\Prop s~\ref{Prop_goodFirstMoment} 
readily yields 
	\begin{eqnarray}
	\sum_{\rho\in\bar\cR}\Erw\brk{\Zrg}&\leq&
		\sum_{\rho\in\bar\cR}\Erw\brk{\Zrb}
			\leq(1+o(1))\Erw\brk{\Zkb}^2\cdot\exp\brk{\sum_{j=1}^\infty\lambda_j\delta_j^2}.
	\label{Cor_opt_02}
	\end{eqnarray}
Additionally, we need to bound the contribution of $\rho\in\cR\setminus\bar\cR$.


We start with $\rho\in\Rg\setminus\cR_{0,\mathrm{good}}$.
Any such $\rho$ has an entry $\rho_{ij}\geq0.51$, whence $\norm{\rho-\bar\rho}_\infty\geq\frac12$.
Therefore, \Prop~\ref{Prop_opt} implies that there is an $n$-independent number $\delta>0$ such that
$f(\rho)<f(\bar\rho)-\delta$.
(This $\delta$ exists because \Prop~\ref{Prop_opt} is {\em not} an asymptotic statement but just a result concerning
the maximum of the $n$-independent function $f$ over the equally $n$-independent compact set $\Dg$.)
Consequently, by \Prop~\ref{Prop_f}
	\begin{eqnarray}
	\Erw\brk{\Zrg}&\leq&\exp\brk{f(\bar\rho)n-\Omega(n)}.
			\label{Cor_opt_04}
	\end{eqnarray}
Moreover, a direct calculation yields
	\begin{eqnarray}
	f(\bar\rho)&=&2\ln k+d\ln(1-1/k)\sim\frac 2n\ln\Erw\brk{\Zkb}\qquad\mbox{[by \Prop~\ref{Prop_KPGWfirstMoment}]}.
			\label{Cor_opt_05}
	\end{eqnarray}
Combining~(\ref{Cor_opt_04}) and~(\ref{Cor_opt_05}), we obtain
	\begin{eqnarray*}
	\Erw\brk{\Zrg}&\leq&\Erw\brk{\Zkb}^2\cdot\exp\brk{-\Omega(n)}.
			\label{Cor_opt_06}
	\end{eqnarray*}
Because the {\em entire} set $\cR$ of overlap matrices has size $|\cR|\leq n^{k^2}$ (with room to spare), we thus obtain
	\begin{eqnarray}			\label{Cor_opt_07}
	\sum_{\rho\in\Rg\setminus\cR_{0,\mathrm{good}}}\Erw\brk{\Zrg}&\leq& n^{k^2}\Erw\brk{\Zkb}^2\cdot\exp\brk{-\Omega(n)}
		=o(\Erw\brk{\Zkb}^2).
	\end{eqnarray}

Further, if $\Zrg>0$ for some $\rho\not\in\Rg$, then $\rho$ must be $k$-stable (because $\Rg$ contains
	all separable overlap matrices that are $s$-stable for some $s<k$).
Thus, let $\cR_k$ be the set of all $k$-stable $\rho\in\cR$.
If $\sigma,\tau$ are balanced $k$-colorings such that $\rho(\sigma,\tau)$ is $k$-stable, then
there is a permutation $\lambda$ of $\brk k$ such that $\lambda\circ\tau\in\cC(\sigma)$.
Therefore, letting $\sigma$ range over good $k$-colorings of $\Gnd$, we obtain
from the upper bound on $|\cC(\sigma)|$ imposed in \Def~\ref{XDef_good}
	\begin{eqnarray}
	\Erw\brk{\sum_{\rho\in\cR_k}\Zrg}&\leq&\Erw\brk{\sum_{\sigma}k!|\cC(\sigma)|}\leq \frac{k!}{n}\cdot\Erw\brk{\Zkb}\Erw\brk{\Zkg}
		=o(\Erw\brk{\Zkb}^2).
			\label{Cor_opt_03}
	\end{eqnarray}

Finally, combining~(\ref{Cor_opt_02}), (\ref{Cor_opt_07}), (\ref{Cor_opt_03}) and \Prop~\ref{Prop_opt_0}, we see that
	\begin{eqnarray}				\label{Cor_opt_0999}
	\Erw\brk{\Zkg^2}&\leq&(1+o(1))\Erw\brk{\Zkb}^2\cdot\exp\brk{\sum_{j=1}^\infty\lambda_j\delta_j^2}+o(\Erw\brk{\Zkb}^2)
	\end{eqnarray}
Furthermore, as $\Erw[\Zkb]\sim\Erw[\Zkg]$ by \Prop~\ref{Prop_goodFirstMoment}, (\ref{Cor_opt_0999}) yields
	\begin{eqnarray}				\label{Cor_opt_099}
	\Erw\brk{\Zkg^2}&\leq&(1+o(1))\Erw\brk{\Zkg}^2\cdot\exp\brk{\sum_{j=1}^\infty\lambda_j\delta_j^2}+o(\Erw\brk{\Zkg}^2).
	\end{eqnarray}
Recalling the values of $\lambda_j$, $\delta_j$ from~(\ref{eqsmallSubgraphConditioning}), we see that the sum $\sum_{j=1}^\infty\lambda_j\delta_j^2$ converges.
Therefore, (\ref{Cor_opt_099}) implies~(\ref{Cor_opt_01}).
\qed\end{proof}

Together with \Thm~\ref{Thm_smallSubgraphConditioning}, \Cor~\ref{Cor_opt_0} implies that
$\Gnd$ is $k$-colorable \whp\ in the case that $k$ divides $n$.
In \Sec~\ref{Sec_Thm_main} we are going to provide a supplementary argument
that allows us to extend this result also to the case that the number of vertices is not divisible by $k$,
thereby completing the proof of the first part of \Thm~\ref{Thm_main}.
But before we come to that, let us prove \Prop s~\ref{Prop_f} and~\ref{Prop_opt_0}
	(under the assumption that $k$ divides $n$).

\subsection{Proof of \Prop~\ref{Prop_f}}\label{Sec_Lemma_f}

Let $\rho$ be a doubly-stochastic $k\times k$ matrix.
Moreover, let $\mu=(\mu_{ijst})_{i,j,s,t\in\brk k}$ have entries in $\brk{0,1}$.
We call $(\rho,\mu)$  a \bemph{compatible pair} if the following conditions are satisfied.
\begin{itemize}
\item $\frac nk\rho_{ij}$ is an integer for all $i,j\in\brk k$.
\item $dn\mu_{ijst}$ is an integer for all $i,j,s,t\in\brk k$.
\item We have
		\begin{eqnarray}\label{eqOverlapProp1}
	\mu_{ijst}&=&\mu_{stij},\quad
		\mu_{ijit}=0,\quad\mu_{ijsj}=0\qquad\forall i,j,s,t\in\brk k,\\
	\sum_{s,t=1}^k\mu_{ijst}&=&\rho_{ij}/k
		\qquad\forall i,j\in\brk k.\label{eqOverlapProp2}
	\end{eqnarray}
\end{itemize}
If $(\rho,\mu)$ is a compatible pair, then~(\ref{eqOverlapProp2}) ensures that $(\frac1k\rho,\mu)$ is $(d,n)$-admissible (cf.\ \Sec~\ref{Sec_partitions}),
if we view $\frac1k\rho$ as a probability distribution on $\brk k\times\brk k$ and $\mu$ as a probability distribution on $(\brk k\times\brk k)^2=\brk k^4$.

Let us also say that a pair $(\sigma,\tau)$  of $k$-colorings  of a multi-graph $\cG$ has \bemph{type $(\rho,\mu)$} if
$\rho(\sigma,\tau)=\rho$ and
	$$e_\cG(\sigma^{-1}(i)\cap\tau^{-1}(j),\sigma^{-1}(s)\cap\tau^{-1}(t))=\mu_{ijst}dn\quad\mbox{ for all }i,j,s,t\in\brk k$$
Let $Z_{\rho,\mu}$ be the number of pairs of $k$-colorings of $\Gnd$ of type $(\rho,\mu)$.
Recall that $H\bc\cdot$ denotes the entropy.
Applied to the notion of compatible pairs, \Cor~\ref{Lemma_skewed} directly yields

\begin{fact}\label{Lemma_compPair}
Let $(\rho,\mu)$ be a compatible pair.
Then
	$$\frac1n\ln\Erw\brk{Z_{\rho,\mu}}=H\bcfr{\rho}k-\frac d2\KL{\mu}{\frac\rho k\tensor\frac\rho k}+O(\ln n/n).$$
\end{fact}
To proceed, we need to rephrase the bound provided by Fact~\ref{Lemma_compPair} in terms of the function $f(\rho)$.

\begin{corollary}\label{Cor_compPair}
Let $(\rho,\mu)$ be a compatible pair.
Let $\cF=\{(i,j,s,t)\in\brk k^4:i=s\vee j=t\}$ and define
	\begin{equation}\label{eqCorKLmurho}
	\hat\rho=\bcfr{\rho_{ij}\rho_{st}\vecone_{(i,j,s,t)\not\in\cF}}{k^2-2k-\norm{\rho}_2^2}_{i,j,s,t\in\brk k}.
	\end{equation}
Then $\hat\rho$ is a probability distribution on $\brk k^4$ and
	\begin{eqnarray*}
	\frac1n\ln\Erw\brk{Z_{\rho,\mu}}&=&
		f(\rho)-\frac d{2k}\KL{\mu}{\hat\rho}+O(\ln n/n).
	\end{eqnarray*}
\end{corollary}
\begin{proof}
Because $\rho$ is doubly-stochastic, we have
	\begin{eqnarray*}
	\sum_{(i,j,s,t)\not\in\cF}\rho_{ij}\rho_{st}&=&\sum_{i,j,s,t\in\brk k}\rho_{ij}\rho_{st}-\sum_{(i,j,s,t)\in\cF}\rho_{ij}\rho_{st}\\
		&=&k^2-\sum_{i,j,t\in\brk k}\rho_{ij}\rho_{it}-\sum_{i,j,s\in\brk k}\rho_{ij}\rho_{sj}+\sum_{i,j=1}^k\rho_{ij}^2=k^2-2k+\norm\rho_2^2.
	\end{eqnarray*}
Thus, $\hat\rho$ is a probability distribution.
Moreover, 
	\begin{eqnarray*}\nonumber
	\KL{\mu}{\frac\rho k\tensor\frac\rho k}+\ln(1-2/k+k^{-2}\norm\rho_2^2)
		\nonumber\\
		&\hspace{-10cm}=&\hspace{-5cm}\sum_{i,j,s,t\in\brk k}\mu_{ijst}\brk{\ln\bcfr{k^2\mu_{ijst}}{\rho_{ij}\rho_{st}}+\ln(1-2/k+k^{-2}\norm\rho_2^2)}
			\quad\mbox{[as $\sum_{i,j,s,t\in\brk k}\mu_{ijst}=1$]}\nonumber\\
		&\hspace{-10cm}=&\hspace{-5cm}\sum_{(i,j,s,t)\not\in\cF}\mu_{ijst}\ln\bc{\mu_{ijst}\cdot\frac{k^2-2k+\norm\rho_2^2}{\rho_{ij}\rho_{st}}}
			\qquad\qquad\qquad\qquad\mbox{[due to (\ref{eqOverlapProp1})]}\nonumber\\
		&\hspace{-10cm}=&\hspace{-5cm}\KL{\hat\mu}{\hat\rho}.\label{eqCor_KL1}
	\end{eqnarray*}
The assertion thus follows from Fact~\ref{Lemma_compPair}.
\qed\end{proof}

\noindent{\em Proof of \Prop~\ref{Prop_f}.}
Let $\rho\in\cR$ and let $\cM(\rho)$ be the set of all probability distributions $\mu$ on $\brk k^4$ such that $(\rho,\mu)$ is a compatible pair.
Then
	$$\Zrb=\sum_{\mu\in\cM(\rho)}Z_{\rho,\mu}.$$
Furthermore, $|M(\rho)|\leq (dn)^{k^4}$ because of the requirement that $\mu_{ijst}dn$ be integral for all $i,j,s,t\in\brk k$.
Hence, 
	\begin{equation}\label{eqProp_fproof1}
	\frac1n\ln\Erw\brk{\Zrb}\leq \frac1n\ln|M(\rho)|+\frac1n\max_{\mu\in\cM(\rho)}\ln\Erw\brk{Z_{\rho,\mu}}
		=O(\ln n/n)+\frac1n\max_{\mu\in\cM(\rho)}\ln\Erw\brk{Z_{\rho,\mu}}.
	\end{equation}
Since $\KL{\mu}{\hat\rho}\geq0$ for any $\mu$, \Cor~\ref{Cor_compPair} yields
	\begin{equation}\label{eqProp_fproof2}
	\frac1n\max_{\mu\in\cM(\rho)}\ln\Erw\brk{Z_{\rho,\mu}}\leq f(\rho)+O(\ln n/n).
	\end{equation}
The assertion is immediate from~(\ref{eqProp_fproof1}) and~(\ref{eqProp_fproof2}).
\qed

\subsection{Proof of \Prop~\ref{Prop_opt_0}}\label{Sec_Prop_opt_0}

We begin by estimating $f(\rho)$ for $\rho$ close to $\bar\rho$.
The proof of the following lemma is based on considering the first two differentials of $f$ at the point $\bar\rho$;
	a very similar calculation appears in~\cite{ACOVilenchik}.

\begin{lemma}\label{Lemma_Hessian}
There is a number $\eta>0$  (independent of $n$) such that for all
	$$\rho\in\tilde\cR_0=\cbc{\rho\in\cR_0:\norm{\rho-\bar\rho}_\infty<\eta}$$
 we have
 	$f(\rho)\leq f(\bar\rho)-\frac14\norm{\rho-\bar\rho}_2^2.$
\end{lemma}
\begin{proof}
By construction, we have $\sum_{i,j=1}^k\rho_{ij}=k$ for all $\rho\in\cR$.
Therefore, we can parametrize the set $\cR$ as follows.
Let
	\begin{eqnarray*}
	\cL&:&\brk{0,1}^{k^2-1}\ra\brk{0,1}^{k^2},\quad
	\hat\rho=(\hat\rho_{ij})_{(i,j)\in\brk k^2\setminus\cbc{(k,k)}}\mapsto\cL(\hat\rho)=(\cL_{ij}(\hat\rho))_{i,j\in\brk k}
	\end{eqnarray*}
where
\begin{eqnarray*}
	\cL_{ij}(\hat\rho)&=&\left\{\begin{array}{cl}\hat\rho_{ij}&\mbox{ if }(i,j)\neq(k,k)\\
		k-\sum_{(s,t)\neq(k,k)}\hat\rho_{st}&\mbox{ if }i=j=k.
		\end{array}\right.
	\end{eqnarray*}
Let $\hat\cR_0=\cL^{-1}(\tilde\cR_0)$.
Then $\cL$ induces a bijection $\hat\cR_0\ra\tilde\cR_0$.

It is straightforward to compute the first two differentials 
of $f\circ\cL=H\circ \cL+E\circ\cL$. 
The result is that the first differential $D(f\circ\cL)$ equals zero at $\bar\rho$.
Furthermore,
for $\hat\rho\in\hat\cR_0$ the second differential is given by
	\begin{eqnarray*}
	\frac{\partial^2f\circ\cL}{\partial\hat\rho_{ij}^2}(\hat\rho)&=&-\frac1k\brk{\frac{1}{\cL_{ij}(\hat\rho)}+\frac{1}{\cL_{kk}(\hat\rho)}}
			+O_k(\ln k/k)\ \quad(i,j\in\brk{k-1})\\
	\frac{\partial^2f\circ\cL}{\partial\hat\rho_{ij}\partial\hat\rho_{ab}}(\hat\rho)&=&-\frac1{k\cL_{kk}(\hat\rho)}
		+\tilde O_k(\ln k/k)\qquad\qquad\qquad\quad(a,b,i,j\in\brk{k-1},\,(a,b)\neq(i,j)).
	\end{eqnarray*}
Thus, 
for $\eta>0$ sufficiently small
the Hessian $D^2(f\circ\cL)$ is negative-definite with all eigenvalues smaller than $-1/2$.
Hence, the assertion follows from Taylor's theorem.
\qed\end{proof}


\medskip\noindent\emph{Proof of \Prop~\ref{Prop_opt_0}.}
Assume that $\rho\in\cR_{0,\mathrm{good}}\setminus\bar\cR$.
We claim that
	\begin{equation}\label{eqProp_opt_0_1}
	f(\rho)\leq f(\bar\rho)-\Omega(n^{-1}\ln^{4/3}n).
	\end{equation}
To see this, let $\eta>0$ be the ($n$-independent) number promised by \Lem~\ref{Lemma_Hessian}.
We consider two cases.
\begin{description}
\item[Case 1: $\norm{\bar\rho-\rho}_\infty<\eta$.]
	By the definition~(\ref{eqbarcR}) of $\bar\cR$ and as $\rho\not\in\bar\cR$, we have $\norm{\rho-\bar\rho}_\infty\geq n^{-\frac12}\ln^{\frac23}n$.
	Moreover, because $\norm{\bar\rho-\rho}_\infty<\eta$, \Lem~\ref{Lemma_Hessian} applies and yields
		$$f(\rho)-f(\bar\rho)\leq -\frac14\norm{\bar\rho-\rho}_2^2\leq -\frac14\norm{\bar\rho-\rho}_\infty^2\leq n^{-1}\ln^{4/3}n,$$
	as desired.
\item[Case 1: $\norm{\bar\rho-\rho}_\infty\geq\eta$.]
	Since $\eta>0$ remains fixed as $n\ra\infty$, \Prop~\ref{Prop_opt} yields an $n$-independent number $\xi=\xi(\eta)>0$ such that
		$f(\rho)\leq f(\bar\rho)-\xi$.
		Hence, (\ref{eqProp_opt_0_1}) is satisfied with room to spare.
\end{description}
Finally, plugging~(\ref{eqProp_opt_0_1}) into \Prop~\ref{Prop_f}, we obtain
	\begin{eqnarray*}
	\sum_{\rho\in\cR_{0,\mathrm{good}}\setminus\bar\cR}\Erw\brk{\Zrb}&\leq&
		\abs{\cR_{0,\mathrm{good}}}\cdot\max_{\rho\in\cR_{0,\mathrm{good}}\setminus\bar\cR}\Erw\brk{\Zrb}\\
		&\leq&n^{O(1)}\cdot \max_{\rho\in\cR_{0,\mathrm{good}}}\exp(f(\rho)n)\qquad\mbox{[as $\cR_{0,\mathrm{good}}\leq\abs\cR\leq n^{k^2}$]}\\
		&\leq&\exp(f(\bar\rho)n-\Omega(\ln^{4/3}))=o(\Erw\brk{\Zkb}^2),
	\end{eqnarray*}
as claimed.
\qed

\subsection{Proof of \Thm~\ref{Thm_main} (part 1)}\label{Sec_Thm_main}

\Cor~\ref{Cor_opt_0} shows that $\Zkg(\Gnd)$ satisfies the assumptions of \Thm~\ref{Thm_smallSubgraphConditioning},
which therefore implies that $\Gnd$ is $k$-colorable \whp\ for $n$ divisible by $k$.
To also deal with the case that the number of vertices is not divisible by $k$, we need a few definitions.
Recall from \Sec~\ref{Sec_good} 
that a balanced $k$-coloring $\sigma$ of $\Gnd$ skewed if
	$$\max_{1\leq i<j\leq k}\abs{e_{\Gnd}(\sigma^{-1}(i),\sigma^{-1}(j))-\frac{dn}{k(k-1)}}>\sqrt n\ln n.$$
In addition, a \bemph{skewed pair} is a pair $(\sigma,\tau)$ of good $k$-colorings 
such that
either 
	\begin{eqnarray*}
	\norm{\rho(\sigma,\tau)-\bar\rho}_\infty&>&\sqrt n\ln^{2/3} n\qquad\mbox{ or}\\
	\max_{i,j,s,t\in\brk k:i\neq s,j\neq t}\abs{e_{\Gnd}(\sigma^{-1}(i)\cap\tau^{-1}(j),\sigma^{-1}(s)\cap\tau^{-1}(t))-\frac{dn}{k^2(k-1)^2}}&>&\sqrt n\ln n.
	\end{eqnarray*}
The following lemma paraphrases the argument from~\cite[\Sec~4]{KPGW}.

\begin{lemma}\label{Lemma_kndivn}
Assume that for $n$ divisible by $k$ the following is true.
\begin{enumerate}
\item The random variable $\Zkg$ satisfies the conditions i.---iii.\ of \Thm~\ref{Thm_smallSubgraphConditioning}.
\item The expected number of skewed $k$-colorings is $o(\Erw\brk{\Zkg})$.
\item The expected number of skewed pairs is $o(\Erw\brk{\Zkg}^2)$.
\end{enumerate}
Then $\cG(n+z,d)$ is $k$-colorable \whp\ for any $0\leq z<k$ such that $d(n+z)$ is even.
\end{lemma}

\medskip\noindent{\em Proof of \Thm~\ref{Thm_main}, part 1.}
Due to \Lem~\ref{Lemma_configurationModel}
we just need to verify the assumptions of \Lem~\ref{Lemma_kndivn}.
\Cor~\ref{Cor_opt_0} readily implies the first assumption.
Furthermore, the second assertion follows from \Cor~\ref{Cor_skewed} and \Prop~\ref{Prop_goodFirstMoment}.


With respect to the third assertion, we call from~(\ref{eqCor_opt_X}) that
%
	\begin{equation}\label{eqFinishingUp3}
	\sum_{\rho:\norm{\rho-\bar\rho}_\infty>n^{-1/2}\ln^{2/3}n}\Erw\brk{\Zrg}
		=o(\Erw\brk{\Zkg}^2).
	\end{equation}

Now, assume that $\rho$ satisfies $\norm{\rho-\bar\rho}_\infty\leq n^{-1/2}\ln^{2/3}n$.
Let $\mu=(\mu_{ijst})_{i,j,s,t\in\brk k}$ be such that $(\rho,\mu)$ is a compatible pair.
Let $Z_{\rho,\mu}$ be as in \Sec~\ref{Sec_Lemma_f} and let
 $\hat\rho$ be as in~(\ref{eqCorKLmurho}).
Then \Cor~\ref{Cor_compPair} yields
	\begin{eqnarray}\label{eqFinishingUp4}
	\Erw\brk{Z_{\rho,\mu}}&=&n^{O(1)}\Erw\brk{\Zrg}\exp\brk{-\frac{dn}{2k}\KL{\mu}{\hat\rho}}.
	\end{eqnarray}
Suppose that $i,j,s,t\in\brk k$, $i\neq s$, $j\neq t$ are indices such that $|\mu_{ijst}-k^{-2}(k-1)^{-2}|>n^{-1/2}\ln n$.
Since $\norm{\rho-\bar\rho}_\infty\leq n^{-1/2}\ln^{2/3}n$, we have
$$|\mu_{ijst}-\hat\rho_{ijst}|=\Omega\bc{n^{-1/2}\ln n}.$$
Therefore, Fact~\ref{Fact_KL} implies that
	$\KL{\hat\mu}{\hat\rho}=\Omega(\ln^2n/n).$
Hence, (\ref{eqFinishingUp4}) yields
	\begin{eqnarray}\label{eqFinishingUp5}
	\Erw\brk{Z_{\rho,\mu}}&=&n^{O(1)}\Erw\brk{\Zrg}\exp\brk{-\Omega(\ln^2n)}=\Erw\brk{\Zrg}\exp\brk{-\Omega(\ln^2n)}.
	\end{eqnarray}
Since the number of possible matrices $\mu$ is bounded by $n^{k^4}$, (\ref{eqFinishingUp5}) entails that the number $Z_\rho'$
of skewed pairs $(\sigma,\tau)$ with overlap $\rho$ satisfies
	\begin{eqnarray}\label{eqFinishingUp6}
	\Erw\brk{Z_{\rho}'}&\leq&n^{k^4}\cdot\Erw\brk{\Zrg}\exp\brk{-\Omega(\ln^2n)}=\Erw\brk{\Zrg}\exp\brk{-\Omega(\ln^2n)}.
	\end{eqnarray}
Since $\sum_{\rho\in\cR}\Erw\brk{\Zrg}=O(\Erw\brk{\Zrg}^2)$ by \Cor~\ref{Cor_opt_0}, (\ref{eqFinishingUp3}) and~(\ref{eqFinishingUp6}) imply that the
total expected number 
 of skewed pairs satisfies is
	$o(\Erw\brk{\Zkg}^2),$
as desired.
\qed

\section{The Lower Bound on the Chromatic Number}\label{Sec_lower}

\subsection{Outline}\label{Sec_lower_Outline}

The goal in this section is to establish the second part of \Thm~\ref{Thm_main}, i.e.,  the lower bound on the chromatic number of $\chi(\gnd)$.
More precisely, we are going to show that with
	\begin{eqnarray}
	 d^+=(2k-1)\ln k-1+\frac{3}{(\ln k)^{3/2}}, \nonumber
	\end{eqnarray}
the random multi-graph $\gnd$ fails to be $k$-colorable \whp\ for $d>d^+$;
	then \Lem~\ref{Lemma_configurationModel} implies that the same is true of $\gnd$.
To get started, we recall the upper bound on the expected number of $k$-colorings of $\Gnd$.
This bound has been attributed to Molloy and Reed~\cite{MolloyReed}.
We include the simple calculation here for the sake of completeness.

\begin{lemma}\label{Lemma_simpleFirst}
Let $\rho=(\rho_1,\ldots,\rho_k)$ be a probability distribution on $\brk k$ such that $\rho_in$ is an integer for all $i\in\brk k$.
Let $Z^\rho$ be the number of $k$-colorings $\sigma$ of $\Gnd$
such that $\abs{\sigma^{-1}(i)}=\rho_i n$ for all $i\in\brk k$.
Then
	\begin{equation}\label{eqProp_simpleFirst1}
	\frac1n\ln\Erw[Z^\rho]= H(\rho)+\frac{d}{2}\ln (1-\norm\rho_2^2)+O(\ln n/n).
	\end{equation}
\end{lemma}
\begin{proof}
Let $M$ be the set of all probability distributions $\mu$ on $\brk k\times\brk k$ such that $(\rho,\mu)$ is $(d,n)$-admissible (as defined in \Sec~\ref{Sec_partitions}).
Moreover, for any $\mu\in M$
let $Z_{\rho,\mu}$ be the number of $k$-colorings of $\Gnd$ such that  $\abs{\sigma^{-1}(i)}=\rho_i n$ for all $i\in\brk k$
and such that $e_{\Gnd}(\sigma^{-1}(i),\sigma^{-1}(j))=dn\mu_{ij}$ for all $i,j\in\brk k$.
Then Fact~\ref{Fact_KL} and \Cor~\ref{Cor:Skewed} yield
	\begin{eqnarray}\label{eqProp_simpleFirst1a}
		\frac1n\ln\Erw\brk{Z_{\rho,\mu}}&=&H(\rho)+\frac{d}{2}\ln (1-\norm\rho_2^2)-\frac{d}2\KL{\mu}{\hat\rho}+O(\ln n/n)\quad\mbox{for any }\mu\in M.
	\end{eqnarray}
Since $\abs M\leq n^{k^2}$ (due to the condition that $dn\mu_{ij}$ must be an integer for all $i,j$), (\ref{eqProp_simpleFirst1a}) implies together with Fact~\ref{Fact_KL} that
	$$\frac1n\ln\Erw[Z^\rho]=\frac1n\ln\sum_{\mu\in M}\Erw\brk{Z_{\rho,\mu}}= H(\rho)+\frac{d}{2}\ln (1-\norm\rho_2^2)+O(\ln n/n),$$
as claimed.
\qed\end{proof}

\begin{corollary}\label{Prop_simpleFirst}
We have
	$$\frac1n\ln\Erw\brk{\Zkc}= \ln k+\frac d2\ln(1-1/k)+O(\ln n/n).$$
Furthermore, if $d\geq(2k-1)\ln k$, then
 $\Erw[\Zkc]\leq\exp(-\Omega(n))$.
\end{corollary}
\begin{proof}
Let $\rho$ be a probability distribution on $\brk k$ and let $Z^\rho$ be as in \Lem~\ref{Lemma_simpleFirst}.
Clearly, the entropy $H(\rho)$ is maximized if $\rho=\frac1k\vecone$ is the uniform distribution.
The uniform distribution $\rho=\frac1k\vecone$ also happens to minimize  $\norm\rho_2^2$.
Therefore,
	(\ref{eqProp_simpleFirst1}) implies that for any probability distribution $\rho$ we have
	\begin{eqnarray}
		\frac1n\ln\Erw\brk{Z^\rho}&\leq&\ln k+\frac{d}{2}\ln (1-1/k)+O(\ln n/n)\label{eqProp_simpleFirst2},
	\end{eqnarray}
with equality in the case that $\norm{\rho-\frac1k\vecone}_\infty=O(n^{-1/2})$.
Since the number of possible distributions $\rho$ such that $\rho_in$ is an integer for all $i\in\brk k$ is bounded by $n^k$,
(\ref{eqProp_simpleFirst2}) implies that
	$$\frac1n\ln\Erw\brk{\Zkc}
		=\ln k+\frac{d}{2}\ln (1-1/k)+O(\ln n/n).$$
Furthermore, for $d\geq (2k-1)\ln k$ 
the elementary inequality $\ln(1-z)\leq-z-z^2/2-z^3/3$ yields
	\begin{eqnarray*}
		\frac1n\ln\Erw\brk{\Zkc}
			&\leq&\ln k-\frac d2\bc{\frac 1k+\frac1{2k^2}+\frac1{3k^3}}+O(\ln n/n)\nonumber\\
			&\leq&-\bc{\frac1{12k^2}-\frac1{6k^3}}\ln k+O(\ln n/n)<0,\label{eqProp_simpleFirst22}
	\end{eqnarray*}
as desired.
\qed\end{proof}

Due to \Cor~\ref{Prop_simpleFirst}, we may assume in the following that $d$ is the unique integer satisfying
	\begin{equation}\label{eqLowerdAssumption}
	d^+\leq d<(2k-1)\ln k.
	\end{equation}
\Cor~\ref{Prop_simpleFirst} shows that for this $d$, the first moment is
	\begin{eqnarray}\label{eqFirstMomentNaive}
	\frac1n\ln\Erw[\Zkc]&=&\ln k-\frac d2\bc{\frac1k+\frac1{2k^2}}+\tilde O_k(k^{-2})\leq\frac1{2k}-\frac3{2k\ln^{3/2}k}+\tilde O_k(k^{-2}).
	\end{eqnarray}
The fact that the right-hand side is positive is not an ``accident'':
indeed the first moment $\Erw[\Zkc]$ is generally exponentially large in $n$ for this $d$.
Therefore, the standard first moment argument does not suffice to prove that $\chi(\gnd)>k$ \whp

Instead, we develop an argument that takes the geometry of the set of $k$-colorings into account:
	we already saw that the $k$-colorings of $\Gnd$ come in clusters that of exponential size.
Roughly speaking, the volume of these clusters is what drives up the first moment, even though $\Gnd$ does not have a single $k$-coloring \whp\
To overcome this issue, we are going to perform a first-moment argument over the number of {\em clusters} rather than individual $k$-colorings.
To implement this idea, we need the following

\begin{definition}
Let $\sigma$ be a $k$-coloring of a multi-graph $\cG$ and let $p\in\brk{0,1}$.
\begin{enumerate}
\item A vertex $v$ is \bemph{rainbow} 
		if for any color $i\in\brk k\setminus\cbc{\sigma(v)}$ there is a neighbor $w$ of $v$ with $\sigma(w)=i$.
\item We call $\sigma$ \bemph{$p$-rainbow} if precisely $pn$ vertices are rainbow.
\end{enumerate}
\end{definition}

For two (not necessarily balanced) $k$-colorings $\sigma,\tau$ of $\Gnd$ we define the overlap $\rho(\sigma,\tau)$ just
as in~(\ref{eqOverlapMatrix}).
Similarly, we define the cluster 
	$$\cC^*(\sigma)=\cbc{\sigma:\mbox{$\sigma$ is a $k$-coloring with $\rho_{ii}(\sigma,\tau)>0.51$ for all $i\in\brk k$}}.$$
(The difference between $\cC(\sigma)$ as defined in~(\ref{XeqCluster}) and $\cC^*(\sigma)$ is that the former only contains {\em balanced} $k$-colorings.)
Furthermore, for $p\in\brk{0,1}$ let
	$$\Sigma^p=\abs{\cbc{\cC^*(\sigma):\mbox{$\sigma$ is a $p$-rainbow $k$-coloring of $\Gnd$}}}.$$
In words, $\Sigma^p$ is the number of {\em clusters} of $p$-rainbow $k$-colorings of $\Gnd$.
(Of course, the cluster $\cC^*(\sigma)$ may contain colorings $\tau\neq\sigma$ that are not $p$-rainbow.)

A priori, the definition of the cluster $\cC^*(\sigma)$ does not ensure that the clusters
of two colorings $\sigma,\tau$ are either disjoint or identical.
In order to enforce that this is indeed the case, we are going to show that we may confine ourselves to ``nice'' $k$-colorings
with certain additional properties.

\begin{definition}\label{Def_nice}
Let $\sigma$ be a $k$-coloring of $\Gnd$.
We call $\sigma$ \bemph{nice} if the following three conditions are satisfied.
\begin{enumerate}
\item Let $\rho=(\rho_i)_{i\in\brk k}$ be the vector with entries $\rho_i=|\sigma^{-1}(i)|/n$. Then
		\begin{eqnarray}\label{eqrhoNice}
		\norm{\rho-k^{-1}\vecone}_2<k^{-1}\ln^{-\frac13}k.
		\end{eqnarray}
\item Let $\mu=(\mu_{ij})_{i,j\in\brk k}$ be the matrix with entries
		$\mu_{ij}=e_{\Gnd}(\sigma^{-1}(i),\sigma^{-1}(j))/dn$.
		Moreover, let $\bar\mu=(\bar\mu_{ij})_{i,j\in\brk k}$ be the matrix with entries $\bar\mu_{ij}=\vecone_{i\neq j}k^{-1}(k-1)^{-1}$.
		Then
		\begin{eqnarray}\label{eqmuNice}
		\norm{\mu-\bar\mu}_2<8k^{-1}(k-1)^{-1}\ln^{-\frac13}k.
		\end{eqnarray}
\item If $\tau\in\cC^*(\sigma)$ is a $k$-coloring such that
		$$\abs{|\tau^{-1}(i)|-\frac nk}<\frac n{k(\ln k)^{1/3}}\quad\mbox{for all }i\in\brk k$$
		 then the overlap matrix satisfies $\rho_{ii}(\sigma,\tau)\geq0.9$ for all $i\in\brk k$.
\end{enumerate}
\end{definition}
Hence, in a nice coloring all the color classes have size about $n/k$, and the edge densities between different color classes
are approximately uniform.
Furthermore, the third condition ensures that for two nice colorings $\sigma,\tau$, the clusters $\cC^*(\sigma)$, $\cC^*(\tau)$
are either disjoint or identical.
Let $Z'$ be the  number of $k$-colorings of $\Gnd$ that fail to be nice.
In Section \ref{sec:prop:ExpctZMaxAiMuij} we are going to derive the following bound.

\begin{proposition}\label{prop:ExpctZMaxAiMuij}
We have $\frac{1}{n}\ln\Erw[Z']\leq -\frac14 k^{-1}\ln^{\frac13}k.$
\end{proposition}

Furthermore, in Section \ref{sec:prop:NoOfCompl+ItsDistr} we are going to establish the following proposition,
which effectively puts a lower bound on the cluster size of a nice $p$-rainbow $k$-coloring.
Let $\Lambda=\Lambda_k(\Gnd)$ denote the set of all nice $k$-colorings of $\Gnd$.
Moreover, let $\Delta=[1-\frac{20}{k}, 1-\frac{1}{20k}]$.

\begin{proposition}\label{prop:NoOfCompl+ItsDistr}
Let $p\in\brk{0,1}$.
For a $k$-coloring $\sigma$ of $\Gnd$ let 
 $C_{p,\sigma}$ denote the number of nice $p$-rainbow $k$-colorings in $\cC^*(\sigma)$.
Moreover, let $A_{p,\sigma}$ be the event that $\sigma$ is $p$-rainbow.
\begin{enumerate}
\item If $p\in\Delta$, then
 there exist numbers $p',q$ satisfying
	\begin{equation}\label{eqqthrm:C^p}
	p'=p+O_k(k^{-1}\ln^{-7/8}k)\mbox{ and }q=1-1/k+O_k(k^{-1}\ln^{-2}k)
	\end{equation}
	such that
	\begin{eqnarray*}\nonumber
	\frac{1}{n}\ln\sum_{\sigma\in [k]^{n}}\sum_{i\geq1}\frac{\pr[\sigma \in \Lambda,A_{p,\sigma},C_{p,\sigma}=i, Z'=0]}{i}\\
	 &\hspace{-6cm}   \leq &\hspace{-3cm}
	 	\ln k+\frac{d}{2}\ln(1-k^{-1})-D_{KL}(p',q)-(1-p)\ln 2+o(1).
	\end{eqnarray*}
\item If $p\not\in\Delta$, then
	 \begin{eqnarray*}\nonumber
	\frac{1}{n}\ln\sum_{\sigma\in [k]^{n}}\sum_{i\geq1}\frac{\pr[\sigma \in \Lambda,A_{p,\sigma},C_{p,\sigma}=i, Z'=0]}{i}
		&\leq&-\frac{1}{4k}.
	\end{eqnarray*}
\end{enumerate}
\end{proposition}

\noindent
Based on \Prop~\ref{prop:NoOfCompl+ItsDistr}, we obtain the following bound on $\Sigma^p$.

\begin{corollary}\label{thrm:C^p}
For any $p\in\Delta$
then there exist $p',q$ satisfying~(\ref{eqqthrm:C^p}) such that
\begin{eqnarray*}\nonumber
 \frac{1}{n}\ln\Erw[\Sigma^p\cdot\vecone_{Z'=0}]&\leq& \max\cbc{\ln k+\frac{d}{2}\ln \left(1-1/k\right)-\KL{p'}q-(1-p)\ln 2,-\frac{\ln^{1/3}k}{4k}}+o(1).
\end{eqnarray*}
Furthermore, if $p\not\in\Delta$, then
$\frac{1}{n}\ln\Erw[\Sigma^p\cdot\vecone_{Z'=0}]<-\frac{1}{4k}$.
\end{corollary}
\begin{proof} 
Let $C_{p,\sigma}$ and $A_{p,\sigma}$ be as in \Prop~\ref{prop:NoOfCompl+ItsDistr}.
Moreover, let $\Lambda_p=\Lambda\cap A_p$ denote the set of nice $p$-rainbow $k$-colorings of ${\cal G}_{n,d}$
and let
	$\Sigma^p_*=\abs{\cbc{\cC^*(\sigma):\sigma\in\Lambda_p}}.$
Since the clusters of two colorings $\sigma,\tau\in\Lambda_p$ are either disjoint or identical, by double-counting we have
	$$\Sigma^p_*=\sum_{\sigma\in \Lambda_p}\frac{1}{C_{p,\sigma}}.$$
Furthermore, because $\Sigma^p\cdot\vecone_{Z'=0}=\Sigma^p_*\cdot\vecone_{Z'=0}$, we find
\begin{eqnarray}\nonumber
\frac1n\ln\Erw[\Sigma^p\cdot\vecone_{Z'=0}]&=&
 \frac1n\ln\Erw[\Sigma^p_*\cdot\vecone_{Z'=0}]=
 	\sum_{\sigma\in [k]^{n}}\sum_{i\geq 1}\frac{\pr[C_{p,\sigma}=i,\;\sigma\in\cZ_p,Z'=0]}{i}\\
	&=&\sum_{\sigma\in [k]^{n}}\sum_{i\geq 1}\frac{\pr[\sigma\in\Lambda,A_{p,\sigma},C_{p,\sigma}=i,Z'=0]}{i}.
	\label{eq:InterECp}
\end{eqnarray}
Together with \Prop~\ref{prop:NoOfCompl+ItsDistr}, (\ref{eq:InterECp})  yields
	\begin{eqnarray*}
	\frac1n\ln\Erw[\Sigma^p\cdot\vecone_{Z'=0}]&\leq&\left\{
		\begin{array}{cl}
		 	\ln k+\frac{d}{2}\ln(1-\frac1k)-\KL{p'}q-(1-p)\ln 2+o(1)&\mbox{if }p\in\Delta,\\
			-\frac{1}{4k}&\mbox{otherwise,}
		\end{array}
		\right.
	\end{eqnarray*}
as claimed.
%
%
%
%
%
%
%
%
\qed\end{proof}

\medskip\noindent
\emph{Proof of \Thm~\ref{Thm_main}, part 2 (assuming \Prop s~\ref{prop:ExpctZMaxAiMuij} and~\ref{prop:NoOfCompl+ItsDistr}).}
We start by deriving an explicit estimate of the bound provided by \Cor~\ref{thrm:C^p}.
Assume that  $p\in\Delta$ and let $p',q$ be the numbers promised by \Cor~\ref{thrm:C^p}.
Let $c>0$ be such that $d=(2k-1)\ln k-c$.
Using the estimate $\ln(1-z)=-z-z^2/2+O(z^3)$, we see that
	\begin{eqnarray}\label{eqthrm:NonExistenceOfClusters1}
	\ln k+\frac d2\ln(1-1/k)&=&\ln k-\frac d2\bc{\frac 1k+\frac1{2k^2}+O_k(k^{-3})}=\frac c{2k}+\tilde O_k(k^{-2}).
	\end{eqnarray}
To proceed, consider the function
	$h(p'',q'')=-D_{KL}(p'',q'')+(1-p'')\ln 2.$
Then
	\begin{eqnarray*}
	\frac{\partial h}{\partial p''}&=&-\ln\frac{p''}{q''}+\ln\frac{1-p''}{1-q''}+\ln 2,\quad
	\frac{\partial^2 h}{\partial {p''}^2}=-\frac{1}{p''(1-p'')}.
	\end{eqnarray*}	
Hence, for a fixed value of $q$ the function $h(p'',q)$ attains its maximum at $p''=p_0=\frac{2q}{1+q}$.
Consequently, for any fixed $q$ satisfying~(\ref{eqqthrm:C^p}) we have
	\begin{eqnarray*}
	h(p',q)&\leq&h(p_0,q)\leq
		-p_0\ln\frac{p_0}{q}-(1-p_0)\ln \frac{1-p_0}{1-q}-(1-p_0)\ln 2\\
 &=&-\frac{2q}{1+q}\ln \frac{2}{1+q}-\frac{1-q}{1+q}\ln \frac{1}{1+q}-\frac{1-q}{1+q}\ln 2 \hspace*{1.4cm}\left[\textrm{as } 1-p_0=\frac{1-q}{1+q}\right]\\
		 &=&\frac{1}{1+q}\left(-2q\ln \frac{2}{1+q}-(1-q)\ln \frac{2}{1+q}\right)\\
 &=&\frac{1}{1+q}\left(2q\ln \left(1-\frac{1-q}{2}\right) +(1-q)\ln \left(1-\frac{1-q}{2}\right)\right)\hspace*{0.2cm}\left[\textrm{as } 1+q=2\left(1-\frac{1-q}{2}\right)\right]\\
 &=&\ln \left(1-\frac{1-q}{2}\right)\leq\frac{1-q}{2}\leq-\frac1{2k}+O_k(k^{-1}\ln^{-2}k).
	\hspace*{0.7cm}\mbox{[as $\ln(1-z)\leq -z$ and due to~(\ref{eqqthrm:C^p})]}.
	\end{eqnarray*}
Combining this last estimate with~(\ref{eqthrm:NonExistenceOfClusters1}) and using \Cor~\ref{thrm:C^p}, we see that for $p\in\Delta$
	\begin{eqnarray}\label{eqthrm:NonExistenceOfClusters2}
	\frac1n\ln\Erw\brk{\Sigma^p\cdot\vecone_{Z'=0}}&\leq&\max\cbc{\frac{c-1}{2k}+\frac{25}{k\ln^2k},-\frac{\ln^{1/3}k}{4k}}+o(1).
	\end{eqnarray}
Our assumption~(\ref{eqLowerdAssumption}) on $d$ ensures that
	$c=(2k-1)\ln k-d<1-3\ln^{-3/2} k$,
whence (\ref{eqthrm:NonExistenceOfClusters2}) implies for $p\in\Delta$ that
	\begin{eqnarray}\label{eqthrm:NonExistenceOfClusters3}
	\frac1n\ln\Erw\brk{\Sigma^p\cdot\vecone_{Z'=0}}&\leq&-k^{-1}\ln^{-3/2}k.
	\end{eqnarray}
Further, \Cor~\ref{thrm:C^p} readily states that $\frac1n\ln\Erw\brk{\Sigma^p\cdot\vecone_{Z'=0}}<-\frac{1}{4k}$ for all $p\not\in\Delta$.
Hence, (\ref{eqthrm:NonExistenceOfClusters3}) does in fact hold for all $p\in\brk{0,1}$.

To complete the proof, consider the random variable
	$\cZ=Z'+\sum_{s=0}^n\vecone_{Z'=0}\cdot\Sigma^{s/n}.$
If $\Gnd$ is $k$-colorable, then $\cZ\geq1$, because either there is a $k$-coloring that is not nice (in which case $Z'\geq1$),
or $\Sigma^{s/n}\geq1$ for some $s$.
Since \Prop~\ref{prop:ExpctZMaxAiMuij} and~(\ref{eqthrm:NonExistenceOfClusters3}) imply that $\Erw\brk\cZ=\exp(-\Omega(n))$,
Markov's inequality entails that $\chi(\Gnd)>k$ \whp
\qed






\subsection{Proof of Proposition \ref{prop:ExpctZMaxAiMuij}}\label{sec:prop:ExpctZMaxAiMuij}


For a probability distribution $\rho=(\rho_1,\ldots,\rho_k)$ on $\brk k$ let $Z^\rho$ denote the number of $k$-colorings $\sigma$
of $\Gnd$ such that $|\sigma^{-1}(i)|=\rho_i n$ for all $i\in\brk k$.

\begin{lemma}\label{lemma:max-ai}
Let $\eps_k=k^{-1}\ln^{-1/3} k$ and let $\rho$ be  such that
$\norm{\rho-\frac1k\vecone}_2>\eps_k$.
Then $\frac1n\ln\Erw[Z^\rho]\leq -\frac{\ln^{1/3} k}{3k}.$
\end{lemma}
\begin{proof}
Let $\bar\rho$ be a probability distribution such that $\norm{\bar\rho-\frac1k\vecone}_\infty=O(n^{-1})$ and such
that $\bar\rho_in$ is an integer for all $i\in\brk k$.
Because the entropy function attains its global maximum at $\frac1k\vecone$,
\Lem~\ref{Lemma_simpleFirst} yields
	\begin{eqnarray}\nonumber
	\frac1n\ln\Erw\brk{Z_\rho}-\frac1n\ln\Erw\brk{Z_{\bar\rho}}&=&
		H(\rho)-H(\bar\rho)+\frac d2\brk{\ln(1-\norm\rho_2^2)-\ln(1-1/k)}+O(\ln n/n)\\
		&\leq&\frac d2\brk{\ln(1-\norm\rho_2^2)-\ln(1-1/k)}+O(\ln n/n).
			\label{eqlemma:max-ai1}
	\end{eqnarray}
To bound this expression, we compute the first two derivatives of the function
$\rho\mapsto\frac d2\ln(1-\norm\rho_2^2)$: for $i,j\in\brk k$, $i\neq j$ we find
	\begin{eqnarray*}
	\frac{\partial}{\partial\rho_i}\ln(1-\norm\rho_2^2)&=&-\frac{2\rho_i}{1-\norm\rho_2^2},\\
	\frac{\partial^2}{\partial^2\rho_i}\ln(1-\norm\rho_2^2)&=&-\frac2{1-\norm\rho_2^2}-\frac{4\rho_i^2}{(1-\norm\rho_2^2)^2},\\
	\frac{\partial^2}{\partial\rho_i\partial\rho_j}\ln(1-\norm\rho_2^2)&=&-\frac{4\rho_i\rho_j}{(1-\norm\rho_2^2)^2}.
	\end{eqnarray*}
Because the rank one matrix $(4\rho_i\rho_j/(1-\norm\rho_2^2))_{i,j\in\brk k}$ is positive semidefinite,
all eigenvalues of the Hessian
	$(\frac{\partial^2}{\partial\rho_i\partial\rho_j}\ln(1-\norm\rho_2^2))_{i,j\in\brk k}$
 are bounded by $-2/(1-\norm\rho_2^2)<-2$.
Therefore,  together with Taylor's formula, (\ref{eqlemma:max-ai1}) yields 
	\begin{eqnarray*}\nonumber
	\frac1n\ln\Erw\brk{Z_\rho}&\leq&\frac1n\ln\Erw\brk{Z_{\bar\rho}}-\frac d2\norm{\rho-\bar\rho}_2^2+O(\ln n/n)
		\leq	\frac1n\ln\Erw\brk{\Zkc}-\frac d2\norm{\rho-\bar\rho}_2^2+O(\ln n/n)	\nonumber\\
		&\leq&\frac1{2k}-\frac d2\norm{\rho-\bar\rho}_2^2+O(\ln n/n)\qquad\mbox{[due to~(\ref{eqFirstMomentNaive})]},\label{eqFirstMomentNotSoNaive}
	\end{eqnarray*}
whence the assertion is immediate. 
\qed
\end{proof}

Let $\rho$ be a probability distribution on $\brk k$ and let $\mu$ be a probability distribution on $\brk k\times\brk k$ such that
$(\rho,\mu)$ is $(d,n)$-admissible.
Let $Z_{\rho,\mu}$ be the number of $k$-colorings $\sigma$ of $\Gnd$ such that
$|\sigma^{-1}(i)|=\rho_in$ and $e_{\Gnd}(\sigma^{-1}(i),\sigma^{-1}(j))=dn\mu_{ij}$ for all $i,j\in\brk k$.
In addition, let $\bar\mu=(\bar\mu_{ij})_{i,j\in\brk k}$ be the probability distribution defined by $\bar\mu_{ij}=\vecone_{i\neq j}k^{-1}(k-1)^{-1}$.

\begin{lemma}\label{Lemma_muFreaks}
With $\eps_k=8/({k(k-1)\ln^{\frac 13}k})$
assume that $\norm{\rho-\frac1k\vecone}_2\leq k^{-1}\ln^{-\frac 13} k$ but $\norm{\mu-\bar\mu}_2>\eps_k$.
%
Then $$\frac1n\ln\Erw[Z_{\rho,\mu}]\leq-\frac14k^{-1}\ln^{1/3}k.$$
\end{lemma}
\begin{proof}
Let $\hat\rho=(\hat\rho_{ij})_{i,j\in\brk k}$ be the probability distribution with
	$\hat\rho_{ij}=\frac{\vecone_{i\neq j}\cdot\rho_i\rho_j}{1-\norm\rho_2^2}.$
Then by Corollaries~\ref{Cor:Skewed} and~\ref{Prop_simpleFirst} we have
	\begin{eqnarray}
	\frac1n\ln\Erw[Z_{\rho,\mu}]&=&
		H(\rho)+\frac{d}{2}\ln (1-\norm\rho_2^2)-\frac{d}2\KL{\mu}{\hat\rho}+O(\ln n/n)\nonumber\\
		&\leq&\frac1n\ln\Erw[\Zkc]-\frac{d}2\KL{\mu}{\hat\rho}+O(\ln n/n)
		\leq\frac1{2k}-\frac{d}2\KL{\mu}{\hat\rho}+O(\ln n/n).\qquad
			\label{eqIntTaylor1}
	\end{eqnarray}

By Fact~\ref{Fact_KL} the function $\mu\mapsto\KL{\mu}{\hat\rho}$ takes its minimum value (namely, zero) at $\mu=\hat\rho$.
Recalling its differentials from (\ref{eqKLdiff1}), (\ref{eqKLdiff2}),
we see that the Hessian $(\frac{\partial^2}{\partial\mu_{ij}\partial\mu_{st}}\KL{\mu}{\hat\rho})_{i,j,s,t\in\brk k:i\neq j,s\neq t}$ is a positive-definite diagonal matrix
with diagonal entries $1/\mu_{ij}$ ($i\neq j$).

Because $\norm{\rho-\frac1k\vecone}_2\leq k^{-1}(\ln k)^{-1/3}$ we have
	$\norm{\hat\rho-\bar\mu}_2\leq\eps_k/2$.
Consequently, our assumption $\norm{\mu-\bar\mu}_2>\eps_k$ implies that
 $\norm{\mu-\hat\rho}_2>\eps_k/2$.
In fact, let $a\in\brk{0,1}$ be such that $\hat\mu=a\mu+(1-a)\hat\rho$ is at $\ell^2$-distance exactly $\eps_k/2$ from $\hat\rho$.
Then due to the convexity of the Kullback-Leibler divergence (Fact~\ref{Fact_KL}), we have $\KL{\mu}{\hat\rho}\geq\KL{\hat\mu}{\hat\rho}$.
Furthermore,  because $\norm{\hat\mu-\hat\rho}_2=\eps_k/2$, we have $\hat\mu_{ij}\leq 2/k^2$ for all $i,j\in\brk k$, $i\neq j$.
Therefore, Taylor's formula implies together with the above analysis of the Hessian of $\KL\cdot{\hat\rho}$ that
	\begin{equation}		\label{eqIntTaylor2}
	\KL{\mu}{\hat\rho}\geq\KL{\hat\mu}{\hat\rho}\geq\frac{k^2}4\norm{\hat\mu-\hat\rho}_2^2=\frac{k^2\eps_k^2}{16}.
	\end{equation}
Plugging~(\ref{eqIntTaylor2}) into~(\ref{eqIntTaylor1}), we see that for any $\mu$ such that $\norm{\mu-\hat\rho}_2>\eps_k$,
	\begin{eqnarray*}
	\frac1n\ln\Erw[Z_{\rho,\mu}]		
		&\leq&\frac1{2k}-\frac{dk^2\eps_k^2}{32}+O(\ln n/n)
		\leq-\frac{\ln^{1/3}k}{k}\qquad\mbox{[as $d\geq1.9k\ln k$]},
			\label{eqIntTaylor3}
	\end{eqnarray*}
thereby completing the proof.
\qed\end{proof}

\Lem s~\ref{lemma:max-ai} and~\ref{Lemma_muFreaks} put a bound on the expected number of $k$-colorings of $\Gnd$
that violate the first two conditions in \Def~\ref{Def_nice}.
To estimate the number of colorings for which the third condition is violated, we need to establish a similar
statement as \Lem~\ref{lem_s_less}, albeit under significantly weaker assumptions.
In particular, we need to work with the ``planted coloring model'' $\cG(\sigma,\mu)$ from \Sec~\ref{Sec_Prop_goodFirstMoment}.
The following statement is reminiscent of \Lem~\ref{lem_s_less};
	the difference is that here we make weaker assumptions as to the ``balancedness'' of the coloring,
	while also aiming at a weaker conclusion.

\begin{lemma}\label{lem_s_less_Lower}
Let $(\rho,\mu)$ be $(d,n)$-admissible and assume that for all $i,j\in\brk k$, $i\neq j$ we have
	\begin{equation}\label{eqlem_s_less_LowerAssumption}
	|\rho_i-1/k|\leq k^{-1}\ln^{-1/3}k,\quad|\mu_{ij}-k^{-1}(k-1)^{-1}|\leq 8/(k(k-1)\ln^{1/3}k).
	\end{equation}
Let $i \in [k]$ and let $0.509 \leq \alpha \leq 0.99$.
Then in $\cG(\sigma,\mu)$ with probability $1-\exp(-n\Omega_k(\ln k/k))$ the following is true.
\begin{equation}\label{prop_s_less_Lower}
\parbox{12cm}{
For any set $S \subset V_i$ of size $|S| = \alpha n/k$
the number of vertices $v \in V \setminus V_i$ that do not have a neighbor in $S$ is less than $\frac nk(1-\alpha-\ln^{-1/4}k)$.}
\end{equation}
\end{lemma}
\begin{proof}
As in the proof of \Lem~\ref{lem_s_less}, we assume $i=1$, fix a set
 $S\subset V_1$ of size $|S|=\alpha n/k$, and let
	$$e_{j,S}=\abs{\cbc{(v,l)\in S\times\brk d:\vec\Gamma_{\sigma,\mu}(v,l)\in V_j\times\brk d}}.$$
Let $p_j=\mu_{1j}/\rho_j$.
Then~(\ref{eqlem_s_less_LowerAssumption}) ensures that $p_j=(1-o_k(1))/k$.
Let $\hat e_{j,S}$ be a $\Bin(|S| d,p_j)$ random variable.
Setting $\delta=10^{-4}$, we obtain from \Lem~\ref{lem_balls_bins_cap}
and the Chernoff bound (Lemma~\ref{Lemma_Chernoff})
	\begin{eqnarray}\nonumber
	\pr\brk{e_{j,S}<\frac{(1-\delta)d|S|}{k-1}}&\leq&O(\sqrt n)\cdot\pr\brk{\hat e_{j,S}<\frac{(1-\delta)d|S|}{k-1}}\\
		&\leq&O(\sqrt n)\exp\brk{-\frac{\delta^2 d|S|}{3(k-1)}}\leq\exp(-n\cdot\Omega_k(\ln k/k)).
			\label{eqEstimateOutgoingEdges2LOWER}
	\end{eqnarray}
Let $\cE_S$ be the event that $e_{j,S}\geq\frac{(1-\delta)d|S|}{k-1}$ for all $j=2,\ldots,k$.
Taking a union bound over all $\leq2^{n/k}$ possible sets $S$ and all $k-1$ colors $j$, 
we obtain from~(\ref{eqEstimateOutgoingEdges2LOWER})
	\begin{eqnarray}\label{eqEstimateOutgoingEdges3LOWER}
	\pr\brk{\exists S:\cE_S\mbox{ does not occur}}\leq (k-1)2^{n/k}\exp(-n\cdot\Omega_k(\ln k/k))\leq-\exp(-n\cdot\Omega_k(\ln k/k)).
	\end{eqnarray}

Conditioning on $\cE_S$,
let $X_{j,S}$ be the number of vertices in $v\in V_j$ that do not have a neighbor in $S$.
Using \Lem~\ref{lem_balls_bins_cap} (the binomial approximation to the hypergeometric distribution),
we can approximate $X_{j,S}$ by a binomial random variable $\hat X_{j,S}=\Bin(\rho_jn,q_j)$,
where
	\begin{eqnarray}\nonumber
	q_j&=&\pr\brk{\Bin\bc{d,\frac{e_{j,S}}{dn\rho_j}}=0}\leq
		\bc{1-\frac{e_{j,S}}{dn\rho_j}}^d\leq 
		\exp\brk{-\frac{(1-\delta)\alpha d}{k-1}}\quad\mbox{[as $e_{j,S}\geq\frac{1-\delta}{k-1}d|S|$]}\\
		&\leq&k^{-2\alpha(1-2\delta)}.
		\label{eqEstimateOutgoingEdges4LOWERa}
	\end{eqnarray}
More precisely, \Lem~\ref{lem_balls_bins_cap} yields
	\begin{eqnarray}\label{eqEstimateOutgoingEdges4LOWER}
	\pr\brk{X_{j,S}\geq t|\cE_S}\leq O(\sqrt n)\pr\brk{\hat X_{j,S}\geq t}\qquad\mbox{for any $t>0$}.
	\end{eqnarray}
Setting $q=k^{-2\alpha(1-2\delta)}$, $\hat X_S=\Bin((1-1/k)n,q)$, and $X_S=\sum_{j=2}^k X_{j,S}$,
we obtain from~(\ref{eqEstimateOutgoingEdges4LOWERa}) and~(\ref{eqEstimateOutgoingEdges4LOWER})
	\begin{eqnarray}\label{eqEstimateOutgoingEdges6LOWER}
	\pr\brk{X_{S}\geq t|\cE_S}\leq O(\sqrt n)\pr\brk{\hat X_{S}\geq t}\qquad\mbox{for any $t>0$}.
	\end{eqnarray}

Let $\alpha'=\alpha+\ln^{-1/4}k$.
By (\ref{eqEstimateOutgoingEdges6LOWER}) and the Chernoff bound,
 \begin{eqnarray}\nonumber
 \pr\brk{X_{S}\geq \frac nk(1-\alpha')|\cE_S}&\leq&
 O(\sqrt n)\pr\left[\hat X_S \geq \frac nk(1-\alpha')\right]\\
& \leq& \exp\left[-\frac nk(1-\alpha'+o(1))\ln\left(\frac{1-\alpha'}{\eul kq}\right)\right]. 
\label{equ_chern_bound_1aLOWER}
 \end{eqnarray}
Further, we take the union bound over
all $\bink{\rho_1 n}{(1-\alpha)\rho_1 n}\leq\exp(\rho_1n(1-\alpha)(1-\ln(1-\alpha)))$ ways to choose the 
set $S$: from~(\ref{equ_chern_bound_1aLOWER}) we obtain
\begin{equation} 	\label{equ_chern_bound_1bLOWER}
 \frac kn\ln\pr\left[\exists S:X_S\geq \frac nk(1-\alpha'),\cE_S\right]\leq
 	(1-\alpha)(1-\ln(1-\alpha))-(1-\alpha')\ln\frac{1-\alpha'}{\eul kq}+o(1).
\end{equation}
Because the function $z\in\brk{0,1}\mapsto -z\ln z$ is bounded, (\ref{equ_chern_bound_1bLOWER}) yields
\begin{eqnarray} 
 \frac kn\ln\pr\left[\exists S:X_S\geq \frac nk(1-\alpha'),\cE_S\right]&\leq&
 	O_k(1)+(1-\alpha')\ln(kq)\nonumber\\
	&\leq&O_k(1)+(1-2\alpha(1-2\delta))(1-\alpha')\ln k.	\label{equ_chern_bound_1bbLOWER}
\end{eqnarray}
Finally, because $0.509\leq\alpha\leq 0.99$ and $\delta=10^{-4}$, we see that $2\alpha(1-2\delta)\geq1.001$.
Hence, (\ref{equ_chern_bound_1bbLOWER}) implies
	\begin{equation}\label{equ_chern_bound_1bbbLOWER}
	\frac 1n\ln\pr\left[\exists S:X_S\geq \frac nk(1-\alpha'),\cE_S\right]\leq-\Omega_k(\ln k/k)n.
	\end{equation}
The assertion follows from~(\ref{eqEstimateOutgoingEdges3LOWER}) and~(\ref{equ_chern_bound_1bbbLOWER}).
\qed\end{proof}

\medskip\noindent{\em Proof of Proposition \ref{prop:ExpctZMaxAiMuij}.}
\Lem s~\ref{lemma:max-ai} and~\ref{Lemma_muFreaks} readily imply the desired bound on
the expected number of colorings that violate the first or the second conditions in \Def~\ref{Def_nice}.
With respect to the third condition, let $(\rho,\mu)$ be an admissible pair
that satisfies~(\ref{eqlem_s_less_LowerAssumption}) and let $Z_{\rho,\mu}''$ be the number of $k$-colorings $\sigma$  
such that $\sigma^{-1}(i)=\rho_in$ and $e_{\Gnd}(\sigma^{-1}(i),\sigma^{-1}(j))=dn\mu_{ij}$ for all $i,j\in\brk k$
that violate~(\ref{prop_s_less_Lower}) for some $0.509\leq\alpha\leq0.99$.
We claim that
	\begin{equation}\label{eqprop:ExpctZMaxAiMuij1}
	\frac1n\ln\Erw\brk{Z_{\rho,\mu}''}\leq-\Omega_k(\ln k/k).
	\end{equation}
Indeed, by (\ref{eqFirstMomentNaive}) the total number $Z_{\rho,\mu}$ of $k$-colorings
such that $\sigma^{-1}(i)=\rho_in$ and $e_{\Gnd}(\sigma^{-1}(i),\sigma^{-1}(j))=dn\mu_{ij}$ for all $i,j\in\brk k$ satisfies
	\begin{equation}\label{eqprop:ExpctZMaxAiMuij2}
	\frac1n\ln\Erw\brk{Z_{\rho,\mu}}\leq\frac1n\ln\Erw\brk{\Zkc}=O_k(k^{-1}).
	\end{equation}
Furthermore, if $\sigma:V\ra\brk k$ is such that $|\sigma^{-1}(i)|=\rho_in$ for all $i\in\brk k$,
then $\cG(\sigma,\mu)$ is nothing but the {\em conditional} distribution of the random graph $\Gnd$ given that 
	$e_{\Gnd}(\sigma^{-1}(i),\sigma^{-1}(j))=dn\mu_{ij}$.
Thus, \Lem~\ref{lem_s_less_Lower} shows that for any such $\sigma$,
	\begin{equation}\label{eqprop:ExpctZMaxAiMuij3}
	\frac1n\ln\pr\brk{\mbox{(\ref{prop_s_less_Lower}) is violated}|e_{\Gnd}(\sigma^{-1}(i),\sigma^{-1}(j))=dn\mu_{ij}\mbox{ for all }i,j\in\brk k}
		\leq-\Omega_k(\ln k/k).
	\end{equation}
Combining~(\ref{eqprop:ExpctZMaxAiMuij2}) and~(\ref{eqprop:ExpctZMaxAiMuij3}) and using the linearity of expectation, we obtain~(\ref{eqprop:ExpctZMaxAiMuij1}).

Finally, assume that $\sigma:V\ra\brk k$ has the property~(\ref{prop_s_less_Lower
		}).
Let $\tau:V\ra\brk k$ be another coloring that satisfies conditions 1.\ and 2.\ in \Def~\ref{Def_nice} and assume that $\tau\in\cC^*(\sigma)$.
Let $i\in\brk k$ and consider the sets $S=\sigma^{-1}(i)\cap\tau^{-1}(i)$ and $T=\tau^{-1}(i)\setminus\sigma^{-1}(i)$.
Because both $\sigma,\tau$ satisfy condition 1.\ in \Def~\ref{Def_nice}, we have $|S|\geq0.509\frac nk$.
For the same reason, the set $T$ satisfies $$|T|\geq\frac nk-|S|-O_k(k^{-1}\ln^{-1/3}k)n>\frac nk-|S|-\frac nk\ln^{-1/4}k.$$
Hence, (\ref{prop_s_less_Lower}) implies that $\frac nk\rho_{ii}(\sigma,\tau)=\abs S>0.99\frac nk$.
Thus, $\sigma$ satisfies the third condition in \Def~\ref{Def_nice}.
Therefore, the assertion follows from~(\ref{eqprop:ExpctZMaxAiMuij1}).
\qed

\section{Lower-bounding the cluster size}\label{sec:prop:NoOfCompl+ItsDistr}

{\em Throughout this section we keep the notation and the assumptions from \Sec~\ref{Sec_lower_Outline}.}

\subsection{Outline}

The aim in this section is to prove \Prop~\ref{prop:NoOfCompl+ItsDistr}.
Essentially this means that we need to establish a lower bound on the size of the cluster $\cC(\sigma)$ of the nice $p$-rainbow $k$-coloring $\sigma$
	(because the cluster size occurs in the denominator of the sums that we are interested in).
Roughly speaking, we are going to show  that almost all the vertices that fail to be rainbow have precisely two colors to choose from,
and that these color choices can be made nearly independently.
In effect, it is going to emerge that for a $p$-rainbow coloring the cluster size is about $(1-p)\ln2$.
The proof of \Prop~\ref{prop:NoOfCompl+ItsDistr} is going to be more or less immediate from this estimate.
Technically, the implementation of this strategy requires a bit of work
because we need to get a rather precise handle on the probability of certain ``rare events''.
That is, we need to perform some large deviations analysis relatively accurately.

More precisely, 
let us fix a probability distribution $\rho$ on $\brk k$ that satisfies the first condition~(\ref{eqrhoNice}) in the definition of ``nice''
along with a map $\sigma:V\ra\brk k$ such that $|\sigma^{-1}(i)|=\rho_i n$ for all $i\in\brk k$.
Let $p\in\brk{0,1}$ and consider the sum
	$$\psi_p\bc\sigma= \sum_{i\geq1}\frac{1}{i}\pr[A_{p,\sigma},C_{p,\sigma}=i,Z'=0|\sigma \in \Lambda].$$
(Recall from \Prop~\ref{prop:NoOfCompl+ItsDistr} that $C_{p,\sigma}$ is the number of nice $p$-rainbow $k$-colorings in the cluster $\cC^*(\sigma)$,
that $A_{p,\sigma}$ is the event that $\sigma$ is $p$-rainbow, that $\Lambda$ is the event that $\sigma$ is a nice $k$-coloring of $\Gnd$,
and that $Z'$ is the number of $k$-colorings that fail to be nice.)
Set	
	$$y_p=n(1-p)(1-\ln^{-\frac13} k).$$
Then
\begin{eqnarray}
 \psi_p(\sigma)&\leq&
	2^{-y_p}\pr[C_{p,\sigma}\geq 2^{y_p}, A_{p,\sigma}|\sigma\in \Lambda]+\pr[C_{p,\sigma}<2^{y_p}, A_{p,\sigma},Z'=0|\sigma\in \Lambda] \nonumber\\
  &\leq& 2^{-y_p}\pr[A_{p,\sigma}|\sigma \in \Lambda]+\pr[C_{p,\sigma}<2^{y_p}, A_{p,\sigma},Z'=0|\sigma\in \Lambda].\label{eq:RedSum2TwoProbs}
\end{eqnarray}

To estimate the two probabilities on the right hand side,
we need to get a handle on the number of rainbow vertices and on the cluster size.
To this end, we say that a vertex $v$ is \bemph{$i$-vacant} in $\Gnd$ with respect to $\sigma$
if $\sigma(v)\neq i$ and if $v$ does not have a neighbor in $\sigma^{-1}(i)$.
Furthermore, in order to deal with the conditioning on the event that $\sigma\in\Lambda$, we are going to work with the
random multi-graph $\cG(\sigma,\mu)$ with $\mu$ a probability distribution on $\brk k\times\brk k$ in which $\sigma$ is a ``planted'' $k$-coloring.


\begin{proposition}\label{Prop_vacant}
Let $\mu$ is a probability distribution on $\brk k\times\brk k$ that satisfies condition~(\ref{eqmuNice}) such that $(\rho,\mu)$ is $(d,n)$-admissible.
Then in the random multi-graph $\cG(\sigma,\mu)$ the following statements are true.
\begin{enumerate}
\item There exist $p',q$ satisfying~(\ref{eqqthrm:C^p}) such that
	\begin{eqnarray}\label{eq:target4NoOfComplete}
	 \pr[A_{p,\sigma}]\leq \exp\left[-\min\cbc{\KL{p'}q,\Omega_k(\ln^{1/8}k/k)} n+O(\ln n)\right].
	\end{eqnarray}
\item Let $\cV^*$ be the set of vertices $v$ such that there exist $1\leq j<j'\leq k$ such that $v$ is both $j$-vacant and $j'$-vacant.
	Then
	$$
	 \pr\brk{|\cV^*|>\frac{n}{k\ln^{3/4}k}}\leq \exp\left[-n\cdot\Omega_k(\ln^{1/9} k/k)\right].
	$$
\item Let $\cV_{ij}$ be the set of $j$-vacant $v\in \sigma^{-1}(i)$ and 	
	$\hat\cV=\sum_{i,j\in\brk k}|\cV_{ij}|\cdot\vecone_{|\cV_{ij}|>n/k^{2.9}}$. Then
	$$\pr\brk{\hat\cV>\frac{2n}{k\ln^{3/4}k}}\leq \exp\left[-n\Omega_k(\ln^{1/9} k/k)\right].$$
\end{enumerate}
\end{proposition}

We defer the proof of \Prop~\ref{Prop_vacant} to \Sec~\ref{Sec_vacant}.
In addition, in \Sec~\ref{Sec_spanned} we are going to prove that the $j$-vacant vertices do not
span a lot of edges \whp\
More precisely, we have

\begin{proposition}\label{Prop_spanned}
With the notation and assumptions of \Prop~\ref{Prop_vacant},
let $\cV_{ij}'=\cV_{ij}\setminus\cV^*$ if $|\cV_{ij}|\leq n/k^{2.9}$, while $\cV_{ij}'=\emptyset$ otherwise.
Moreover, for each $j\in\brk k$ let $E_j$ be the number of edges 
spanned by $\bigcup_{i\in\brk k}\cV_{ij}'$ and set $E=\sum_{j\in\brk k}E_j$.
Then in $\cG(\sigma,\mu)$ we have
	$$
	\pr\brk{E>\frac n{k\ln^{4/5}k}}\leq\exp\brk{-\Omega_k(\ln^{1/9}k/k)n}.
	$$
\end{proposition}

%
%

\medskip\noindent
{\em Proof of \Prop~\ref{prop:NoOfCompl+ItsDistr}.}
Given $\rho,\sigma$, let $M$ be the set of all probability distributions $\mu$ 
on $\brk k\times\brk k$ that satisfy~(\ref{eqmuNice}) such that $(\rho,\mu)$ is $(d,n)$-admissible.
Then Bayes' formula shows that
	\begin{eqnarray}\label{eqBayes}
	\pr_{\Gnd}[A_{p,\sigma},C_{p,\sigma}=i, Z'=0|\sigma \in \Lambda]&\leq&(1+o(1))\max_{\mu\in M}\pr_{\cG(\sigma,\mu)}[A_{p,\sigma},C_{p,\sigma}=i, Z'=0].
	\end{eqnarray}

First assume that $p\in\Delta=[1-\frac{20}k,1-\frac1{20k}].$
With $(\cV_{ij}')_{i\neq j}$ the sets from \Prop~\ref{Prop_spanned}, let $\cV'=\bigcup_{i\neq j}\cV_{ij}'$.
To get an intuitive picture, observe that ${\cal V}'$ contains all the sets ${\cal V}_{ij}$ which are not
``exceedingly'' large.  Let $B_{p,\sigma}$ be the event that either $\abs{\cV'}<(1-p)(1-\ln^{-2/3}k)n$ or $E>nk^{-1}\ln^{-4/5}k$.
Then \Prop s~\ref{Prop_vacant} and~\ref{Prop_spanned} and~(\ref{eqBayes}) imply
	\begin{equation}\label{eqprop:NoOfCompl+ItsDistr1}
	\pr\brk{A_{p,\sigma},B_{p,\sigma}|\sigma\in\Lambda}\leq\exp(-\Omega_k(\ln^{1/9}k/k)n).
	\end{equation}	
Suppose that $\sigma\in\Lambda$ and that the event $A_{p,\sigma}$ occurs but $B_{p,\sigma}$ does not.
Let $\cV''$ be the set of vertices $v\in\cV'$ that are $j$-vacant and that are not adjacent to any other $j$-vacant vertex in $\cV'$  for some $j\in\brk k$.
Because $B_{p,\sigma}$ does not occur, we have $|\cV''|\geq(1-p)(1-2\ln^{-2/3}k)n$.
Furthermore, for any subset $S\subset\cV''$ there exists a $k$-coloring $\tau$ such that $\tau(v)\neq\sigma(v)$ for all $v\in S$ and $\tau(v)=\sigma(v)$ for all $v\in V\setminus S$.
More precisely, since every vertex $v\in S$ is $j$-vacant for some $j\neq\sigma(v)$, we can set $\tau(v)=j$.
This yields a $k$-coloring because, by the construction of $\cV''$, no two vertices in $S$ that receive color $j$ under $\tau$ are adjacent.
Let $\cC_*(\sigma)$ denote the set of colorings $\tau$ that can be obtained in this way.

Because $|\cV_{ij}'|\leq n/k^{2.9}$ for all $i,j\in\brk k$, we have $\cC_*(\sigma)\subset\cC^*(\sigma)$.
In addition, any coloring $\tau\in\cC_*(\sigma)$ has the exact same rainbow vertices as $\sigma$ does, and thus their number is $pn$.
Hence, if all the colorings in $\cC_*(\sigma)$ are nice, we have $C_{p,\sigma}\geq|\cC_*(\sigma)|\geq 2^{y_p}$.
Therefore, (\ref{eqprop:NoOfCompl+ItsDistr1}) yields
	\begin{eqnarray}
	\pr\brk{C_{p,\sigma}<2^{y_p},A_{p,\sigma},Z'=0|\sigma\in\Lambda}&\leq&\pr\brk{A_{p,\sigma},B_{p,\sigma}|\sigma\in\Lambda}
		\leq\exp(-\Omega_k(\ln^{1/9}k/k)n).\label{eqprop:NoOfCompl+ItsDistr2}
	\end{eqnarray}
Together with~(\ref{eq:RedSum2TwoProbs}) and~(\ref{eq:target4NoOfComplete}), (\ref{eqprop:NoOfCompl+ItsDistr2}) implies
	\begin{equation}\label{eqprop:NoOfCompl+ItsDistr3}
	\psi_p(\sigma)\leq\exp\left[-\KL{p'}q n+O(\ln n)\right]\quad\mbox{[as $\KL{p'}q=\Theta_k(1/k)$ for $p\in\Delta$].}
	\end{equation}
Summing~(\ref{eqprop:NoOfCompl+ItsDistr3}) over all $\sigma$ and using the bound~(\ref{eqFirstMomentNaive}) on the expected number
yields the first part of \Prop~\ref{prop:NoOfCompl+ItsDistr}.

Finally, if $p\not\in\Delta$, then for any $p',q$ satisfying~(\ref{eqqthrm:C^p}) we have $\KL{p'}q\geq0.94/k$.
Therefore, the first part of \Prop~\ref{Prop_vacant} implies together with (\ref{eqFirstMomentNaive}) and~(\ref{eqBayes}) that
	\begin{eqnarray*}
	\frac{1}{n}\ln\sum_{\sigma\in [k]^{n}}\sum_{i\geq1}\frac{\pr[\sigma \in \Lambda,A_{p,\sigma},C_{p,\sigma}=i, Z'=0]}{i}
		&\leq&\frac{1}{n}\ln\sum_{\sigma\in [k]^{n}}\pr[\sigma \in \Lambda,A_{p,\sigma},Z'=0]\\
		&\leq&\frac1n\ln\Erw\brk{\Zkc}-\frac{0.94}k+o(1)\leq-\frac{1}{4k},
	\end{eqnarray*}
as claimed.
\qed

\subsection{Proof of \Prop~\ref{Prop_vacant}}\label{Sec_vacant}

Clearly, whether a vertex is $i$-vacant or not only depends on the {\em colors} of its neighbors.
Recall from \Sec~\ref{Sec_Prop_goodFirstMoment} that for a given map $\sigma$ and a probability distribution $\mu$ on $\brk k\times\brk k$ we denote by
$\vec\Gamma_{\sigma,\mu}:V\times\brk d\ra V\times\brk d$ a random configuration that respects $\sigma$ and $\mu$.
Furthermore, because we are only interested in the colors of the neighbors of the vertices, we let
	$\vec\Gamma_{\sigma,\mu}^*:V\times\brk d\ra\brk k$ map each clone $(v,j)$ to the color $i$ such that $\vec\Gamma_{\sigma,\mu}\in V_i\times\brk d$.
Let $V_i=\sigma^{-1}(i)$ for all $i\in\brk k$.

To describe the distribution of the random map $\vec\Gamma_{\sigma,\mu}^*$ in simpler terms,
let $\vec g_{\sigma,\mu}=(g(v,j))_{j\in\brk k,v\in V_j}$ be a family of independent $\brk k$-valued random variables such that
	$$\pr\brk{g(v,j)=i}=\frac{\mu_{ij}}{\rho_i}\qquad\mbox{for }i,j\in\brk k,v\in V_j.$$
Let $\cB_\mu$ be the event that
	$\abs{\cbc{v\in V_j:g(v,j)=i}}=\mu_{ij}dn$ for all $i,j\in\brk k.$
Then we have the following multivariate analogue of \Lem~\ref{lem_balls_bins_cap} (the binomial approximation to the hypergeometric distribution).

\begin{fact}\label{Fact_g}
For any event $\cE$ we have
	$\pr\brk{\vec\Gamma_{\sigma,\mu}^*\in\cE}=\pr\brk{\vec g_{\sigma,\mu}\in\cE|\cB_{\mu}}\leq n^{O(1)}\cdot\pr\brk{\vec g_{\sigma,\mu}\in\cE}.$
\end{fact}
%

Let us call $v\in V$ \bemph{$j$-vacant} in $\vec g_{\sigma,\mu}$ if $\sigma(v)\neq j$ and $g(v,l)\neq j$ for all $l\in\brk d$.
Armed with Fact~\ref{Fact_g}, we can analyze the number of $j$-vacant vertices fairly easily.

\begin{lemma}\label{Lemma_dobulyVacant} 
Let $U^*$ be the number of vertices $v\in V$ such that for two distinct colors $j,j'\in\brk k\setminus\cbc{\sigma(v)}$,
$v$ is both $j$-vacant and $j'$-vacant in $\vec g_{\sigma,\mu}$.
Then 
	$$ \pr\brk{U^*>\frac{n}{k\ln^{3/4}k}}\leq \exp\left[-n\cdot\Omega_k(\ln^{1/9} k/k)\right].$$
\end{lemma}
\begin{proof}
For a vertex $v$ and colors $j,j'\in\brk k\setminus\cbc{\sigma(v)}$, $j\neq j'$ let
	$p_{v,j,j'}=\pr\brk{\mbox{$g(v,l)\not\in\cbc{j,j'}$ for all $l\in\brk d$}}.$
Because the $(g(v,l))_{l\in\brk d}$ are mutually independent, we have
	$$p_{v,j,j'}=\bc{1-\frac{\mu_{ij}+\mu_{ij'}}{\rho_i}}^d.$$
Our assumptions~(\ref{eqrhoNice}) and~(\ref{eqmuNice}) on $\rho$ and $\mu$ ensure that $(\mu_{ij}+\mu_{ij'})/\rho_i\geq1.99/k$.
As, moreover, $d\geq1.99k\ln k$, we obtain
	$$p_{v,j,j'}\leq \bc{1-0.99/k}^{1.99k\ln k}\leq k^{-1.9}.$$
Because this estimate holds for all $v,j,j'$ and since the $(g(v,l))_{v\in V,l\in\brk d}$ are mutually independent,
we conclude that $U^*$ is stochastically dominated by a binomial random variable $\Bin(n,k^{-1.9})$.
Therefore, the Chernoff bound (Lemma \ref{Lemma_Chernoff}) yields
	\begin{eqnarray*}
	\pr\brk{U^*>\frac{n}{k\ln^{3/4}k}}&\leq&\pr\brk{\Bin(n,k^{-1.9})>\frac{n}{k\ln^{3/4}k}}\\
		&\leq&\exp\brk{-\frac{n}{k\ln^{3/4}k}\cdot\ln\bcfr{k^{0.9}}{\eul\ln^{3/4}k}}\leq
			\exp\brk{-n\cdot\Omega_k(\ln^{1/9}k/k)},
	\end{eqnarray*}
as claimed.
\qed\end{proof}

\begin{lemma}\label{Lemma_vacantKL} 
Let $U$ be the number of $v\in V$ that are $j$-vacant in $\vec g_{\sigma,\mu}$ for some color $j\in\brk k\setminus\cbc{\sigma(v)}$.
For any $p\in[0,1]$ there exist $p',q$ satisfying~(\ref{eqqthrm:C^p}) such that
	$$\frac1n\ln\pr\brk{U=(1-p)n}\leq\max\cbc{-\KL{p'}{q'},-\Omega_k(\ln^{1/8}k/k)}+o(1).$$
\end{lemma}
\begin{proof}
Let $\cI$ be the set of pairs $(i,j)$ such that
	$$|\rho_i-1/k|\leq k^{-1}\ln^{-4}k\mbox{ and }|\mu_{ij}-k^{-1}(k-1)^{-1}|\leq k^{-1}(k-1)^{-1}\ln^{-4}k.$$
Our assumptions~(\ref{eqrhoNice}) and~(\ref{eqmuNice}) on $\rho,\mu$ ensure that $\abs\cI\geq k^2-2\ln^9k$.
Let $U_{\cI}$ be the number of vertices $v\in V_i$ that are $j$-vacant for a pair $(i,j)\in\cI$, and let
$U_{\bar\cI}$ be the number $j$-vacant $v\in V_i$ with $(i,j)\not\in\cI$.
Then $U_{\cI}\leq U\leq U_{\cI}+U_{\bar\cI}$.
We are going to estimate $U_{\cI}$, $U_{\bar\cI}$ separately.

Assume that $(i,j)\in\cI$.
The probability that a vertex $v\in V_i$ is $j$-vacant in $\vec g_{\sigma,\mu}$ is
	\begin{eqnarray*}
	p_{ij}&=&(1-\mu_{ij}/{\rho_i})^d=(1-1/k+O_k(k^{-1}\ln^{-4}k))^d=\exp(-2\ln k+O_k(\ln^{-4}k)).
	\end{eqnarray*}
Furthermore, if $j'\not\in\cbc{i,j}$ is another index such that $(i,j')\in\cI$, then the probability that $v\in V_i$ is both $j$-vacant and $j'$-vacant in $\vec g_{\sigma,\mu}$ is
	\begin{eqnarray*}
	p_{ijj'}&=&(1-(\mu_{ij}+\mu_{ij'})/{\rho_i})^d=(1-2/k+O_k(k^{-1}\ln^{-4}k))^d=\exp(-4\ln k+O_k(\ln^{-4}k)).
	\end{eqnarray*}
Hence, by inclusion/exclusion the probability that $v$ is $j$-vacant in $\vec g_{\sigma,\mu}$  for some $j$ with $(i,j)\in\cI$ is
	\begin{eqnarray}\nonumber
	p_{i}&=&\abs{\cbc{j\in\brk k\setminus\cbc i:(i,j)\in\cI}}\cdot\exp(-2\ln k+O_k(\ln^{-4}k))\\
		&=&(k+O_k(\ln^9k))\exp(-2\ln k+O_k(\ln^{-4}k))=k^{-1}(1+O_k(\ln^{-4}k)).
			\label{eqLemStochDom1}
	\end{eqnarray}
Because the events $\cbc{\mbox{$v$ is $j$-vacant in $\vec g_{\sigma,\mu}$}}$ are mutually independent for all $v$ by the definition of $\vec g_{\sigma,\mu}$,
(\ref{eqLemStochDom1}) implies that $U_\cI$ is stochastically dominated by a random variable with distribution $\Bin(n,q^*)$ with parameter $q^*=k^{-1}(1+O_k(\ln^{-4}k)))$.
On the other hand, (\ref{eqLemStochDom1}) also implies that $U_\cI$ stochastically dominates a random variable with 
distribution $\Bin(n,q_*)$ with another $q_*=k^{-1}(1+O_k(\ln^{-4}k)))$.
Hence, for any integer $\nu$ we have
	\begin{eqnarray}\label{eqLemStochDom2}
	\min\cbc{\pr\brk{\Bin(n,q^*)=\nu},\pr\brk{\Bin(n,q_*)=\nu}}
	&\leq&\pr\brk{U_{\cI}=\nu}\\
	&\hspace{-6cm}\leq&\hspace{-3cm}\max\cbc{\pr\brk{\Bin(n,q^*)=\nu},\pr\brk{\Bin(n,q_*)=\nu}}.\nonumber
	\end{eqnarray}

To estimate $U_{\bar\cI}$, we observe that our assumption on $\mu,\rho$ ensures that for any $i\in\brk k$, $j\in\brk k\setminus\cbc i$ and any $v\in V_i$ we have
	\begin{eqnarray}\label{eqLemStochDom3}
	\pr\brk{\mbox{$v$ is $j$-vacant in $\vec g_{\sigma,\mu}$}}&=&(1-\mu_{ij}/\rho_i)^d\leq(1-0.99/k)^{1.99k\ln k}\leq k^{-1.9}.
	\end{eqnarray}
Let $\bar n$ be the number of vertices $v$ that belong to a class $V_i$ such that $(i,j)\in\cI$ for some $j\in\brk k$.
Then $\bar n\leq 1.01n\ln^9 k/k$ because $|\brk k^2\setminus\cI|\leq\ln^9 k$ and because~(\ref{eqrhoNice}) ensures that
 $|V_i|=\rho_in\leq1.01/k$ for all $i$.
Hence, (\ref{eqLemStochDom3}) implies that $U_{\bar\cI}$ is stochastically dominated by
a random variable with distribution $\Bin(\lceil1.01n\ln^9 k/k\rceil,\ln^9k/k^{1.9})$, i.e., for any $\nu\geq0$ we have
	\begin{eqnarray}\label{eqLemStochDom4}
	\pr\brk{U_{\bar\cI}\geq\nu}&\leq&
	\pr\brk{\Bin(\lceil1.01n\ln^9 k/k\rceil,\ln^9k/k^{1.9})\geq\nu}.
	\end{eqnarray}
Consequently, \Lem~\ref{Lemma_Chernoff} (the Chernoff bound) gives
	\begin{eqnarray}\label{eqLemStochDom4}
	\pr\brk{U_{\bar\cI}\geq\frac n{k\ln^{7/8}k}}&\leq&
	\exp\brk{-n\Omega_k(\ln^{1/8}k/k)}.
	\end{eqnarray}

Suppose that $\nu=(1-p)n$ is an integer.
Since $U_{\cI}\leq U\leq U_{\cI}+U_{\bar\cI}$, (\ref{eqLemStochDom4})  yields
	\begin{eqnarray}\label{eqLemStochDom5}
	\pr\brk{U=pn}&\leq&\pr\brk{U_{\cI}=n(p+O_k(k^{-1}\ln^{-7/8}k))}+\exp\brk{-n\Omega_k(\ln^{1/8}k/k)}.
	\end{eqnarray}
Hence, consider a number $p'=p+O_k(k^{-1}\ln^{-7/8}k)$.
Then (\ref{eqLemStochDom2}) shows that there exists $q'=k^{-1}(1+O_k(\ln^{-4}k))$ such that
	\begin{eqnarray}\label{eqLemStochDom6}
	\pr\brk{U_{\cI}=p'n}&=&\pr\brk{\Bin(n,q')=p'n}\;\stacksign{(\ref{eqBinLDP})}{=}\;\exp\brk{-\KL{p'}{q'}n+O(\ln n)}.
	\end{eqnarray}
The assertion follows from~(\ref{eqLemStochDom5}) and~(\ref{eqLemStochDom6}).
\qed\end{proof}

\begin{lemma}\label{Lemma_bigVacantSets} 
Let $U_{ij}$ be the number of vertices $v\in V_i$ that $j$-vacant in $\vec g_{\sigma,\mu}$.
The random variable $$\hat U=\sum_{i,j\in\brk k}\abs{U_{ij}}\cdot\vecone_{\abs{U_{ij}}>n/k^{2.9}}$$ satisfies
	$$\pr\brk{\hat U>\frac{2n}{k\ln^{3/4}k}}\leq \exp\left[-n\Omega_k(\ln^{1/9} k/k)\right].$$
\end{lemma}
\begin{proof}
Let $U_{ij}'$ be the number of vertices $v\in V_i$ that $j$-vacant in $\vec g_{\sigma,\mu}$ but not $j'$-vacant in $\vec g_{\sigma,\mu}$ for any $j\in\brk k\setminus\cbc{i,j}$.
Let $$\hat U'=\sum_{i,j\in\brk k}\abs{U_{ij}'}\cdot\vecone_{\abs{U_{ij}'}>n/k^{2.9}}.$$
Due to \Lem~\ref{Lemma_dobulyVacant} it suffices to prove that
	\begin{equation}\label{eqLemma_bigVacantSets1}
	\pr\brk{\hat U'>\frac{n}{k\ln^{3/4}k}}\leq \exp\left[-n\Omega_k(\ln^{1/9} k/k)\right].
	\end{equation}

To establish~(\ref{eqLemma_bigVacantSets1}) we use a first moment argument.
Let $\cI\subset\brk k^2$ be a set of pairs $(i,j)$ such that $i\neq j$.
Moreover, let $\vec s=(s_{ij})_{i,j\in\cI}$ be a family of non-negative integers such that 
	\begin{equation}\label{eqLemma_bigVacantSets2}
	\mbox{$s_{ij}>n/k^{2.9}$ for all $(i,j)\in\cI$ and }
	\sum_{(i,j)\in\cI}s_{ij}=\left\lceil\frac{n}{k\ln^{3/4}k}\right\rceil.
	\end{equation}
Furthermore, let 
	$\vec S=(S_{ij})_{i,j\in\brk \cI}$ be a family or pairwise disjoint sets such that
	\begin{equation}\label{eqLemma_bigVacantSets3}
	\mbox{$S_{ij}\subset V_i$ and $|S_{ij}|=s_{ij}$ for all $(i,j)\in\cI$.}
	\end{equation}
Let $\cE(\vec S)$ be the event that for all $(i,j)\in\cI$ the vertices $v\in S_{ij}$ are $j$-vacant in $\vec g_{\sigma,\mu}$,
and let $\cE(\vec s)$ be the event that there exists $\vec S$ satisfying~(\ref{eqLemma_bigVacantSets3}) such that $\cE(\vec S)$ occurs.
Clearly, if $\hat U>nk^{-1}\ln^{-3/4}k$, then $\cE(\vec s)$ occurs for some $\cI$ and some $\vec s$ satisfying~(\ref{eqLemma_bigVacantSets2}).
Thus, we need to bound $\pr\brk{\cE(\vec s)}$.

We begin by estimating $\pr\brk{\cE(\vec S)}$.
Consider a vertex $v\in S_{ij}$ for some $(i,j)\in\cI$.
Our assumptions~(\ref{eqrhoNice}) and~(\ref{eqmuNice}) on $\mu$ and $\rho$ ensure that
	$$\pr\brk{\mbox{$v$ is $j$-vacant in $g$}}=(1-\mu_{ij}/\rho_i)^d\leq(1-0.99/k)^{1.99k\ln k}\leq k^{-1.95}.$$
Since these events occur independently for all $v\in S_{ij}$ and because the sets $S_{ij}$ are pairwise disjoint, we obtain
	\begin{eqnarray}\label{eqLemma_bigVacantSets4}
	\pr\brk{\cE(\vec S)}&\leq&\prod_{(i,j)\in\cI}\prod_{v\in S_{ij}}\pr\brk{\mbox{$v$ is $j$-vacant in $\vec g_{\sigma,\mu}$}}\leq k^{-1.95\sum_{(i,j)\in\cI}s_{ij}}.
	\end{eqnarray}
		
To estimate $\pr\brk{\cE(\vec s)}$, we use the union bound.
More precisely, for a given $\vec s$ satisfying~(\ref{eqLemma_bigVacantSets2}) the number of possible $\vec S$ satisfying~(\ref{eqLemma_bigVacantSets3}) is bounded by
	\begin{eqnarray}\nonumber
	\cH&=&\prod_{(i,j)\in\cI}\bink{\rho_i n}{s_{ij}}\leq\bink{2n/k}{s_{ij}}\qquad\mbox{[by our assumption~(\ref{eqrhoNice}) on the $\rho_i$]}\\
		&\leq&\exp\brk{\sum_{(i,j)\in\cI}s_{ij}\ln\bcfr{2\eul n/k}{s_{ij}}}\leq\exp\brk{\sum_{(i,j)\in\cI}s_{ij}\ln\bc{2\eul k^{1.9}}}\quad\mbox{[as $s_{ij}>k^{-2.9}n$]}.
			\label{eqLemma_bigVacantSets5}
	\end{eqnarray}
Combining~(\ref{eqLemma_bigVacantSets4}) and~(\ref{eqLemma_bigVacantSets5}), we obtain
	\begin{eqnarray}\nonumber
	\pr\brk{\cE(\vec s)}&\leq&\cH\cdot k^{-1.95\sum_{(i,j)\in\cI}s_{ij}}\leq\exp\brk{\sum_{(i,j)\in\cI}s_{ij}\ln\bc{2\eul k^{-0.05}}}\\
		&\leq&\exp\bc{-n\Omega_k(\ln^{1/4}k/k)}\qquad\qquad\mbox{[as $\sum_{(i,j)\in\cI}s_{ij}>nk^{-1}\ln^{-3/4}k$]}.\label{eqLemma_bigVacantSets6}
	\end{eqnarray}
Since the total number of sets $\cI$ and vectors $\vec s$ satisfying~(\ref{eqLemma_bigVacantSets2}) is bounded by a polynomial in $n$,
the assertion follows from~(\ref{eqLemma_bigVacantSets6}).
\qed\end{proof}

\medskip\noindent
Finally, \Prop~\ref{Prop_vacant} follows by combining Fact~\ref{Fact_g} with \Lem s~\ref{Lemma_dobulyVacant}, \ref{Lemma_vacantKL} and~\ref{Lemma_bigVacantSets}.


\subsection{Proof of \Prop~\ref{Prop_spanned}}\label{Sec_spanned}


The proof is based on a first moment argument.
Let $V_i=\sigma^{-1}(i)$ for all $i\in\brk k$.
Let $\cI\subset\brk k^2$ be a set of pairs $(i,j)$ such that $i\neq j$.
Moreover, let $\vec s=(s_{ij})_{(i,j)\in\cI}$ be a non-negative integer vector such that
	\begin{equation}\label{eqProp_spanned1}
	\mbox{$0<s_{ij}\leq k^{-2.9}n$ for all $(i,j)\in\cI$ and $0.01\frac nk\leq\sum_{(i,j)\in\cI}s_{ij}\leq100\frac nk.$}
	\end{equation}
Further, let $\vec S=(S_{ij})_{(i,j)\in\cI}$ be a family of pairwise disjoint sets such that
	\begin{equation}\label{eqProp_spanned2}
	\mbox{$S_{ij}\subset V_i$ and $|S_{ij}|=s_{ij}$ for all $(i,j)\in\cI$.}
	\end{equation}
In addition, let $Q$ be a set of edges of the complete graph on $V\times\brk d$ such that the following is true.
	\begin{equation}\label{eqProp_spanned3}
	\parbox[c]{12cm}{We have $|Q|=\lceil nk^{-1}\ln^{-4/5}k\rceil$.
		Moreover, for any edge
		$\cbc{(v,l),(v',l')}\in Q$  there exist indices $i,i',j$ such that $i\neq i'$, $(i,j)\in\cI$, $(i',j)\in\cI$, $v\in S_{ij}$, $v'\in S_{i'j}$.}
	\end{equation}
In words, any edge in $Q$ connects clones of vertices in sets $S_{ij}$, $S_{i'j}$ with $i\neq i'$.

Now, let $\cE(\vec S,Q)$ be the event that the vertices in $S_{ij}$ are $j$-vacant for all $(i,j)\in\cI$
and that the matching $\vec\Gamma_{\sigma,\mu}$ contains $Q$.
Furthermore, let $\cE(\vec S)$ be the event that $\cE(\vec S,Q)$ occurs for some $Q$ satisfying~(\ref{eqProp_spanned3}), 
let $\cE(\vec s)$ be the event that $\cE(\vec S)$ occurs for some $\vec S$ satisfying~(\ref{eqProp_spanned2}), and 
let $\cE$ be the event that $\cE(\vec s)$ occurs for some $s$ that satisfies~(\ref{eqProp_spanned1}).
If $E>n k^{-1}\ln^{-4/5}k$, then the event $\cE$ occurs.
Hence, our task is to prove that
	\begin{equation}\label{eqProp_spanned4}
	\pr\brk\cE\leq\exp(-\Omega_k(\ln^{1/9}k/k)n).
	\end{equation}

To establish~(\ref{eqProp_spanned4}), we are going to work our way from bounding $\pr[\cE(\vec S,Q)]$ to bounding $\pr[\cE]$.
Let us begin with the following simple bound on the probability that the edges $Q$ occur in $\vec\Gamma_{\sigma,\mu}$.

\begin{lemma}\label{Lemma_Q}
Suppose that $\vec s$, $\vec S$ and $Q$ satisfy~(\ref{eqProp_spanned1})--(\ref{eqProp_spanned3}).
Then
	$\pr\brk{Q\subset \vec\Gamma_{\sigma,\mu}}\leq\bcfr{5}{dn}^{|Q|}$
\end{lemma}
\begin{proof}
This follows immediately from \Lem~\ref{lem_edges_cond} and \Rem~\ref{rem_edges_cond}.
\qed\end{proof}

\noindent
Based on \Lem~\ref{Lemma_Q}, we can estimate $\pr[\cE(\vec S,Q)]$.

\begin{lemma}\label{Lemma_ESQ}
Suppose that $\vec s$, $\vec S$ and $Q$ satisfy~(\ref{eqProp_spanned1})--(\ref{eqProp_spanned3}).
Let $s=\sum_{(i,j)\in\cI}s_{ij}$.
Then
	$$\pr\brk{\cE(\vec S,Q)|Q\subset\vec\Gamma_{\sigma,\mu}}\leq k^{-(2+O_k(\ln^{-4}k))s}.$$
\end{lemma}
\begin{proof}
Let $W\subset V\times\brk d$ be the set of all clones that do not occur in any of the edges in $Q$.
Moreover, let $q_{ij}$ be the number of $V_i\times\brk d$-$V_j\times\brk d$ edges in $Q$ and set
	$\mu_{ij}'=\mu_{ij}-\frac{q_{ij}}{dn}.$
In addition, let $\rho_i'=\sum_{j\in\brk k}\mu_{ij}'$.
Furthermore, let $\vec g':W\ra\brk k$ be a random map defined as follows.
	\begin{quote}
	For each pair $(v,l)\in W$ with $v\in V_i$ and any $j\in\brk k\setminus\cbc i$ let $g'(v,l)=j$ with probability $\mu_{ij}'/\rho_i'$,
		independently of all others.
	\end{quote}
Then in analogy to Fact~\ref{Fact_g}, we have
	\begin{equation}\label{eqLemma_ESQ1}
	\pr\brk{\vec\Gamma_{\sigma,\mu}^*\in\cA}\leq n^{O(1)}\pr\brk{\vec g'\in\cA}\qquad\mbox{for any event $\cA$.}
	\end{equation}

Since (\ref{eqProp_spanned3}) provides that $|Q|/n\sim k^{-1}\ln^{-4/5}k$, we see that
	\begin{equation}\label{eqLemma_ESQ2}
	\norm{\rho-\rho'}_1\leq\norm{\mu-\mu'}_1\leq O_k(k^{-2}\ln^{-9/5}k). 
	\end{equation}
Now, let $\cI'$ be the set of all $(i,j)\in\cI$ such that
	$$|\mu_{ij}-k^{-1}(k-1)^{-1}|\leq\frac2{k(k-1)\ln^4k}\quad\mbox{and}\quad|\rho_i'-k^{-1}|\leq\frac2{k\ln^4k}.$$
Then~(\ref{eqLemma_ESQ2}) implies together with our assumption on $\rho,\mu$ that 
	\begin{equation}\label{eqLemma_ESQ3}
	|\cI\setminus\cI'|\leq\ln^{12}k.
	\end{equation}
Furthermore, for $(i,j)\in\cI'$ we let
	$$S_{ij}'=\cbc{v\in S_{ij}:|(\cbc v\times\brk d)\cap W|\geq d-k^{7/8}}.$$
In other words, $S_{ij}'$ contains all $v\in S_{ij}$ that occur in no more than $k^{7/8}$ edges in $Q$.

The bound~(\ref{eqLemma_ESQ1}) implies together with the construction of $\vec g'$ that
	\begin{eqnarray}\nonumber
	\pr\brk{\cE(\vec S,Q)|Q\subset\vec\Gamma_{\sigma,\mu}}&\leq&n^{O(1)}\cdot\pr\brk{\forall (i,j)\in\cI,v\in S_{ij}:\mbox{$v$ is $j$-vacant in $\vec g'$}}\\
		&\leq&n^{O(1)}\cdot\pr\brk{\forall (i,j)\in\cI',v\in S_{ij}':\mbox{$v$ is $j$-vacant in $\vec g'$}}\nonumber\\
		&=&n^{O(1)}\prod_{(i,j)\in\cI'}\prod_{v\in S_{ij}'}\pr\brk{\mbox{$v$ is $j$-vacant in $\vec g'$}}.
			\label{eqLemma_ESQ4}
	\end{eqnarray}	
Further, because for any $v\in S_{ij}'$ the values $(g(v,l))_{l:(v,l)\in W}$ are independent, we have
	\begin{eqnarray}
	\pr\brk{\mbox{$v$ is $j$-vacant in $\vec g'$}}&=&(1-\mu_{ij}'/\rho_i')^{|(\cbc v\times d)\cap W|}\leq
						(1-\mu_{ij}'/\rho_i')^{d-k^{7/8}}\qquad\mbox{[as $v\in S_{ij}'$]}\nonumber\\
				&\leq&(1-k^{-1}(1+O_k(\ln^{-4}k)))^{d-k^{7/8}}\qquad\qquad\mbox{[because $(i,j)\in\cI'$]}\nonumber\\
				&\leq&k^{-2+O_k(\ln^{-4}k)}.			\label{eqLemma_ESQ5}
	\end{eqnarray}
To complete the proof, let $s'=\sum_{(i,j)\in\cI'}|S_{ij}'|$.
Because $|Q|/n\sim k^{-1}\ln^{-4/5}k$ by~(\ref{eqProp_spanned3}), we have
	$$\sum_{(i,j)\in\cI'}|S_{ij}\setminus S_{ij}'|\leq \frac12k^{-15/8}n.$$
Furthermore, as $|S_{ij}|\leq k^{-2.9}n$ for all $(i,j)\in\cI$, we have
	$$\sum_{(i,j)\in\cI\setminus\cI'}|S_{ij}|\leq|\cI\setminus\cI'|k^{-2.9}n\leq k^{-2.8}n\qquad\qquad\mbox{[due to~(\ref{eqLemma_ESQ3})]}.$$
Combining these two bounds, we see that
	$s'\geq s-k^{-15/8}n.$
Thus, (\ref{eqLemma_ESQ4}) and~(\ref{eqLemma_ESQ5}) yield
	\begin{eqnarray*}
	\pr\brk{\cE(\vec S,Q)|Q\subset\vec\Gamma_{\sigma,\mu}}&\leq&k^{-(2+O_k(\ln^{-4}k))s'}\leq k^{-(2+O_k(\ln^{-4}k))s},
	\end{eqnarray*}
as desired.
\qed\end{proof}

\begin{corollary}\label{Cor_ESQ}
Suppose that $\vec s$ and $\vec S$ satisfy~(\ref{eqProp_spanned1}) and (\ref{eqProp_spanned2}).
Let $s=\sum_{(i,j)\in\cI}s_{ij}$.
Then
	$$\pr\brk{\cE(\vec S)}\leq\exp\brk{-2s\ln k-\Omega_k(\ln^{1/9}k/k)n}.$$
\end{corollary}
\begin{proof}
Given $\vec s$ and $\vec S$, let $\cH=\cH(\vec s,\vec S)$ be the number of set $Q$ that satisfy~(\ref{eqProp_spanned3}).
Any such set $Q$ decomposes into sets $Q_j$ of edges joining two clones in $\bigcup_{i:(i,j)\in\cI}S_{ij}$.
Since $|S_{ij}|\leq k^{-2.9}n$ for all $i,j$, we have $|\bigcup_{i:(i,j)\in\cI}S_{ij}|\leq k^{-1.9}n$ for all $j$.
Let $\eta=|Q|=\lceil nk^{-1}\ln^{-4/5}k\rceil$.
Because the uniform distribution maximizes the entropy, we get
	\begin{eqnarray}\label{eqCor_ESQ1}
	\cH&\leq& \exp(o(n))\cdot\bink{\bink{dn/k^{1.9}}2}{\eta/k}^k
		=\exp\brk{(1+o_k(1))n\cdot\eta\ln\frac{d^2}{k^{2.8}\eta}}.
	\end{eqnarray}
Hence, \Lem s~\ref{Lemma_Q} and~\ref{Lemma_ESQ} and the union bound yield
	\begin{eqnarray}\nonumber
	\pr\brk{\cE(\vec S)}&\leq&\sum_Q\pr\brk{\cE(\vec S,Q)}=\sum_Q\pr\brk{\cE(\vec S,Q)|Q\subset\vec\Gamma_{\sigma,\mu}}\cdot\pr\brk{Q\subset\vec\Gamma_{\sigma,\mu}}\\
		&\leq&\exp\brk{-2s(\ln k+O_k(\ln^{-4}k))}\cdot
			\sum_Q\pr\brk{Q\subset\vec\Gamma_{\sigma,\mu}}\nonumber\\
		&\leq&\exp\brk{-2s(\ln k+O_k(\ln^{-4}k))}\cdot\cH\cdot\bcfr{5}{dn}^\eta\nonumber\\
		&\leq&\exp\brk{-2s\ln k+n\bc{O_k(k^{-1})+\eta \ln\frac{d}{k^{2.8}\eta}}}\qquad\mbox{[due to~(\ref{eqCor_ESQ1})]}.
			\label{eqCor_ESQ2}
	\end{eqnarray}
Finally, our assumptions on $d$ and $\eta$ ensure that $d/\bc{k^{2.8}\eta}\leq k^{-0.7}$.
Consequently, $$\eta\ln\frac d{k^{2.8}\eta}\leq-\Omega_k(\ln k/k),$$ and thus the assertion follows from~(\ref{eqCor_ESQ2}).
\qed\end{proof}

\begin{corollary}\label{Cor_Es}
Suppose that $\vec s$ satisfies~(\ref{eqProp_spanned1}).
Then
	$\pr\brk{\cE(\vec s)}\leq\exp\brk{-\Omega_k(\ln^{1/9}k/k)n}.$
\end{corollary}
\begin{proof}
For a given $\vec s$ let $\cH=\cH(\vec s)$ be the number of $\vec S$ satisfying~(\ref{eqProp_spanned2}).
Let $s=\sum_{(i,j)\in\cI}s_{ij}$. 
Because the uniform distribution maximizes the entropy, we have
	\begin{eqnarray}\label{eqCor_Es1}
	\cH&\leq&\bink ns k^s\leq\exp\brk{s\bc{1+\ln\frac{kn}{s}}}=\exp\brk{2s\ln k+O_k(k^{-1})n};
	\end{eqnarray}
the last inequality follows because~(\ref{eqProp_spanned1}) provides that $s=\Theta_k(k^{-1})n$.
The assertion follows from~(\ref{eqCor_Es1}), \Cor~\ref{Cor_ESQ} and the union bound.
\qed\end{proof}

Finally, as there is only a polynomial number $n^{O(1)}$ of vectors $\vec s$ that satisfy~(\ref{eqProp_spanned1}),
\Cor~\ref{Cor_Es} implies~(\ref{eqProp_spanned4}), whence the proof of \Prop~\ref{Prop_spanned} is complete.

\end{document}